\documentclass[a4paper,leqno]{amsart}
\usepackage[margin=30truemm]{geometry}
\usepackage{fancyhdr,appendix,graphicx,layout}
\usepackage{amsmath,amsthm,amssymb}
\usepackage[shortlabels]{enumitem}
\setlist[enumerate]{font=\normalfont}
\usepackage{physics}

\usepackage[numbers]{natbib}
\usepackage{bbm}
\usepackage{bm}
\usepackage{tikz}
\usepackage{tikz-cd}
\usetikzlibrary{positioning,arrows.meta,decorations.markings,calc}
\tikzset{
    cells={font=\everymath\expandafter{\the\everymath\displaystyle}},
}
\usepackage{caption}
\usepackage[pagewise]{lineno}

\usepackage{float}

\usepackage{color}
\definecolor{darkgreen}{rgb}{0,0.7,0}

\definecolor{darkblue}{rgb}{0,0,0.7} % darkblue color
\newcommand{\defn}[1]{\textsl{\color{darkblue} #1}}

\usepackage[hyperpageref]{backref}
\usepackage[pagebackref]{hyperref}
\hypersetup{
    colorlinks=true,
    linkcolor=red,  %section,table of contents
    filecolor=blue,  %local file link
    urlcolor=orange, %url
    citecolor=violet   %citation
}
\renewcommand*{\backref}[1]{}  
   \renewcommand*{\backrefalt}[4]{
      \ifcase #1 
         Not cited.
      \or
         Cited on page #2.
      \else
         Cited on pages #2.
      \fi}

\usepackage{cleveref}

\makeatletter
\newcommand\ReDeclareMathOperator[2]{%
    \begingroup \escapechar\m@ne\xdef\@gtempa{{\string#1}}\endgroup
    \expandafter\@ifundefined\@gtempa
        {\@latex@error{Command \string#1 undefined}\@ehc}
        \relax
    \let\@ifdefinable\@rc@ifdefinable
    \DeclareMathOperator#1{#2}}
\makeatother

\crefname{thm}{Theorem}{Theorems}
\crefname{dfn}{Definition}{Definitions}
\crefname{dfnprop}{Definition-Proposition}{Definition-Proposition}
\crefname{prop}{Proposition}{Propositions}
\crefname{lem}{Lemma}{Lemmas}
\crefname{cor}{Corollary}{Corollaries}
\crefname{clm}{Claim}{Claims}
\crefname{ass}{Assumption}{Assumption}
\crefname{cond}{Condition}{Condition}
\crefname{nota}{Notation}{Notation}
\crefname{conj}{Conjecture}{Conjecture}
\crefname{fct}{Fact}{Facts}
\crefname{rmk}{Remark}{Remarks}
\crefname{eg}{Example}{Examples}
\crefname{que}{Question}{Questions}
\crefname{figure}{Figure}{Figures}
\crefname{table}{Table}{Tables}
\crefname{section}{Section}{Sections}
\crefname{subsection}{Subsection}{Subsections}
\crefname{appendix}{Appendix}{Appendices}
\crefname{equation}{}{}

\theoremstyle{plain}
\newtheorem{thm}{Theorem}[section]
\newtheorem{prop}[thm]{Proposition}
\newtheorem{lem}[thm]{Lemma}
\newtheorem{cor}[thm]{Corollary}

\theoremstyle{definition}
\newtheorem{dfn}[thm]{Definition}

\newtheorem{cond}[thm]{Condition}

\newtheorem{conj}[thm]{Conjecture}

\newtheorem{eg}[thm]{Example}

\theoremstyle{remark}

\numberwithin{equation}{section}
\makeatletter
\let\c@equation\c@thm
\makeatother

\newcommand{\A}{\mathcal A}
\newcommand{\D}{\mathcal D}
\newcommand{\K}{\mathcal K}

\newcommand{\T}{\mathcal T}
\newcommand{\U}{\mathcal U}

\newcommand{\RR}{\mathbf{R}}

\newcommand{\ZZ}{\mathbb{Z}}

\DeclareMathOperator{\Add}{Add}

\ReDeclareMathOperator{\top}{top}

\DeclareMathOperator{\id}{id}

\DeclareMathOperator{\proj}{proj}

\let\mod\relax
\DeclareMathOperator{\mod}{mod}

\DeclareMathOperator{\thick}{thick}

\DeclareMathOperator{\per}{per}

\DeclareMathOperator{\Hom}{Hom}
\DeclareMathOperator{\End}{End}

\DeclareMathOperator{\RHom}{\RR Hom}
\DeclareMathOperator{\REnd}{\RR End}

\ReDeclareMathOperator{\l}{\ell}

\newcommand{\vv}{\mathsf{v}}
\newcommand{\ww}{\mathsf{w}}
\newcommand{\xx}{\mathsf{x}}

\DeclareMathOperator{\weight}{wt}

\usepackage[T1]{fontenc}
\usepackage{microtype,needspace}
\AtBeginEnvironment{prop}{\crefalias{thm}{prop}}
\AtBeginEnvironment{lem}{\crefalias{thm}{lem}}
\AtBeginEnvironment{cor}{\crefalias{thm}{cor}}
\AtBeginEnvironment{dfn}{\crefalias{thm}{dfn}}
\AtBeginEnvironment{rmk}{\crefalias{thm}{rmk}}
\AtBeginEnvironment{eg}{\crefalias{thm}{eg}}
\definecolor{Blue}{rgb}{0,0,0.7}

\newcommand{\B}{\mathcal B}
\newcommand{\C}{\mathcal C}
\newcommand{\F}{\mathcal F}

\newcommand{\m}{\mathbbm{m}}
\newcommand{\dG}{\mathbbm{d}}
\DeclareMathOperator{\Ob}{Ob}
\DeclareMathOperator{\Tw}{Tw}

\newcommand{\red}{\mathrm{red}}

\newcommand{\rdeg}[1]{\lVert #1\rVert}
\allowdisplaybreaks[2]
\hypersetup{pdftitle={Lifting A-infinity-structures through coverings with applications to AFBGAs},pdfauthor={Bohan Xing}}

\title[Lifting $A_\infty$-structures through coverings]
{Lifting $A_\infty$-structures through coverings with applications to fractional Brauer graph algebras}
\author{Bohan Xing}
\date{\today}

\newcommand{\Addresses}{{%
  \bigskip
  \footnotesize

  B. Xing, \textsc{School of Mathematical Sciences, Laboratory of Mathematics and Complex Systems, Beijing Normal University, Beijing 100875, P.R. China}\par\nopagebreak
  \textit{E-mail address}: \texttt{bhxing@mail.bnu.edu.cn}
}}

\begin{document}
\begin{abstract}
Under suitable group actions, it is known that $A_\infty$-structures descend to orbit categories. We study the converse problem of lifting an $A_\infty$-structure to a prescribed covering. We show that a compatible group grading determines a canonical lift, uniquely characterized by the strictness of the covering projection. As an application, we construct $A_\infty$-categories associated with admissible fractional Brauer graph algebras by lifting Brauer graph $A_\infty$-categories. Their geometric data augment the surface models of Brauer graph algebras with a Nakayama character encoding the covering. We prove that admissible fractional Brauer graph algebras with equivalent geometric data are derived equivalent. Furthermore, we establish a set of combinatorial derived invariants and prove their completeness in reduced genus zero and in reduced genus at least two when the defining ribbon graph is non-bipartite.
\end{abstract}
\maketitle
\enlargethispage{4pt}
\begingroup
\small
\tableofcontents
\endgroup
%\clearpage

\section{Introduction}

Covering theory in the representation theory of algebras was developed in the early 1980s through the foundational work of Riedtmann \cite{Riedtmann1980a}, Gabriel \cite{Ga1981}, Bongartz and Gabriel \cite{BG1982}, Green \cite{Gre83a}, and Mart{\'i}nez-Villa and de la Pe{\~n}a \cite{MP1983}. For a recent exposition, see the lecture note \cite{LLXC25}. Covering functors provide pull-up and push-down functors between module categories, allowing one to compare the representation theory of the categories involved. 
Two basic constructions arise from group actions and group gradings. An action of a group $G$ on a small linear category $\C$, free on objects, gives a canonical Galois covering $\C\to\C/G$, whereas a $G$-grading on $\C$ gives rise to the smash product covering $\C\#G\to\C$ \cite{CM06}. These constructions lead naturally to questions about the descent of additional structures to orbit categories and their lifting to covering categories. In the homological setting, Keller used dg orbit categories and triangulated hulls to construct triangulated structures on orbit categories under suitable hypotheses \cite{Kel05orbit}. A prominent example is the construction of cluster categories by Buan, Marsh, Reineke, Reiten and Todorov \cite{BMRRT06}. Asashiba extended covering techniques to homotopy categories and derived equivalences \cite{Asa97,Asa2011}. 

At the $A_\infty$ level, orbit constructions have been used in the derived equivalence classification of Brauer graph algebras \cite{OZ22} and in the study of partially wrapped Fukaya categories of orbifold surfaces \cite{BSW24,AP26}. These constructions pass from categories equipped with group actions to orbit or skew-group categories. Our focus is on the complementary problem of lifting a given $A_\infty$-structure on a base category to prescribed covering data. More precisely, we show that a group grading on the morphism spaces determines a canonical lift whenever all structure maps respect the group grading and the strict units have identity weight. The lifted structure is unique once the canonical projection is required to be strict and unit-preserving.

\begin{thm}\textnormal{(see \cref{prop:lift,prop:lift-orbit})}
\label{thm:intro-lifting}
Let $G$ be a group with identity element $e$, and let $\A$ be a small strictly unital $A_\infty$-category over the field $\Bbbk$ with higher multiplications $\mu_\A^n$ $(n\geq1)$. Write $\Ob\A$ for its object set, $\Hom_\A(X,Y)$ for its graded morphism spaces, and $1_X$ for the strict unit at $X$.
Suppose that, for every $X,Y\in\Ob\A$, there is a decomposition
\[
\Hom_\A(X,Y)=\bigoplus_{g\in G}\Hom_\A^g(X,Y),
\qquad
1_X\in\Hom_\A^e(X,X),
\]
where each $\Hom_\A^g(X,Y)$ is a graded subspace and the superscript $g$ denotes its group weight. Assume that, for every $n\geq1$, objects $X_0,\ldots,X_n\in\Ob\A$, group elements $g_1,\ldots,g_n\in G$, and morphisms $a_i\in\Hom_\A^{g_i}(X_{i-1},X_i)$ for $1\leq i\leq n$, one has
\[
\mu_\A^n(a_n,\ldots,a_1)
\in\Hom_\A^{g_1\cdots g_n}(X_0,X_n).
\]

Then there is a strictly unital $A_\infty$-category $\widetilde\A$ with
\[
\Ob\widetilde\A=\Ob\A\times G,
\qquad
\Hom_{\widetilde\A}((X,h),(Y,l))=\Hom_\A^{h^{-1}l}(X,Y),
\]
whose higher multiplications $\mu_{\widetilde\A}^n$ are the restrictions of $\mu_\A^n$. This structure is unique subject to the requirement that the canonical projection $\pi:\widetilde\A\to\A$, given by $(X,h)\mapsto X$ and the inclusions on morphism spaces, be a strict, unit-preserving $A_\infty$-functor.
\end{thm}

This construction is analogous to Sheridan's pullback construction for $A_\infty$-categories with abelian grading data \cite[Section~2.3]{She15}. In the case of graded associative algebras, it is also closely related to the constructions of associated ordinary algebras in \cite{Fus26,FX26}. 

In connection with the orbit constructions in \cite{OZ22,BSW24,AP26}, consider the canonical projection $\A\to\A/G$ associated with a group action free on objects. Let $\mu_\A=(\mu_{\A,n})_{n\geq1}$ be a family of $G$-equivariant higher operations and let $\mu_{\A/G}$ be the corresponding operations defined by the orbit formula \eqref{eq:orbit}, with compatible strict units. These two families determine one another, and the following are equivalent by Corollary~\ref{prop:reflection}.
\begin{enumerate}
\item $\A$ is an $A_\infty$-category with higher operations $\mu_\A$.
\item $\A/G$ is an $A_\infty$-category with higher operations $\mu_{\A/G}$.
\end{enumerate}
Furthermore, using $\mathbb Z$-gradings and their reductions modulo $r$, we extend the trivial extension construction for $A_\infty$-categories in \cite[Definition~2.10]{OZ22} to repetitive categories and $r$-fold trivial extensions for every $r\geq1$; see Subsection~\ref{sec:trivial-lift-comparison}.

We next apply this framework to the study of derived equivalences of finite-dimensional algebras.
Note that geometric models and the $A_\infty$-structures attached to them have become important tools in the representation theory of finite-dimensional algebras, such as gentle algebras, which form a classical class of quadratic monomial special biserial algebras, originating in the study of tilted and iterated tilted algebras of types $A$ and $\widetilde A$ \cite{Ass82,AS87}. Gentle algebras also arise from triangulations of unpunctured marked surfaces \cite[Theorem 2.7]{ABCP10}. A major development was the connection between graded gentle algebras and partially wrapped Fukaya categories of surfaces established by Haiden, Katzarkov and Kontsevich \cite{HKK17} and further developed by Lekili and Polishchuk \cite{LP20}. 

This geometric viewpoint makes a wide range of representation-theoretic structures and invariants accessible through curves, dissections and line fields. Geometric models describe module categories of gentle, skew-gentle and string algebras \cite{BCS21,HZZ23,BCS24}, as well as the derived categories of gentle algebras \cite{OPS18} and objects in the derived categories of skew-gentle algebras \cite{LSV22,Ami23}. They also give descriptions of torsion classes \cite{CD20}, silting objects and their mutations and reductions \cite{APS23,CS23silting,JSW23}, simple-minded collections and algebraic hearts \cite{Cha26hearts}, and tilting-completion \cite{Cha24tilting}, together with criteria for $\tau$-tilting finiteness and silting-discreteness \cite{CJSW25}. Further applications include geometric descriptions of Hochschild cohomology and its algebraic structures \cite{CSSS26,BSSWW26,Opp26b}, and of deformations of the associated Fukaya categories \cite{BSW25deform}. Of particular relevance here, the closure of gentle algebras under derived equivalence, established by Schr\"oer and Zimmermann \cite{SZ03}, admits a geometric proof \cite{APS23}; geometric methods also provide complete derived invariants for gentle algebras \cite{APS23,Opp19} and their graded counterparts \cite{JSW23,Opp25}. Orbifold models also provide geometric criteria for derived equivalences between skew-gentle algebras \cite{AB22} and constructions of algebras derived equivalent to them \cite{BSW24,AP26}, with further developments in preparation \cite{BCKRSWprep}.

Closely related to gentle algebras are Brauer graph algebras, a classical class of symmetric algebras originating in the modular representation theory of finite groups \cite{Dad66,DF78,Don79}. They generalize Brauer tree algebras and, over an algebraically closed field, coincide up to Morita equivalence with symmetric special biserial algebras \cite{Sch15}. Green, Schroll and Snashall developed a covering theory for Brauer graphs and their algebras \cite{GSS14}. The connection with gentle algebras through trivial extensions in the multiplicity-one case \cite{Sch15}, together with covering and orbit constructions, leads to the Brauer graph $A_\infty$-categories of Opper and Zvonareva \cite{OZ22}. Combinatorial and geometric methods have yielded derived and stable equivalence invariants \cite{Ant07st,Ant07,Ant09,AZ19st,AZ22}, as well as closure and classification results for derived equivalence \cite{AZ22,OZ22} and stable equivalence of Morita type \cite{CLLX26,LCLX26}. Further applications concern tilting mutations \cite{Kau98,Sot24a,Sot24b}, tilting-discreteness \cite{AAC18}, torsion classes \cite{CD20}, simple-minded systems in the domestic case \cite{Zha24sms,Zha25sms,Zha26sms}, and Hochschild cohomology \cite{CSS,LX,LWX}. For surveys of Brauer graph algebras and their derived equivalences, see \cite{Sch18,Zvo25}.

Recently, Li and Liu generalized Brauer graph algebras within their theory of fractional Brauer configuration algebras \cite{LLI26,LLII24,LL26}, extending the Brauer configuration algebras of Green and Schroll \cite{GS16,GS17}. The admissible fractional Brauer graph algebras form a class of self-injective special biserial algebras; see \cite[Introduction]{Xin26}. Their relation to Brauer graph algebras is analogous to the passage from symmetric to self-injective Nakayama algebras. Moreover, the repetitive construction for gentle algebras \cite{Rin97} yields $r$-fold trivial extensions, for positive integers $r$, which belong to this class \cite{Xin26}. This connection motivates us to extend the geometric models and $A_\infty$-structures discussed above to admissible fractional Brauer graph algebras.

Every admissible fractional Brauer graph algebra $A$ has a Brauer graph algebra $A_{\mathrm{red}}$ as its reduced form, obtained by taking the orbit under the Nakayama action \cite{LL26,Xin26}. The resulting cyclic covering of the associated linear categories allows us to lift the Brauer graph $A_\infty$-categories of Opper and Zvonareva \cite{OZ22} by applying our lifting theorem. The construction is governed by geometric data $\mathfrak D=(\Sigma,\mathcal P,\eta,\m,r,\chi)$, where $(\Sigma,\mathcal P,\eta,\m)$ records the punctured ribbon surface, line field and multiplicities associated with $A_{\mathrm{red}}$, while $r$ is the order of the Nakayama action, equivalently the number of sheets of the covering, and $\chi:H_1(\Sigma\setminus\mathcal P;\ZZ)\to\ZZ/r\ZZ$ is the Nakayama character recording the monodromy of the corresponding surface covering, with the Nakayama deck transformation identified with $\bar1$; see Definition~\ref{def:datum}. For a graded admissible arc system $\U$, denote the resulting $A_\infty$-category by $\F_{\mathfrak D}(\U)$. As in the gentle and Brauer graph settings, we prove that its Morita equivalence class depends only on the geometric datum, yielding derived equivalences between the corresponding admissible fractional Brauer graph algebras.

\begin{thm}\textnormal{(see Theorem~\ref{thm:geometric} and Corollary~\ref{cor:ordinary})}\label{thm:intro-geometric}
For $i=1,2$, let $\mathfrak D_i=(\Sigma_i,\mathcal P_i,\eta_i,\m_i,r,\chi_i)$ be geometric data, and let $\U_i$ be graded admissible arc systems on $(\Sigma_i,\mathcal P_i,\eta_i)$. Suppose there is an orientation-preserving diffeomorphism $f:(\Sigma_1,\mathcal P_1)\longrightarrow(\Sigma_2,\mathcal P_2)$ satisfying
\[
f^*\eta_2\simeq\eta_1,\qquad
\m_2\circ f|_{\mathcal P_1}=\m_1,\qquad
\chi_2\circ f_*=\chi_1,
\]
where $\simeq$ denotes homotopy of line fields and $f_*$ is the induced map on first integral homology of the punctured surfaces. Then $\F_{\mathfrak D_1}(\U_1)$ and $\F_{\mathfrak D_2}(\U_2)$ are Morita equivalent. If both lifted categories are concentrated in degree zero, their associated algebras $A_1$ and $A_2$ are derived equivalent, that is,
\[
\D^b(\mod\text{-}A_1)\simeq\D^b(\mod\text{-}A_2).
\]
In particular, admissible fractional Brauer graph algebras with equivalent canonical geometric data are derived equivalent.
\end{thm}

Note that the coverings considered above arise from Nakayama actions. For a finite-dimensional self-injective algebra, every basic tilting complex is invariant up to isomorphism under the Nakayama functor \cite[Theorem A.4]{Aih13}; see also \cite[Theorem 2.1]{AR13}. This compatibility suggests that our lifted geometric models may capture all derived equivalences between ordinary admissible fractional Brauer graph algebras. Motivated by the classification of Brauer graph algebras \cite{OZ22}, we therefore seek a complete set of derived invariants. Combining the invariants established in \cite{Xin26} with the description of exceptional tubes in the stable Auslander--Reiten quiver in \cite[Section 5.1]{LL26}, we obtain the following result.

\begin{prop}\textnormal{(see Proposition~\ref{prop:four-derived-invariants})}\label{prop:intro-four-derived-invariants}
Let $\Lambda_i=\Lambda(\Gamma_i,\dG_i)$, $i=1,2$, be finite-dimensional admissible fractional Brauer graph algebras, with connected defining ribbon graphs. Write $\nu_i$ for the Nakayama action on $\Gamma_i$ and $\Gamma_{i,\mathrm{red}}=\Gamma_i/\langle\nu_i\rangle$ for its reduced Brauer graph. Assume that neither reduced graph is a single loop of multiplicity one or a single edge joining two vertices both of multiplicity two. If $\Lambda_1$ and $\Lambda_2$ are derived equivalent, then the following conditions hold.
\begin{enumerate}[label=\textup{(\arabic*)}]
\item The Nakayama actions have the same order $r$, and the reduced graphs have the same numbers of vertices, edges and faces.
\item The multisets of fractional multiplicities $\dG_i(v)/\operatorname{val}_{\Gamma_i}(v)$, indexed by $v\in V(\Gamma_i)$, agree.
\item The multisets of pairs
\[
\left\{\!\left\{(\ell(F),\weight(F)):F\in F(\Gamma_{i,\mathrm{red}})\right\}\!\right\},\qquad i=1,2,
\]
agree. Here $F(\Gamma_{i,\mathrm{red}})$ denotes the set of faces, $\ell(F)$ is the perimeter of $F$, and $\weight(F)\in\ZZ/r\ZZ$ is the unique residue such that the lifted (Green) walk associated with $F$ in $\Gamma_i$, starting at an edge $e$, reaches $\nu_i^{\weight(F)}e$ after $\ell(F)$ steps. The cyclic groups are identified by $\nu_1\mapsto\nu_2$.
\item Either both $\Gamma_1$ and $\Gamma_2$ are bipartite or neither is.
\end{enumerate}
\end{prop}

We conjecture that these four conditions form a complete derived invariant, extending the classification of Brauer graph algebras in \cite{OZ22}.

\begin{conj}\textnormal{(see Conjecture~\ref{conj:four-condition-all-genus})}\label{conj:intro-four-derived-invariants}
Let $\Lambda_1$ and $\Lambda_2$ satisfy the hypotheses on the algebras and reduced graphs in Proposition~\ref{prop:intro-four-derived-invariants}. Then $\Lambda_1$ and $\Lambda_2$ are derived equivalent if and only if conditions \textup{(1)}--\textup{(4)} of that proposition hold.
\end{conj}

Recall that the genus of the ribbon surface associated with a ribbon graph $\Gamma$ is given by
\[
g(\Gamma)=1-\frac12\bigl(|V(\Gamma)|-|E(\Gamma)|+|F(\Gamma)|\bigr),
\]
where $V(\Gamma)$, $E(\Gamma)$ and $F(\Gamma)$ are its sets of vertices, edges and faces, respectively. For an admissible fractional Brauer graph algebra $\Lambda$, we write $g_{\mathrm{red}}(\Lambda)=g(\Gamma_{\mathrm{red}})$ for its reduced genus. We prove the above conjecture in reduced genus zero and in the non-bipartite case of reduced genus at least two.

\begin{thm}\textnormal{(see Theorem~\ref{thm:four-condition-completeness})}\label{thm:intro-four-condition-completeness}
Let $\Lambda_i=\Lambda(\Gamma_i,\dG_i)$, $i=1,2$, satisfy the hypotheses on the algebras and reduced graphs in Proposition~\ref{prop:intro-four-derived-invariants}. Suppose that, for each $i$, either $g_{\mathrm{red}}(\Lambda_i)=0$, or $g_{\mathrm{red}}(\Lambda_i)\geq2$ and $\Gamma_i$ is non-bipartite. Then $\Lambda_1$ and $\Lambda_2$ are derived equivalent if and only if conditions \textup{(1)}--\textup{(4)} of Proposition~\ref{prop:intro-four-derived-invariants} hold.
\end{thm}

\medskip
\noindent
\textbf{Organization of the paper.}
In Section~\ref{sec:ainfinity-preliminaries}, we recall the necessary background on $A_\infty$-categories and derived constructions. In Section~\ref{sec:lifting}, we establish the lifting theorem and apply it to repetitive categories and $r$-fold trivial extensions. In Section~\ref{sec:algebra-surface-preliminaries}, we review Brauer graph algebras and their associated $A_\infty$-categories. In Section~\ref{sec:nakayama}, we construct geometric models and lifted $A_\infty$-structures for admissible fractional Brauer graph algebras. In Section~\ref{sec:derived}, we study elementary moves and establish geometric criteria for derived equivalence. In Section~\ref{sec:numerical}, we introduce combinatorial derived invariants and prove their completeness in special cases.

\medskip
\noindent
{\bf Conventions and notation.}
Let $\Bbbk$ be a field. From Section~\ref{sec:algebra-surface-preliminaries} onwards, we assume that $\Bbbk$ is algebraically closed. All categories are $\Bbbk$-linear, and the $A_\infty$-categories used as coefficient categories are small. All tensor products and linear duals are over $\Bbbk$, unless another base is indicated.  All modules are right modules unless a left-module category is explicitly displayed. We write $\mod\text{-}A$ and $\proj\text{-}A$ for finite-dimensional modules and finitely generated projective modules, respectively, and $A\text{-}\mod$ for finite-dimensional left modules. All paths are read from right to left.

All gradings are cohomological unless specified otherwise. For $V=\bigoplus_{j\in\ZZ}V^j$ and $a\in V^j$, put $|a|=j$ and $\rdeg{a}=|a|-1$. A map of degree $d$ sends $V^j$ to $W^{j+d}$. Complexes are cochain complexes, with differential $d^j:V^j\to V^{j+1}$ and cohomology $H^j(V)$. Our shift convention is $V[s]^j=V^{j+s}$ and $d_{V[s]}=(-1)^s d_V$. The graded dual is denoted by $D$, with $(DV)^j=\Hom_\Bbbk(V^{-j},\Bbbk)$.

\medskip
\noindent
\textbf{Use of AI.}
The material in Sections~1--6 was developed primarily by the author, with assistance from ChatGPT in preparing figures and identifying relevant literature in geometry. Conjecture~\ref{conj:four-condition-all-genus} in Section~\ref{sec:numerical} was formulated by the author. GPT-6 Astra suggested that the conjecture could be proved in certain special cases, leading to Theorem~\ref{thm:four-condition-completeness}, and also proposed the equivalent formulation given in Proposition~\ref{prop:four-condition-reduction}. All mathematical arguments were independently verified and finalized by the author. The author takes full responsibility for the accuracy, originality, and integrity of the paper.

\medskip
\noindent
\textbf{Acknowledgements.}
The author would like to express his sincere gratitude to his supervisor, Yuming Liu, for his continued support, encouragement, and many helpful discussions. He is also grateful to Aaron Chan for valuable suggestions and discussions on skew-group algebras, covering theory, and fractional Brauer graph algebras, and to Zhengfang Wang for generously sharing his knowledge of $A_\infty$-categories and gentle algebras. He thanks Pengyun Chen and Nengqun Li for many stimulating discussions. This work is supported by the China Scholarship Council (No.~202506040127).

%%%%%%%%%%%%%%%%%%%%%%%%%%%%%%%%%%%%%%%%%%%%%%%%%%%%%%%%%%%%%%%%%%%%%%%%%%%%%%%%%%%
\section{Preliminaries on categories and derived constructions}
\label{sec:ainfinity-preliminaries}

In this section, we recall the basic notions concerning $A_\infty$-categories and their derived categories that will be used in our covering constructions and derived-equivalence arguments. For general background on $A_\infty$-algebras, categories, and modules, we refer to \cite{Kel01,Kel06,LH03}. A detailed treatment of twisted complexes, exact triangles, and idempotent completion can be found in \cite[Part I]{Sei08}. For derived categories of dg modules, compact generation, and Morita theory, we refer to \cite{Kel94,Kel06dg}.

\subsection{\texorpdfstring{$A_\infty$}{A-infinity}-categories and functors}

We work with strictly unital $A_\infty$-categories and use the sign convention of \cite[Section 2.1]{OZ22}.

\begin{dfn}
\label{dfn:ainfinity}
An \defn{$A_\infty$-category} $\A$ consists of a set of objects $\Ob\A$, a graded $\Bbbk$-vector space $\Hom_\A(X,Y)$ for each pair of objects $X,Y\in\Ob\A$, and graded $\Bbbk$-linear maps
\[
 \mu^n_\A:
 \Hom_\A(X_{n-1},X_n)\otimes\cdots\otimes\Hom_\A(X_0,X_1)
 \longrightarrow \Hom_\A(X_0,X_n)
\]
of degree $2-n$ for every $n\geq1$ and every sequence of objects $X_0,\ldots,X_n$. These maps satisfy the \defn{Stasheff identities}: for every $n\geq1$ and every sequence of homogeneous morphisms $a_i\in\Hom_\A(X_{i-1},X_i)$, $1\leq i\leq n$,
\begin{equation*}
\tag{SI}\label{eq:relations}
\begin{split}
 \sum_{m=1}^{n}\sum_{j=0}^{n-m}
 &(-1)^{\sum_{i=1}^j\rdeg{a_i}}
 \mu_\A^{n-m+1}\bigl(a_n,\ldots,a_{j+m+1},
 \mu_\A^m(a_{j+m},\ldots,a_{j+1}),a_j,\ldots,a_1\bigr)=0.
\end{split}
\end{equation*}
The category $\A$ is \defn{minimal} if $\mu^1_\A=0$. The category $\A$ is \defn{strictly unital} if each object $X$ admits an element $1_X\in\Hom_\A^0(X,X)$, called a \defn{strict unit}, such that
\[
 \mu^1_\A(1_X)=0,\qquad
 \mu^2_\A(a,1_X)=a,\qquad
 \mu^2_\A(1_Y,a)=(-1)^{|a|}a
\]
for every homogeneous $a\in\Hom_\A(X,Y)$, and $\mu^n_\A$ vanishes whenever one of its inputs is a strict unit and $n\geq3$.
\end{dfn}

$A_\infty$-categories generalize differential graded categories. For completeness, recall that a \defn{differential graded category} or \defn{dg category}, has cochain complexes as morphism spaces and degree-zero associative composition satisfying the Leibniz rule
\[
 d(b\circ a)=d(b)\circ a+(-1)^{|b|}b\circ d(a)
\]
for homogeneous composable morphisms $a:X\to Y$ and $b:Y\to Z$.

For an $A_\infty$-category $\A$, put
\[
 d(a)=(-1)^{|a|}\mu^1_\A(a),\qquad
 b\circ a=(-1)^{|a|}\mu^2_\A(b,a)
\]
for homogeneous composable morphisms $a$ and $b$. The case $n=1$ of \eqref{eq:relations} gives $(\mu^1_\A)^2=0$, and hence $d^2=0$. The case $n=2$ gives the Leibniz rule above, while the case $n=3$ expresses associativity up to a homotopy determined by $\mu^3_\A$. In particular, an $A_\infty$-category with $\mu^n_\A=0$ for all $n\geq3$ is a dg category under this sign conversion, and every dg category arises in this way.

The \defn{homotopy category} $H^0(\A)$ has the same objects as $\A$, with morphism spaces
\[
 \Hom_{H^0(\A)}(X,Y)
 =H^0(\Hom_\A(X,Y),\mu^1_\A).
\]
Composition is given by $[b]\circ[a]=[\mu^2_\A(b,a)]$ for degree-zero morphisms $a:X\to Y$ and $b:Y\to Z$ satisfying $\mu^1_\A(a)=\mu^1_\A(b)=0$.

\begin{dfn}
\textnormal{(cf.~\cite[Sections 3.1 and 3.5]{Kel01})}
\label{dfn:associated-algebra}\label{dfn:total-algebra}
An \defn{$A_\infty$-algebra} is a graded $\Bbbk$-vector space $A$ equipped with graded $\Bbbk$-linear maps
\[
 \mu_A^n:A^{\otimes n}\longrightarrow A,\qquad n\geq1,
\]
of degree $2-n$, satisfying the Stasheff identities \eqref{eq:relations}. The algebra $A$ is \defn{strictly unital} if it admits an element $1_A\in A^0$ satisfying the unit conditions in \cref{dfn:ainfinity}.

Let $\A$ be an $A_\infty$-category with a finite nonempty object set. Its \defn{associated $A_\infty$-algebra} is
\[
 A_\A=\bigoplus_{X,Y\in\Ob\A}\Hom_\A(X,Y),
\]
with operations obtained from those of $\A$ by multilinear extension, assigning zero to tensors of noncomposable morphisms. The strict unit of $A_\A$ is $1_{A_\A}=\sum_{X\in\Ob\A}1_X$. We retain the distinguished idempotents $e_X=1_X$ and the decomposition of $A_\A$ into the morphism spaces $\Hom_\A(X,Y)$.
\end{dfn}

In our sign convention, the product associated with $\mu_A^2$ is $ b\cdot a=(-1)^{|a|}\mu_A^2(b,a)$ for homogeneous elements $a,b\in A$.

\begin{dfn}
\textnormal{(cf.~\cite[Definition 2.5]{BSW24})}
\label{dfn:strict-functor}
An \defn{$A_\infty$-functor} $F:\A\to\B$ consists of a map $F:\Ob\A\to\Ob\B$ and graded $\Bbbk$-linear maps
\[
 F^n:  \Hom_\A(X_{n-1},X_n)\otimes\cdots\otimes\Hom_\A(X_0,X_1)
 \longrightarrow \Hom_\B(FX_0,FX_n)
\]
of degree $1-n$ for every $n\geq1$ and every sequence $X_0,\ldots,X_n\in\Ob\A$. For each decomposition $n=i_1+\cdots+i_r$ into positive integers, put $I_0=0$ and $I_t=i_1+\cdots+i_t$ for $1\leq t\leq r$. For homogeneous morphisms $a_i\in\Hom_\A(X_{i-1},X_i)$, require
\[
\begin{aligned}
 &\sum_{m=1}^n\sum_{j=0}^{n-m}
 (-1)^{\sum_{i=1}^j\rdeg{a_i}}
 F^{n-m+1}\bigl(
 a_n,\ldots,a_{j+m+1},
 \mu^m_\A(a_{j+m},\ldots,a_{j+1}),
 a_j,\ldots,a_1\bigr)\\
 &\qquad=
 \sum_{\substack{r\geq1,\ i_1,\ldots,i_r\geq1\\
                  i_1+\cdots+i_r=n}}
 \mu^r_\B\bigl(
 F^{i_r}(a_{I_r},\ldots,a_{I_{r-1}+1}),
 \ldots,F^{i_1}(a_{I_1},\ldots,a_1)\bigr).
\end{aligned}
\]
We also require $F^1(1_X)=1_{FX}$ for every $X\in\Ob\A$, and $F^n$ vanishes whenever one of its inputs is a strict unit and $n\geq2$.
\end{dfn}

Let $F:\A\to\B$ be an $A_\infty$-functor. Its first component induces a functor $H^0(F):H^0(\A)\to H^0(\B)$. We use the following terminology.
\begin{itemize}
\item The functor $F$ is \defn{strict} if $F^n=0$ for all $n\geq2$. In this case, we also write $F$ for $F^1$, and the identities reduce to $F\bigl(\mu^n_\A(a_n,\ldots,a_1)\bigr)=\mu^n_\B(Fa_n,\ldots,Fa_1)$.

\item The functor $F$ is a \defn{quasi-equivalence} if $ F^1:\Hom_\A(X,Y)\rightarrow\Hom_\B(FX,FY)$ is a quasi-isomorphism for every $X,Y\in\Ob\A$, and every object of $H^0(\B)$ is isomorphic to $FX$ for some $X\in\Ob\A$.

\item The functor $F$ is a \defn{full inclusion} if it is strict and identifies $\Ob\A$ with a subset of $\Ob\B$ and $\Hom_\A(X,Y)$ with $\Hom_\B(X,Y)$ for all $X,Y\in\Ob\A$, with all operations restricted from $\B$.
\end{itemize}

Note that two $A_\infty$-categories are called \defn{quasi-equivalent} if they are connected by a zigzag of quasi-equivalences.

\begin{dfn}
\textnormal{(cf.~\cite[Definition 2.6]{BSW24})}
\label{dfn:formality}
Let $H^*(\A)$ denote the $A_\infty$-category with the same objects as $\A$ and morphism spaces
\[
 \Hom_{H^*(\A)}(X,Y)
 =H^*(\Hom_\A(X,Y),\mu^1_\A).
\]
Its operations are given by
\[
 \mu^2_{H^*(\A)}([b],[a])=[\mu^2_\A(b,a)],
 \qquad
 \mu^n_{H^*(\A)}=0\quad(n\ne2),
\]
where $a:X\to Y$ and $b:Y\to Z$ are homogeneous morphisms satisfying $\mu^1_\A(a)=\mu^1_\A(b)=0$.

The category $\A$ is \defn{formal} if it is quasi-equivalent to $H^*(\A)$.
\end{dfn}

%%%%%%%%%%%%%%%%%%%%%%%%%%%%%%%%%%
\subsection{\texorpdfstring{$A_\infty$}{A-infinity}-modules and derived categories}

We introduce right $A_\infty$-modules and their morphism complexes to define the derived category $\D(\A)$. Representable modules will relate this construction to twisted complexes in the next subsection.

\begin{dfn}
\textnormal{(cf.~\cite[Sections 4.2 and 7.4]{Kel01})}
\label{dfn:ainfinity-module}
Let $\A$ be an $A_\infty$-category. A \defn{right $A_\infty$-module} $M$ over $\A$ assigns a graded $\Bbbk$-vector space $M(X)$ to every $X\in\Ob\A$ and graded $\Bbbk$-linear maps
\[
 \mu_M^{d+1}:
 M(X_d)\otimes\Hom_\A(X_{d-1},X_d)\otimes\cdots
 \otimes\Hom_\A(X_0,X_1)
 \longrightarrow M(X_0)
\]
of degree $1-d$, for every $d\geq0$ and every sequence $X_0,\ldots,X_d$ of objects of $\A$. For homogeneous elements $m\in M(X_d)$ and $a_i\in\Hom_\A(X_{i-1},X_i)$, these maps satisfy
\[
\begin{aligned}
 0={}&
 \sum_{j=0}^d
 (-1)^{\sum_{i=1}^j\rdeg{a_i}}
 \mu_M^{j+1}\bigl(
 \mu_M^{d-j+1}(m,a_d,\ldots,a_{j+1}),
 a_j,\ldots,a_1
 \bigr)\\
 &+
 \sum_{r=1}^d\sum_{j=0}^{d-r}
 (-1)^{\sum_{i=1}^j\rdeg{a_i}}
 \mu_M^{d-r+2}\bigl(
 m,a_d,\ldots,a_{j+r+1},
 \mu_\A^r(a_{j+r},\ldots,a_{j+1}),
 a_j,\ldots,a_1
 \bigr).
\end{aligned}
\]
In particular, the case $d=0$ gives $(\mu_M^1)^2=0$. We also require $\mu_M^2(m,1_X)=m$ for every $X\in\Ob\A$ and $m\in M(X)$, and, for $d\geq3$, $\mu_M^d$ vanishes whenever one of its $\A$-morphism inputs is a strict unit.
\end{dfn}

Recall that for $X\in\Ob\A$, the \defn{representable module} $y_X$ is given by $y_X(Y)=\Hom_\A(Y,X)$ for all $Y\in\Ob\A$, with operations
\[
 \mu_{y_X}^{d+1}(b,a_d,\ldots,a_1)
 =
 \mu_\A^{d+1}(b,a_d,\ldots,a_1),
 \qquad d\geq0,
\]
where $b\in\Hom_\A(X_d,X)$ and $a_i\in\Hom_\A(X_{i-1},X_i)$.

The dg category $\mathcal C_\infty(\A)$ has the above right $A_\infty$-modules as objects. For modules $M,N$ and $q\in\ZZ$, an element $\phi\in\Hom_{\mathcal C_\infty(\A)}^q(M,N)$ is a family of graded $\Bbbk$-linear maps
\[
 \phi^{d+1}:
 M(X_d)\otimes\Hom_\A(X_{d-1},X_d)\otimes\cdots
 \otimes\Hom_\A(X_0,X_1)
 \longrightarrow N(X_0),
 \qquad d\geq0,
\]
of degree $q-d$, for all sequences $X_0,\ldots,X_d$. For $d\geq1$, we require $\phi^{d+1}$ to vanish whenever one of its $\A$-morphism inputs is a strict unit. For $d=0$, we write $\phi_X^1:M(X)\longrightarrow N(X)$ for the component at $X$.

For homogeneous $m\in M(X_d)$ and $a_i\in\Hom_\A(X_{i-1},X_i)$, the differential is given by
\[
\begin{aligned}
 &\bigl(d_{\mathcal C_\infty(\A)}\phi\bigr)^{d+1}
       (m,a_d,\ldots,a_1)\\
 &\quad=
 \sum_{j=0}^d(-1)^{q\sum_{i=1}^j\rdeg{a_i}}
 \mu_N^{j+1}\bigl(
 \phi^{d-j+1}(m,a_d,\ldots,a_{j+1}),
 a_j,\ldots,a_1
 \bigr)\\
 &\qquad-
 (-1)^q\sum_{j=0}^d(-1)^{\sum_{i=1}^j\rdeg{a_i}}
 \phi^{j+1}\bigl(
 \mu_M^{d-j+1}(m,a_d,\ldots,a_{j+1}),
 a_j,\ldots,a_1
 \bigr)\\
 &\qquad-
 (-1)^q\sum_{r=1}^d\sum_{j=0}^{d-r}
 (-1)^{\sum_{i=1}^j\rdeg{a_i}}
 \phi^{d-r+2}\bigl(
 m,a_d,\ldots,a_{j+r+1},
 \mu_\A^r(a_{j+r},\ldots,a_{j+1}),
 a_j,\ldots,a_1
 \bigr).
\end{aligned}
\]
For homogeneous morphisms $\phi:M\to N$ and $\psi:N\to P$, with $|\phi|=q$, composition is given by
\[
\begin{aligned}
 &(\psi\circ\phi)^{d+1}(m,a_d,\ldots,a_1)=
 \sum_{j=0}^d(-1)^{q\sum_{i=1}^j\rdeg{a_i}}
 \psi^{j+1}\bigl(
 \phi^{d-j+1}(m,a_d,\ldots,a_{j+1}),
 a_j,\ldots,a_1
 \bigr).
\end{aligned}
\]
The identity of $M$ has components $(1_M)_X^1=1_{M(X)}$ and $(1_M)^{d+1}=0$ for $d\geq1$. These formulas define a dg category; cf.~\cite[Sections 6.3 and 7.4]{Kel01}.

An $A_\infty$-module morphism $\phi:M\to N$ is a degree-zero element satisfying $d_{\mathcal C_\infty(\A)}(\phi)=0$. In particular, its first components satisfy $ \mu_N^1\phi_X^1=\phi_X^1\mu_M^1$. We call $\phi$ a \defn{quasi-isomorphism} if $\phi_X^1:(M(X),\mu_M^1)\rightarrow(N(X),\mu_N^1)$ is a quasi-isomorphism for every $X\in\Ob\A$. Since $\Bbbk$ is a field, every such quasi-isomorphism is a homotopy equivalence; see \cite[Section 4.2]{Kel01}.

\begin{dfn}
\label{dfn:derived-categories}
Let $\A$ be an $A_\infty$-category.
\begin{itemize}
\item The \defn{derived category} of $\A$ is $\D(\A)=H^0\bigl(\mathcal C_\infty(\A)\bigr).$

\item For a collection $\mathcal S$ of objects of a triangulated category $\T$, let $\thick_\T(\mathcal S)$ denote the smallest full triangulated subcategory containing $\mathcal S$ and closed under isomorphisms and direct summands. The \defn{perfect derived category} of $\A$ is
\[
 \per(\A)
 =
 \thick_{\D(\A)}\{y_X:X\in\Ob\A\}.
\]
Its objects are called \defn{perfect}.
\end{itemize}
\end{dfn}

%%%%%%%%%%%%%%%%%%%%%%%%%%%%%%%%%%
\subsection{Twisted complexes and perfect categories}

We recall the definition of twisted complexes, which generalize bounded complexes of finitely generated projective modules to the setting of $A_\infty$-categories; see \cite[Remark 2.6]{OZ22}. The Yoneda construction below relates twisted complexes to the perfect modules defined in \cref{dfn:derived-categories}.

Let $\A$ be an $A_\infty$-category.  We denote by $\Add(\A)$ the $A_\infty$-category obtained from $\A$ by formally adjoining shifts and finite direct sums. Its objects are
\[
 E=\bigoplus_{i=1}^N E_i[s_i],
 \qquad E_i\in\Ob\A,\quad s_i\in\ZZ,
\]
with the empty sum giving the zero object. For $F=\bigoplus_{j=1}^M F_j[t_j]$, set
\[
 \Hom_{\Add(\A)}^q(E,F)
 =
 \bigoplus_{i=1}^N\bigoplus_{j=1}^M
 \Hom_\A^{q+t_j-s_i}(E_i,F_j).
\]
Note that for a morphism $f:E\to F$, we write $f_{ij}$ for its component from $E_i[s_i]$ to $F_j[t_j]$.

For $n\geq1$, let $E^k=\bigoplus_{i\in I_k}E_i^k[s_i^k]$, $0\leq k\leq n$, and let $f_k\in\Hom_{\Add(\A)}(E^{k-1},E^k)$, $1\leq k\leq n$, be homogeneous morphisms. For $i_0\in I_0$ and $i_n\in I_n$, the component of $\mu^n_{\Add(\A)}(f_n,\ldots,f_1)$ from $E_{i_0}^0[s_{i_0}^0]$ to $E_{i_n}^n[s_{i_n}^n]$ is defined by
\[
\begin{aligned}
 &\bigl(\mu^n_{\Add(\A)}(f_n,\ldots,f_1)\bigr)_{i_0i_n}=
 \sum_{\substack{i_1\in I_1,\ldots,i_{n-1}\in I_{n-1}}}
 (-1)^{s_{i_0}^0+\cdots+s_{i_{n-1}}^{n-1}}
 \mu^n_\A\bigl(
 (f_n)_{i_{n-1}i_n},\ldots,(f_1)_{i_0i_1}
 \bigr).
\end{aligned}
\]
The strict unit of $E$ has diagonal entries $1_{E_i}$ and all other entries zero.

\begin{dfn}
\textnormal{(cf.~\cite[Definition 2.5]{OZ22})}
\label{dfn:twisted}
A \defn{twisted complex} over $\A$ is a pair $(E,\delta)$, where $E=\bigoplus_{i=1}^N E_i[s_i]$ is an object of $\Add(\A)$ and $\delta\in\Hom_{\Add(\A)}^1(E,E)$ satisfies
\[
 \delta_{ij}\in\Hom_\A^{1+s_j-s_i}(E_i,E_j),
 \qquad
 \delta_{ij}=0\quad\text{unless }i<j,
\]
together with the \defn{Maurer--Cartan equation}
\begin{equation*}
\tag{MC}\label{eq:mc-prelim}
 \sum_{r\geq1}
 \mu^r_{\Add(\A)}(\delta,\ldots,\delta)=0.
\end{equation*}
The Maurer--Cartan equation generalizes the square-zero condition on a differential to $A_\infty$-categories. Since the matrix $(\delta_{ij})$ is strictly upper triangular, a nonzero contribution to $\mu^r_{\Add(\A)}(\delta,\ldots,\delta)$ requires
\[
 1\leq i_0<i_1<\cdots<i_r\leq N.
\]
Thus all terms with $r\geq N$ vanish, and the sum in \eqref{eq:mc-prelim} is finite.

The \defn{category of twisted complexes} $\Tw\A$ has these pairs as objects and graded morphism spaces
\[
 \Hom_{\Tw\A}\bigl((E,\delta_E),(F,\delta_F)\bigr)
 =
 \Hom_{\Add(\A)}(E,F).
\]
For $n\geq1$ and homogeneous morphisms $f_i:(E_{i-1},\delta_{i-1})\to(E_i,\delta_i)$, $1\leq i\leq n$, its operations are
\begin{equation}\label{eq:twisted-operations}
\begin{aligned}
 &\mu^n_{\Tw\A}(f_n,\ldots,f_1)=
 \sum_{j_0,\ldots,j_n\geq0}
 \mu^{n+j_0+\cdots+j_n}_{\Add(\A)}
 \bigl(
 \delta_n^{\otimes j_n}\otimes f_n
 \otimes\delta_{n-1}^{\otimes j_{n-1}}\otimes\cdots
 \otimes f_1\otimes\delta_0^{\otimes j_0}
 \bigr).
\end{aligned}
\end{equation}
Here $\delta_i^{\otimes j_i}$ denotes $j_i$ tensor factors equal to $\delta_i$, with these factors omitted when $j_i=0$. The strict upper triangularity of each $\delta_i$ again ensures that the sum is finite.
\end{dfn}

\begin{prop}
\textnormal{(cf.~\cite[Section 2.2]{OZ22} and \cite[Section 2.5]{BSW24})}
\label{prop:twisted-standard}
Let $\A$ and $\B$ be $A_\infty$-categories. The homotopy category $H^0(\Tw\A)$ is triangulated. Every strict $A_\infty$-functor $F:\A\to\B$ extends to a strict $A_\infty$-functor $\Tw F:\Tw\A\rightarrow\Tw\B$. For $E=\bigoplus_{i=1}^N E_i[s_i]$, its action on objects is
\[
 \Tw F(E,\delta)
 =
 \left(
 \bigoplus_{i=1}^N F(E_i)[s_i],
 \bigl(F(\delta_{ij})\bigr)_{i,j}
 \right),
\]
and its action on morphisms is $\bigl((\Tw F)^1(f)\bigr)_{ij}=F(f_{ij})$. The functor $H^0(\Tw F):H^0(\Tw\A)\to H^0(\Tw\B)$ is exact. If $F$ is a full inclusion, then so is $\Tw F$.
\end{prop}

Let $\T$ be an additive category. Its \defn{idempotent completion} $\T^\natural$ has objects $(E,p)$, where $E\in\Ob\T$ and $p\in\End_\T(E)$ satisfies $p^2=p$, and morphism spaces
\[
 \Hom_{\T^\natural}\bigl((E,p),(F,q)\bigr)
 =
 \{f\in\Hom_\T(E,F):f=qfp\}.
\]
Composition is inherited from $\T$, and the identity of $(E,p)$ is $p$. We call $\T$ \defn{split-closed} if the canonical inclusion $\T\to\T^\natural$, sending $E$ to $(E,1_E)$ and each morphism to itself, is an equivalence. See \cite[Section 2.4]{BSW24}.

We identify each $X\in\Ob\A$ with $(X,0)\in\Ob(\Tw\A)$. Applying the Yoneda functor of $\Tw\A$ and restricting modules along $\A\hookrightarrow\Tw\A$ gives an $A_\infty$-functor
\[
 \mathcal Y:\Tw\A\longrightarrow\mathcal C_\infty(\A),
 \qquad T\longmapsto y_T,
\]
extending $X\mapsto y_X$. Explicitly, for $T=(E,\delta)$ with $E=\bigoplus_{i=1}^N E_i[s_i]$, its underlying graded spaces are
\[
 y_T(X)
 =
 \Hom_{\Tw\A}((X,0),T)
 =
 \bigoplus_{i=1}^N\Hom_\A(X,E_i)[s_i].
\]
For $d\geq0$, its module operations are
\[
\begin{aligned}
 &\mu_{y_T}^{d+1}(b,a_d,\ldots,a_1)=
 \sum_{r\geq0}
 \mu_{\Add(\A)}^{d+r+1}
 \bigl(
 \delta^{\otimes r}\otimes b\otimes
 a_d\otimes\cdots\otimes a_1
 \bigr),
\end{aligned}
\]
where $b\in\Hom_{\Add(\A)}(X_d,E)$ and $a_i\in\Hom_\A(X_{i-1},X_i)$. The sum is finite because $\delta$ is strictly upper triangular.

On homotopy categories, a closed degree-zero morphism $f:T\to T'$ is sent to the class of the module morphism $\phi_f:y_T\to y_{T'}$ with components
\[
 \phi_f^{d+1}(b,a_d,\ldots,a_1)
 =
 (-1)^{|b|+\sum_{i=1}^d\rdeg{a_i}}
 \mu_{\Tw\A}^{d+2}(f,b,a_d,\ldots,a_1),
 \qquad d\geq0.
\]
The induced functor
\[
 H^0(\mathcal Y):H^0(\Tw\A)\longrightarrow\D(\A)
\]
is fully faithful and exact. Its essential image is the smallest full triangulated subcategory of $\D(\A)$ containing all $y_X$ and closed under isomorphisms; see \cite[Sections 7.5--7.6]{Kel01}. It therefore induces an equivalence
\[
 H^0(\Tw\A)^\natural\simeq\per(\A);
\]
see also \cite[Remark 2.15]{BSW24}.

%%%%%%%%%%%%%%%%%%%%%%%%%%%%%%%%%%
\subsection{Compact generators and Morita equivalence}

We recall compact generation and the Morita criteria used in our derived-equivalence constructions.

Let $\T$ be a triangulated category admitting arbitrary direct sums. An object $P\in\T$ is \defn{compact} if the natural map
\[
 \bigoplus_{i\in I}\Hom_\T(P,M_i)
 \longrightarrow
 \Hom_\T\Bigl(P,\bigoplus_{i\in I}M_i\Bigr)
\]
is an isomorphism for every family $(M_i)_{i\in I}$ of objects of $\T$, with $I$ a set. A set $\mathcal S$ of compact objects \defn{compactly generates} $\T$ if the smallest full triangulated subcategory containing $\mathcal S$ and closed under isomorphisms and arbitrary direct sums is $\T$. An object $P$ is a \defn{compact generator} if $\{P\}$ compactly generates $\T$.

For every $X\in\Ob\A$, $M\in\D(\A)$ and $j\in\ZZ$, the Yoneda isomorphism gives
\[
 \Hom_{\D(\A)}(y_X,M[j])
 \cong H^j(M(X),\mu_M^1).
\]
The category $\D(\A)$ admits arbitrary direct sums, computed objectwise, and the representable modules $y_X$ compactly generate it. Moreover,
\[
 \D(\A)^c=\per(\A),
\]
where $\D(\A)^c$ denotes the full subcategory of compact objects; see \cite[Sections 4--5]{Kel94} and \cite[Sections 6--7]{Kel01}.

Morita equivalences are defined using the categories of twisted complexes. Every $A_\infty$-functor $F:\A\to\B$ extends to an $A_\infty$-functor $\Tw F:\Tw\A\to\Tw\B$, and $H^0(\Tw F)$ is exact; see \cite[Part I, Section 3]{Sei08}. This generalizes the assertion for strict functors in \cref{prop:twisted-standard}.

\begin{dfn}\label{dfn:morita}
An $A_\infty$-functor $F:\A\to\B$ is a \defn{Morita equivalence} if the induced functor
\[
 H^0(\Tw F)^\natural:
 H^0(\Tw\A)^\natural
 \longrightarrow
 H^0(\Tw\B)^\natural
\]
is an equivalence. Two $A_\infty$-categories are called \defn{Morita equivalent} if they are connected by a zigzag of Morita equivalences.
\end{dfn}

\begin{prop}
\label{prop:morita-criterion}
Let $F:\A\to\B$ be an $A_\infty$-functor.
\begin{enumerate}
\item Suppose that $F$ is a full inclusion. Then the following are equivalent.
\begin{enumerate}
\item[(a)] $F$ is a Morita equivalence.

\item[(b)] $\thick_{\D(\B)}\{y_{FX}:X\in\Ob\A\}=\per(\B)$.

\item[(c)] For every $Y\in\Ob\B$, there exists $T_Y\in\Tw\A$ such that $(Y,0)$ is a direct summand of $(\Tw F)(T_Y)$ in $H^0(\Tw\B)$.
\end{enumerate}

\item If $F$ is a quasi-equivalence, then $F$ is a Morita equivalence.

\item If $F$ is a Morita equivalence, then there exists an exact equivalence $\widetilde F:\D(\A)\longrightarrow\D(\B)$ restricting to an equivalence $\widetilde F|_{\per(\A)}: \per(\A)\longrightarrow\per(\B)$.
\end{enumerate}
\end{prop}

\begin{proof}
\noindent
(1) This follows from \cite[Lemma 2.18]{BSW24}, under the Yoneda equivalences $H^0(\Tw\A)^\natural\simeq\per(\A)$ and $H^0(\Tw\B)^\natural\simeq\per(\B)$.

\medskip\noindent
(2) This is \cite[Remark 2.17]{BSW24}, which follows from \cite[Part I, Lemma 3.25]{Sei08}.

\medskip\noindent
(3) The existence of $\widetilde F$ and its restriction to perfect objects follow from Morita theory; see \cite[Section 8]{Kel94} and \cite[Sections 3.8 and 4.6]{Kel06dg}.
\end{proof}

For right $A_\infty$-modules $T,M$ over $\A$, we use
\[
 \RHom_\A(T,M)
 =
 \Hom_{\mathcal C_\infty(\A)}(T,M),
 \qquad
 \REnd_\A(T)
 =
 \Hom_{\mathcal C_\infty(\A)}(T,T).
\]
The dg algebra $\REnd_\A(T)$ has the differential and multiplication of $\mathcal C_\infty(\A)$; see \cite[Sections 6.3--6.4 and 7.4]{Kel01}. The right $\REnd_\A(T)$-action on $\RHom_\A(T,M)$ is given by precomposition $h\cdot e=h\circ e$.
If $T$ is a compact generator of $\D(\A)$, Morita theory gives an equivalence
\[
 \RHom_\A(T,-):
 \D(\A)\longrightarrow\D\bigl(\REnd_\A(T)\bigr);
\]
see \cite[Section 8]{Kel94}.

For a finite-dimensional ordinary algebra $A$, there is an equivalence
\[
 \per(A)\simeq\K^b(\proj\text{-}A).
\]
A \defn{tilting complex} over $A$ is an object $T\in\K^b(\proj\text{-}A)$ such that $\thick_{\K^b(\proj\text{-}A)}\{T\}=\K^b(\proj\text{-}A)$ and $\Hom_{\K^b(\proj\text{-}A)}(T,T[j])=0$ for every $j\ne0$.

%%%%%%%%%%%%%%%%%%%%%%%%%%%%%%%%%%%%%%%%%%%%%%%%%%%%%%%%%%%%%%%%%%%%%%%%%%%%%%%%%%%%%%%%%%%%%
\section{Coverings and lifting \texorpdfstring{$A_\infty$}{A-infinity}-structures}
\label{sec:lifting}

\subsection{Strict group actions and orbit categories}

We recall strict group actions and the orbit categories used in our lifting arguments.

Note that a \defn{strict $A_\infty$-automorphism} of $\A$ is a strict $A_\infty$-functor $F:\A\to\A$ whose object map is bijective and whose maps $F^1:\Hom_\A(X,Y)\to\Hom_\A(FX,FY)$ are bijective for all $X,Y\in\Ob\A$.

\begin{dfn}
\textnormal{(cf.~\cite[Definitions 5.4 and 5.6]{OZ22} and \cite[Section 2.6]{BSW24})}
\label{dfn:orbit}
Let $\A$ be an $A_\infty$-category and let $G$ be a group with identity $e$. A \defn{strict action} of $G$ on $\A$ assigns a strict $A_\infty$-automorphism $g:\A\to\A$ to every $g\in G$, such that $e=\id_\A$ and $g\circ h=gh$ for all $g,h\in G$. The action on objects is \defn{free} if $gX=X$ implies $g=e$ for every $g\in G$ and $X\in\Ob\A$.

The \defn{orbit category} $\A/G$ has
\[
 \Ob(\A/G)=\Ob\A,
 \qquad
 \Hom_{\A/G}(X,Y)
 =\bigoplus_{g\in G}\Hom_\A(X,gY).
\]
For $n\geq1$, $g_1,\ldots,g_n\in G$, and homogeneous morphisms $a_i\in\Hom_\A(X_{i-1},g_iX_i)$, define the operations by multilinear extension of
\begin{equation}\label{eq:orbit}
 \mu^n_{\A/G}(a_n,\ldots,a_1)
 =\mu^n_\A\bigl(
 (g_1\cdots g_{n-1})a_n,\ldots,g_1a_2,a_1
 \bigr).
\end{equation}
The output belongs to the summand $\Hom_\A(X_0,(g_1\cdots g_n)X_n)$.
\end{dfn}

For a choice of one representative from each $G$-orbit in $\Ob\A$, let $(\A/G)_{\mathrm{rep}}$ denote the full $A_\infty$-subcategory of $\A/G$ on the chosen objects.

\begin{prop}
\textnormal{(cf.~\cite[Lemma 5.7 and Corollary 5.8]{OZ22} and \cite[Lemmas 2.20--2.21]{BSW24})}
\label{prop:orbit-standard}
Let $G$ act strictly on an $A_\infty$-category $\A$.
\begin{enumerate}
\item The operations \eqref{eq:orbit} make $\A/G$ an $A_\infty$-category. The strict unit at $X$ is $1_X$ in the $e$-summand. The identity on objects and the inclusions
\[
 \Hom_\A(X,Y)\hookrightarrow
 \bigoplus_{g\in G}\Hom_\A(X,gY)
\]
into the $e$-summands define a strict functor $\A\to\A/G$.

\item The action of $G$ extends to a strict action on $\Tw\A$. For $(E,\delta)\in\Tw\A$, with $E=\bigoplus_{i=1}^N E_i[s_i]$, the extension is given by
\[
 g(E,\delta)
 =\left(\bigoplus_{i=1}^N(gE_i)[s_i],g\delta\right),
 \qquad
 (g\delta)_{ij}=g(\delta_{ij}).
\]
For a morphism $f$ in $\Tw\A$, its components transform as $(gf)_{ij}=g(f_{ij})$.

\item The inclusion $(\A/G)_{\mathrm{rep}}\hookrightarrow\A/G$ is a quasi-equivalence and hence a Morita equivalence.
\end{enumerate}
\end{prop}

\begin{proof}
\noindent
(1) The proof of \cite[Lemma 5.7]{OZ22} applies to any strict action. For arbitrary $G$, every morphism has only finitely many nonzero components in the defining direct sum, and each fixed \textup{(SI)} has finitely many summands. Thus all operations and identities are defined by finite sums on each input. The unit identities and the strictness of $\A\to\A/G$ follow from \eqref{eq:orbit}.

\medskip\noindent
(2) This follows from \cref{prop:twisted-standard}, since the extension preserves composition of strict functors.

\medskip\noindent
(3) For $X\in\Ob\A$ and $g\in G$, regard $1_X$ as a degree-zero morphism $X\to gX$ through the summand $\Hom_\A(X,g^{-1}(gX))=\Hom_\A(X,X)$ of $\Hom_{\A/G}(X,gX)$. Similarly, regard $1_{gX}$ as a degree-zero morphism $gX\to X$ through the $g$-summand $\Hom_\A(gX,gX)$ of $\Hom_{\A/G}(gX,X)$. Both morphisms satisfy $\mu_{\A/G}^1(1_X)=0$ and $\mu_{\A/G}^1(1_{gX})=0$, so they represent morphisms in $H^0(\A/G)$. Since the action preserves units, \eqref{eq:orbit} gives
\[
\begin{aligned}
 \mu_{\A/G}^2(1_{gX},1_X)
 &=\mu_\A^2(g^{-1}(1_{gX}),1_X)=1_X,\\
 \mu_{\A/G}^2(1_X,1_{gX})
 &=\mu_\A^2(g(1_X),1_{gX})=1_{gX},
\end{aligned}
\]
where the outputs lie in the $e$-summands. Thus $X$ and $gX$ are isomorphic in $H^0(\A/G)$.

Every object of $\A/G$ is therefore isomorphic in $H^0(\A/G)$ to a chosen orbit representative. The inclusion $(\A/G)_{\mathrm{rep}}\hookrightarrow\A/G$ is the identity on each morphism complex between chosen representatives, and hence is a quasi-equivalence. It is consequently a Morita equivalence by \cref{prop:morita-criterion}.
\end{proof}

%%%%%%%%%%%%%%%%%%%%%%%%%%%%%%%%%%%%%
\subsection{The strict lifting criterion}
\label{sec:strict-lifting}

In the preceding subsection, we constructed the $A_\infty$-structure on an orbit category $\A/G$ from a strict group action on an $A_\infty$-category $\A$.

We now consider the natural converse problem: given an $A_\infty$-category $\B$ and covering data specified by a decomposition of its morphism spaces into group weights, we seek an $A_\infty$-structure on the covering for which the projection to $\B$ is a strict $A_\infty$-functor. We begin with a lemma showing how injective maps compatible with the operations detect \textup{(SI)} and the strict unit identities.

\begin{lem}\label{lem:faithful}
Let $\C$ consist of a set $\Ob\C$, a graded $\Bbbk$-vector space $\Hom_\C(X,Y)$ for each $X,Y\in\Ob\C$, and graded $\Bbbk$-linear maps
\[
 \mu_\C^n:
 \Hom_\C(X_{n-1},X_n)\otimes\cdots\otimes\Hom_\C(X_0,X_1)
 \longrightarrow\Hom_\C(X_0,X_n)
\]
of degree $2-n$, for every $n\geq1$ and every sequence $X_0,\ldots,X_n\in\Ob\C$. Let $\B$ be an $A_\infty$-category. Suppose that a map $F:\Ob\C\to\Ob\B$ and injective degree-zero $\Bbbk$-linear maps
$ F_{X,Y}:\Hom_\C(X,Y)\rightarrow\Hom_\B(FX,FY)$ satisfy
\[
 F_{X_0,X_n}\bigl(\mu_\C^n(a_n,\ldots,a_1)\bigr)
 =\mu_\B^n(Fa_n,\ldots,Fa_1)
\]
for every $n\geq1$ and homogeneous $a_i\in\Hom_\C(X_{i-1},X_i)$, where $Fa_i=F_{X_{i-1},X_i}(a_i)$.  Then the operations $\mu_\C^n$ satisfy \textup{(SI)}.

If, for every $X\in\Ob\C$, there exists an element $1_X\in\Hom_\C^0(X,X)$ such that $F_{X,X}(1_X)=1_{FX}$, then these elements are strict units. In this case, $\C$ is a strictly unital $A_\infty$-category, and $F$ defines a strict $A_\infty$-functor $\C\to\B$.
\end{lem}

\begin{proof}
Fix $n\geq1$ and homogeneous $a_i\in\Hom_\C(X_{i-1},X_i)$. Apply $F_{X_0,X_n}$ to the left-hand side of \textup{(SI)} for the operations of $\C$ on these inputs. Compatibility with the outer and inner operations replaces each summand by the corresponding summand for $\B$ on $Fa_n,\ldots,Fa_1$. The signs agree because $\rdeg{Fa_i}=\rdeg{a_i}$. The resulting sum vanishes by \textup{(SI)} for $\B$, so injectivity of $F_{X_0,X_n}$ proves \textup{(SI)} for $\C$.

Suppose now that $F_{X,X}(1_X)=1_{FX}$ for every $X$. For homogeneous $a\in\Hom_\C(X,Y)$, compatibility gives
\[
\begin{aligned}
 F_{X,X}\bigl(\mu_\C^1(1_X)\bigr)&=0,\\
 F_{X,Y}\bigl(\mu_\C^2(a,1_X)-a\bigr)&=0,\\
 F_{X,Y}\bigl(\mu_\C^2(1_Y,a)-(-1)^{|a|}a\bigr)&=0.
\end{aligned}
\]
For $n\geq3$, the image of $\mu_\C^n(a_n,\ldots,a_1)$ under $F_{X_0,X_n}$ also vanishes whenever an input is one of the elements $1_X$. Injectivity therefore gives all the strict unit identities. The stated compatibility then makes $F$ a strict $A_\infty$-functor.
\end{proof}

\begin{thm}\label{prop:lift}
Let $G$ be a group with identity $e$, and let $\B$ be an $A_\infty$-category equipped with decompositions
\[
 \Hom_\B(X,Y)=\bigoplus_{g\in G}\Hom_\B^g(X,Y),
 \qquad 1_X\in\Hom_\B^e(X,X),
\]
for all $X,Y\in\Ob\B$. Each $\Hom_\B^g(X,Y)$ is a cohomologically graded subspace; the superscript $g$ specifies its \defn{group weight}. Consider the object set and graded morphism spaces
\[
 \Ob\C=\Ob\B\times G,
 \qquad
 \Hom_\C((X,h),(Y,l))=\Hom_\B^{h^{-1}l}(X,Y).
\]
Define $\pi$ on objects by $\pi(X,h)=X$ and on morphism spaces by the inclusions
\[
 \pi_{(X,h),(Y,l)}:
 \Hom_\B^{h^{-1}l}(X,Y)\hookrightarrow\Hom_\B(X,Y).
\]
Then the following are equivalent.
\begin{enumerate}[label=\textup{(\arabic*)}]
\item There exist operations $\mu_\C^n$ and strict units on the prescribed objects and morphism spaces such that $\C$ is an $A_\infty$-category and the prescribed maps $\pi$ define a strict $A_\infty$-functor $\pi:\C\to\B$.

\item For every $n\geq1$, $X_0,\ldots,X_n\in\Ob\B$, and $g_1,\ldots,g_n\in G$, we have
\begin{equation}\label{eq:homogeneous}
\begin{aligned}
 &\mu_\B^n\bigl(
 \Hom_\B^{g_n}(X_{n-1},X_n)\otimes\cdots
 \otimes\Hom_\B^{g_1}(X_0,X_1)
 \bigr)\subseteq\Hom_\B^{g_1\cdots g_n}(X_0,X_n).
\end{aligned}
\end{equation}
\end{enumerate}
\end{thm}

\begin{proof}
$\underline{(1)\Rightarrow(2)}$. Fix the resulting $A_\infty$-structure on $\C$. For $n\geq1$, choose homogeneous $a_i\in\Hom_\B^{g_i}(X_{i-1},X_i)$. Set $h_0=e$ and $h_i=g_1\cdots g_i$ for $1\leq i\leq n$. Since $h_{i-1}^{-1}h_i=g_i$, we can regard $a_i$ as an element of $\Hom_\C((X_{i-1},h_{i-1}),(X_i,h_i))$. Strictness of $\pi$ gives
\[
 \mu_\B^n(a_n,\ldots,a_1)
 =\pi_{(X_0,e),(X_n,h_n)}
   \bigl(\mu_\C^n(a_n,\ldots,a_1)\bigr).
\]
The image of $\pi_{(X_0,e),(X_n,h_n)}$ is $\Hom_\B^{g_1\cdots g_n}(X_0,X_n)$. Thus \eqref{eq:homogeneous} follows.

\smallskip\noindent
$\underline{(2)\Rightarrow(1)}$. For objects $(X_0,h_0),\ldots,(X_n,h_n)$, the equality $(h_0^{-1}h_1)(h_1^{-1}h_2)\cdots(h_{n-1}^{-1}h_n) =h_0^{-1}h_n$ and \eqref{eq:homogeneous} allow us to define
\[
 \mu_\C^n(a_n,\ldots,a_1)=\mu_\B^n(a_n,\ldots,a_1)
 \in\Hom_\C((X_0,h_0),(X_n,h_n))
\]
for $a_i\in\Hom_\C((X_{i-1},h_{i-1}),(X_i,h_i))$. These maps have degree $2-n$. The assumption on the units gives $1_X\in\Hom_\B^e(X,X)=\Hom_\C((X,h),(X,h))$, so we may set $1_{(X,h)}=1_X$. The maps $\pi_{(X,h),(Y,l)}$ are injective, send these elements to the units of $\B$, and commute with every restricted operation. By \cref{lem:faithful}, the operations and units make $\C$ an $A_\infty$-category, and $\pi$ is a strict $A_\infty$-functor. This proves \textup{(1)}.
\end{proof}

The lifted structure has the following uniqueness and orbit properties.

\begin{prop}
\label{prop:lift-orbit}
In the setting of \cref{prop:lift}, assume \eqref{eq:homogeneous}.
\begin{enumerate}[label=\textup{(\arabic*)}]
\item The $A_\infty$-structure on $\C$ for which $\pi$ is a strict $A_\infty$-functor is unique. Its operations are restrictions of those of $\B$, and its strict units are $1_{(X,h)}=1_X$.

\item The assignments $t(X,h)=(X,th)$ and
\[
 t:\Hom_\C((X,h),(Y,l))
 \longrightarrow\Hom_\C((X,th),(Y,tl)),
 \qquad a\longmapsto a,
\]
for $t\in G$, define a strict action of $G$ on $\C$ which is free on objects.

\item For the representatives $(X,e)$, $X\in\Ob\B$, there is a strict isomorphism $(\C/G)_{\mathrm{rep}}\cong\B$ sending $(X,e)$ to $X$.
\end{enumerate}
\end{prop}

\begin{proof}
\noindent
(1) Strictness of $\pi$ requires $\pi_{(X_0,h_0),(X_n,h_n)} \bigl(\mu_\C^n(a_n,\ldots,a_1)\bigr) =\mu_\B^n(a_n,\ldots,a_1)$. Since each morphism map of $\pi$ is the prescribed inclusion, every $\mu_\C^n$ must be the corresponding restriction of $\mu_\B^n$. Moreover, preservation of units requires $\pi_{(X,h),(X,h)}(1_{(X,h)})=1_X$, while the injectivity gives $1_{(X,h)}=1_X$. Thus the structure constructed in \cref{prop:lift} is unique.

\medskip\noindent
(2) For $t,h,l\in G$, the equality $(th)^{-1}(tl)=h^{-1}l$ identifies the source and target of the stated morphism map with the same graded space $\Hom_\B^{h^{-1}l}(X,Y)$. The assignments satisfy the group laws and preserve the restricted operations and units, so the action is strict. It is free on objects because $th=h$ implies $t=e$.

\medskip\noindent
(3) Every object $(X,h)$ lies in the orbit of $(X,e)$, and these are distinct orbits for distinct $X$. For the chosen representatives, we have
\[
\begin{aligned}
 \Hom_{(\C/G)_{\mathrm{rep}}}((X,e),(Y,e))
 &=\bigoplus_{g\in G}\Hom_\C((X,e),(Y,g))\\
 &=\bigoplus_{g\in G}\Hom_\B^g(X,Y)
 \cong\Hom_\B(X,Y).
\end{aligned}
\]
For $a_i\in\Hom_\B^{g_i}(X_{i-1},X_i)$, regarded as morphisms between representatives in $\C/G$, \eqref{eq:orbit} gives
\[
\begin{aligned}
 \mu_{\C/G}^n(a_n,\ldots,a_1)
 &=\mu_\C^n\bigl(
 (g_1\cdots g_{n-1})a_n,\ldots,g_1a_2,a_1
 \bigr)\\
 &=\mu_\B^n(a_n,\ldots,a_1).
\end{aligned}
\]
The second equality holds because the action changes only the source and target objects, while the operations of $\C$ are restrictions of those of $\B$. The canonical identifications preserve every operation and the strict units, proving the stated strict isomorphism.
\end{proof}

We conclude by showing that operations related by the orbit formula satisfy the Stasheff identities \eqref{eq:relations} simultaneously.

\begin{cor}\label{prop:reflection}
Let $\Ob\A$ be a set, with a graded $\Bbbk$-vector space $\Hom_\A(X,Y)$ for each $X,Y\in\Ob\A$, and $\Bbbk$-linear maps
\[
 \mu_\A^n:
 \Hom_\A(X_{n-1},X_n)\otimes\cdots\otimes\Hom_\A(X_0,X_1)
 \longrightarrow\Hom_\A(X_0,X_n)
\]
of degree $2-n$ for every $n\geq1$ and every sequence $X_0,\ldots,X_n\in\Ob\A$. Suppose that a group $G$ acts on $\Ob\A$ and by degree-zero $\Bbbk$-linear isomorphisms $g:\Hom_\A(X,Y)\longrightarrow\Hom_\A(gX,gY)$ such that
\[
 g\bigl(\mu_\A^n(a_n,\ldots,a_1)\bigr)
 =\mu_\A^n(ga_n,\ldots,ga_1)
\]
for every $g\in G$, $n\geq1$, and composable inputs $a_i\in\Hom_\A(X_{i-1},X_i)$. Give $\A/G$ the objects and graded morphism spaces of \cref{dfn:orbit}, and define its operations by \eqref{eq:orbit}. Then the following are equivalent.
\begin{enumerate}[label=\textup{(\arabic*)}]
\item The operations $\mu_\A^n$ satisfy the Stasheff identities \eqref{eq:relations}.
\item The operations $\mu_{\A/G}^n$ satisfy the Stasheff identities \eqref{eq:relations}.
\end{enumerate}
\end{cor}

\begin{proof}
$\underline{(1)\Rightarrow(2)}$. Note that the verification of \eqref{eq:relations} in the proof of \cref{prop:orbit-standard} uses only $G$-equivariance and \eqref{eq:relations} for the original operations, so it applies here.

\smallskip\noindent
$\underline{(2)\Rightarrow(1)}$. For $X,Y\in\Ob\A$, let $\iota_{X,Y}:\Hom_\A(X,Y)\hookrightarrow\Hom_{\A/G}(X,Y)$ be the inclusion into the $e$-summand, where $e$ is the identity of $G$. Taking $g_1=\cdots=g_n=e$ in \eqref{eq:orbit} gives
\[
 \mu_{\A/G}^n(\iota a_n,\ldots,\iota a_1)
 =\iota_{X_0,X_n}\bigl(\mu_\A^n(a_n,\ldots,a_1)\bigr)
\]
for homogeneous $a_i\in\Hom_\A(X_{i-1},X_i)$. As in the proof of \cref{lem:faithful}, applying $\iota_{X_0,X_n}$ to the left-hand side of \eqref{eq:relations} for $\A$ gives the left-hand side for $\A/G$ on $\iota a_n,\ldots,\iota a_1$, with the same signs. The latter vanishes by \textup{(2)}, and injectivity of $\iota_{X_0,X_n}$ proves \textup{(1)}.
\end{proof}

%%%%%%%%%%%%%%%%%%%%%%%%%%
\subsection{An example of gentle algebras}
\label{sec:gentle-cover}

We illustrate \cref{prop:lift} with a simple example involving gentle algebras. In the quivers below, a dashed curve joining two consecutive arrows indicates that their composition is zero. Let
\[
Q=\quad
% BEGIN BASE QUIVER
\begin{tikzpicture}[>=stealth,baseline=(current bounding box.center)]
 \node (q1) at (0,0) {$1$};
 \node (q0) at (2,0) {$0$};
 \node (q2) at (1,-1.4) {$2$};
 \draw[->] (q0) -- node[above] {$a$} (q1);
 \draw[->] (q0) to[bend right=50] node[above] {$d$} (q1);
 \draw[->] (q1) -- node[left] {$b$} (q2);
 \draw[->] (q2) -- node[right] {$c$} (q0);
 % The three dashed curves mark ba=0, ac=0, and cb=0.
 \draw[densely dashed] (0.58,0)
 .. controls (0.58,-0.19) and (0.47,-0.39)
 .. (0.3,-0.42);
 \draw[densely dashed] (1.7,-0.42)
 .. controls (1.53,-0.39) and (1.42,-0.19)
 .. (1.42,0);
 \draw[densely dashed] (0.7,-0.98)
 .. controls (0.88,-0.8) and (1.12,-0.8)
 .. (1.3,-0.98);
\end{tikzpicture}
% END BASE QUIVER
\qquad A=\Bbbk Q/(ba,cb,ac).
\]
Give $|c|=1$ and $|a|=|b|=|d|=0$. Let $\A$ have objects $0,1,2$ and morphism spaces $\Hom_\A(i,j)=e_jAe_i$. Put $\mu_\A^1=0$ and $\mu_\A^2(y,x)=(-1)^{|x|}yx$. The nonzero $\mu_\A^3$ on basis morphisms are given by
\[
\begin{aligned}
 \mu_\A^3(c,b,a)&=e_0,&
 \mu_\A^3(a,c,b)&=e_1,&
 \mu_\A^3(b,a,c)&=e_2,\\
 \mu_\A^3(dc,b,a)&=d,&
 \mu_\A^3(bdc,b,a)&=bd,\\
 \mu_\A^3(a,c,bd)&=d,&
 \mu_\A^3(a,c,bdc)&=-dc.
\end{aligned}
\]
Moreover, set $\mu_\A^n=0$ for $n\geq4$. By \cite[Proposition 3.1]{HKK17}, $\A$ is an $A_\infty$-category, with strict units $e_i$.

\smallskip\noindent
\emph{The infinite cyclic covering with $G=\mathbb{Z}$.} Define $\weight:Q_1\to\ZZ$ by
\[
 \weight(a)=\weight(b)=\weight(c)=0,\qquad \weight(d)=1.
\]
Extend $\weight$ to paths by addition, with $\weight(e_i)=0$. The covering quiver $\widetilde Q_{\ZZ}$ is given by:
\[
% BEGIN INFINITE QUIVER
\begin{tikzpicture}[>=stealth,baseline=(current bounding box.center)]
 \node (dotsL) at (-1.8,0) {$\cdots$};
 \node (L1) at (0,0) {$1_{s-1}$};
 \node (L0) at (2,0) {$0_{s-1}$};
 \node (L2) at (1,-1.4) {$2_{s-1}$};
 \node (M1) at (4,0) {$1_s$};
 \node (M0) at (6,0) {$0_s$};
 \node (M2) at (5,-1.4) {$2_s$};
 \node (R1) at (8,0) {$1_{s+1}$};
 \node (R0) at (10,0) {$0_{s+1}$};
 \node (R2) at (9,-1.4) {$2_{s+1}$};
 \node (dotsR) at (11.8,0) {$\cdots$};
 \draw[->] (L0) -- node[above] {$a_{s-1}$} (L1);
 \draw[->] (L1) -- node[left] {$b_{s-1}$} (L2);
 \draw[->] (L2) -- node[right] {$c_{s-1}$} (L0);
 \draw[->] (M0) -- node[above] {$a_s$} (M1);
 \draw[->] (M1) -- node[left] {$b_s$} (M2);
 \draw[->] (M2) -- node[right] {$c_s$} (M0);
 \draw[->] (R0) -- node[above] {$a_{s+1}$} (R1);
 \draw[->] (R1) -- node[left] {$b_{s+1}$} (R2);
 \draw[->] (R2) -- node[right] {$c_{s+1}$} (R0);
 \draw[->] (dotsL) -- node[above] {$d_{s-2}$} (L1);
 \draw[->] (L0) -- node[above] {$d_{s-1}$} (M1);
 \draw[->] (M0) -- node[above] {$d_s$} (R1);
 \draw[->] (R0) -- node[above] {$d_{s+1}$} (dotsR);
 % The three dashed curves mark ba=0, ac=0, and cb=0.
 \draw[densely dashed] (0.58,0)
 .. controls (0.58,-0.19) and (0.47,-0.39)
 .. (0.3,-0.42);
 \draw[densely dashed] (1.7,-0.42)
 .. controls (1.53,-0.39) and (1.42,-0.19)
 .. (1.42,0);
 \draw[densely dashed] (0.7,-0.98)
 .. controls (0.88,-0.8) and (1.12,-0.8)
 .. (1.3,-0.98);
 % The three dashed curves mark ba=0, ac=0, and cb=0.
 \draw[densely dashed] (4.58,0)
 .. controls (4.58,-0.19) and (4.47,-0.39)
 .. (4.3,-0.42);
 \draw[densely dashed] (5.7,-0.42)
 .. controls (5.53,-0.39) and (5.42,-0.19)
 .. (5.42,0);
 \draw[densely dashed] (4.7,-0.98)
 .. controls (4.88,-0.8) and (5.12,-0.8)
 .. (5.3,-0.98);
 % The three dashed curves mark ba=0, ac=0, and cb=0.
 \draw[densely dashed] (8.58,0)
 .. controls (8.58,-0.19) and (8.47,-0.39)
 .. (8.3,-0.42);
 \draw[densely dashed] (9.7,-0.42)
 .. controls (9.53,-0.39) and (9.42,-0.19)
 .. (9.42,0);
 \draw[densely dashed] (8.7,-0.98)
 .. controls (8.88,-0.8) and (9.12,-0.8)
 .. (9.3,-0.98);
\end{tikzpicture}
% END INFINITE QUIVER
\]
with $s\in\ZZ$. Let $\widetilde\A_{\ZZ}$ have these vertices as objects and morphism spaces spanned by paths modulo the relations
\[
 b_sa_s=c_sb_s=a_sc_s=0\qquad(s\in\ZZ),
\]
with degrees inherited from $Q$. The projection on objects and arrows is
\[
 \pi_{\ZZ}:\widetilde\A_{\ZZ}\longrightarrow\A,
 \qquad i_s\longmapsto i,\qquad x_s\longmapsto x
 \quad(i\in \{0,1,2\},\; x\in\{a,b,c,d\}),
\]
and it extends to paths by concatenation. For $t\in\ZZ$, translation sends $i_s$ to $i_{s+t}$ and $x_s$ to $x_{s+t}$.

Both $\mu_\A^2$ and $\mu_\A^3$ preserve weight, so \eqref{eq:homogeneous} holds. By \cref{prop:lift}, $\widetilde\A_{\ZZ}$ admits an $A_\infty$-structure for which $\pi_{\ZZ}$ is strict. Write $\widetilde\mu^n$ for the lifted operations. The nonzero values of $\widetilde\mu^3$ on triples of basis paths are
\[
\begin{aligned}
 \widetilde\mu^3(c_s,b_s,a_s)&=e_{0_s},&
 \widetilde\mu^3(a_s,c_s,b_s)&=e_{1_s},&
 \widetilde\mu^3(b_s,a_s,c_s)&=e_{2_s},\\
 \widetilde\mu^3(d_sc_s,b_s,a_s)&=d_s,&
 \widetilde\mu^3(b_{s+1}d_sc_s,b_s,a_s)&=b_{s+1}d_s,\\
 \widetilde\mu^3(a_{s+1},c_{s+1},b_{s+1}d_s)&=d_s,&
 \widetilde\mu^3(a_{s+1},c_{s+1},b_{s+1}d_sc_s)&=-d_sc_s.
\end{aligned}
\]
Here $s\in\ZZ$. By \cref{prop:lift-orbit}, the lifted structure is unique, translation is a strict action, and $(\widetilde\A_{\ZZ}/\ZZ)_{\mathrm{rep}}\cong\A$.

\smallskip\noindent
\emph{The double covering  with $G=\mathbb{Z}/2\mathbb{Z}$.} Reduce the weights and all subscripts above modulo $2$. The quiver $\widetilde Q_2$ is
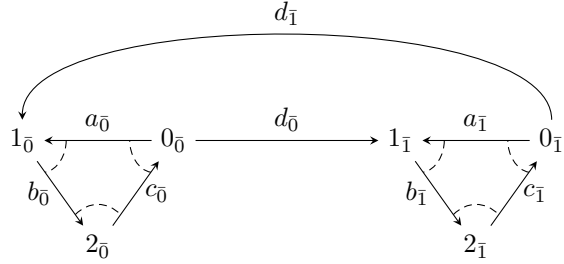
\[
% BEGIN DOUBLE QUIVER
\begin{tikzpicture}[>=stealth,baseline=(current bounding box.center)]
 \node (L1) at (0,0) {$1_{\bar0}$};
 \node (L0) at (2,0) {$0_{\bar0}$};
 \node (L2) at (1,-1.4) {$2_{\bar0}$};
 \node (R1) at (5,0) {$1_{\bar1}$};
 \node (R0) at (7,0) {$0_{\bar1}$};
 \node (R2) at (6,-1.4) {$2_{\bar1}$};
 \draw[->] (L0) -- node[above] {$a_{\bar0}$} (L1);
 \draw[->] (L1) -- node[left] {$b_{\bar0}$} (L2);
 \draw[->] (L2) -- node[right] {$c_{\bar0}$} (L0);
 \draw[->] (R0) -- node[above] {$a_{\bar1}$} (R1);
 \draw[->] (R1) -- node[left] {$b_{\bar1}$} (R2);
 \draw[->] (R2) -- node[right] {$c_{\bar1}$} (R0);
 \draw[->] (L0) -- node[above] {$d_{\bar0}$} (R1);
 \draw[->] (R0.north) .. controls +(0,1.5) and +(0,1.5)
 .. node[above] {$d_{\bar1}$} (L1.north);
 % The three dashed curves mark ba=0, ac=0, and cb=0.
 \draw[densely dashed] (0.58,0)
 .. controls (0.58,-0.19) and (0.47,-0.39)
 .. (0.3,-0.42);
 \draw[densely dashed] (1.7,-0.42)
 .. controls (1.53,-0.39) and (1.42,-0.19)
 .. (1.42,0);
 \draw[densely dashed] (0.7,-0.98)
 .. controls (0.88,-0.8) and (1.12,-0.8)
 .. (1.3,-0.98);
 % The three dashed curves mark ba=0, ac=0, and cb=0.
 \draw[densely dashed] (5.58,0)
 .. controls (5.58,-0.19) and (5.47,-0.39)
 .. (5.3,-0.42);
 \draw[densely dashed] (6.7,-0.42)
 .. controls (6.53,-0.39) and (6.42,-0.19)
 .. (6.42,0);
 \draw[densely dashed] (5.7,-0.98)
 .. controls (5.88,-0.8) and (6.12,-0.8)
 .. (6.3,-0.98);
\end{tikzpicture}
% END DOUBLE QUIVER
\]
Put $ \widetilde A_2=\Bbbk\widetilde Q_2/ (b_{\bar s}a_{\bar s},c_{\bar s}b_{\bar s}, a_{\bar s}c_{\bar s}\mid\bar s\in\ZZ/2\ZZ)$, and let $\widetilde\A_2$ be its associated category. The projection
\[
 \pi_2:\widetilde\A_2\longrightarrow\A,
 \qquad i_{\bar s}\longmapsto i,\qquad
 x_{\bar s}\longmapsto x
\]
is strict for the $A_\infty$-structure obtained by reducing the subscripts in the formulas for $\widetilde\mu^3$ modulo $2$. Moreover, \cref{prop:lift-orbit} gives $(\widetilde\A_2/(\ZZ/2\ZZ))_{\mathrm{rep}}\cong\A.$

%%%%%%%%%%%%%%%%%%%%%%%%%%%%%%%%
\subsection{\texorpdfstring{$r$}{r}-fold trivial extensions as lifts}
\label{sec:trivial-lift-comparison}

We apply \cref{prop:lift} to trivial extensions of $A_\infty$-categories. The method is to assign weight $0$ to $\Hom_\A(X,Y)$ and weight $1$ to $D\Hom_\A(Y,X)$, and then apply \cref{prop:lift}. These integer weights give a $\ZZ$-covering of $T(\A)$, and their reductions modulo $r$ give the $r$-fold trivial extension $T_r(\A)$.

\begin{dfn}
\textnormal{(cf.~\cite[Definition 2.10]{OZ22})}
\label{dfn:trivial-extension}
Let $\A$ be an $A_\infty$-category. Its \defn{trivial extension} $T(\A)$ has the same objects as $\A$, with $\Hom_{T(\A)}(X,Y) =\Hom_\A(X,Y)\oplus D\Hom_\A(Y,X)$. For $n\geq1$ and $X_0,\ldots,X_n\in\Ob\A$, the operations are determined by the following rules.
\begin{enumerate}
\item For homogeneous $a_j\in\Hom_\A(X_{j-1},X_j)$, $1\leq j\leq n$, $\mu_{T(\A)}^n(a_n,\ldots,a_1) =\mu_\A^n(a_n,\ldots,a_1).$
\item For $1\leq i\leq n$ and homogeneous $f\in D\Hom_\A(X_i,X_{i-1})$, the operation takes values in $D\Hom_\A(X_n,X_0)$ and is given by
\[
\begin{aligned}
 &\mu_{T(\A)}^n(a_n,\ldots,a_{i+1},f,a_{i-1},\ldots,a_1)(x)\\
 &\qquad=(-1)^\dagger
 f\bigl(\mu_\A^n(a_{i-1},\ldots,a_1,x,a_n,\ldots,a_{i+1})\bigr),
\end{aligned}
\]
for homogeneous $x\in\Hom_\A(X_n,X_0)$, where $\dagger=\rdeg{f} +\sum_{\substack{1\leq j\leq n\\j\ne i}}\rdeg{a_j}.$
\item For homogeneous $b_j\in\Hom_{T(\A)}(X_{j-1},X_j)$, $\mu_{T(\A)}^n(b_n,\ldots,b_1)=0$ whenever at least two of the $b_j$ belong to the respective dual summands $D\Hom_\A(X_j,X_{j-1})$.
\end{enumerate}
The strict unit at $X$ is $(1_X,0)$.
\end{dfn}

For $X,Y\in\Ob\A$, assign group weight $0$ to $\Hom_\A(X,Y)$ and group weight $1$ to $D\Hom_\A(Y,X)$, independently of the cohomological degrees.

\begin{lem}\label{lem:dual-weight}
Let $G=\ZZ$ or $G=\ZZ/r\ZZ$ for some $r\geq1$. The above weights, viewed as elements of $G$, satisfy \eqref{eq:homogeneous}. The units have weight $0$.
\end{lem}

\begin{proof}
By \cref{dfn:trivial-extension}, \eqref{eq:homogeneous} holds in both cases. The units $(1_X,0)$ have weight $0$.
\end{proof}

\begin{thm}\label{thm:rfold}
Let $r\geq1$. There is a unique $A_\infty$-structure on the object set $\Ob\A\times\ZZ/r\ZZ$ and graded morphism spaces
\begin{equation}\label{eq:rfold-hom}
 \Hom_{T_r(\A)}((X,\bar i),(Y,\bar j))
 =\begin{cases}
 \Hom_\A(X,Y)\oplus D\Hom_\A(Y,X),&r=1,\\
 \Hom_\A(X,Y),&r\geq2,\ \bar j=\bar i,\\
 D\Hom_\A(Y,X),&r\geq2,\ \bar j=\bar i+\bar1,\\
 0,&\text{otherwise}
 \end{cases}
\end{equation}
for which the map
\[
 \pi_r:T_r(\A)\longrightarrow T(\A),
 \qquad (X,\bar i)\longmapsto X,
\]
given on morphism spaces by the inclusions into $\Hom_{T(\A)}(X,Y)$, is a strict $A_\infty$-functor. The operations are
\[
 \mu_{T_r(\A)}^n(b_n,\ldots,b_1)
 =\mu_{T(\A)}^n(b_n,\ldots,b_1)
\]
for composable morphisms in $T_r(\A)$, and the unit at $(X,\bar i)$ is the copy of $1_X$ in $\Hom_\A(X,X)$.

We call $T_r(\A)$ the \defn{$r$-fold trivial extension} of $\A$. For $r=1$, both summands occur in every morphism space, and we identify $T_1(\A)$ with $T(\A)$ via $(X,\bar0)\mapsto X$.
\end{thm}

\begin{proof}
By \cref{lem:dual-weight}, the weights modulo $r$ satisfy \eqref{eq:homogeneous}, with units of weight $\bar0$. The existence of the stated structure follows from \cref{prop:lift}; its uniqueness and the formulas for its operations and units follow from \cref{prop:lift-orbit}.
\end{proof}

\begin{cor}\label{cor:repetitive}
The $\ZZ$-grading in \cref{lem:dual-weight} gives a lift $\widehat\A$ of $T(\A)$, called the \defn{repetitive category} of $\A$. Its objects are $(X,i)\in\Ob\A\times\ZZ$, and its morphism spaces are
\[
 \Hom_{\widehat\A}((X,i),(Y,j))
 =\begin{cases}
 \Hom_\A(X,Y),&j=i,\\
 D\Hom_\A(Y,X),&j=i+1,\\
 0,&\text{otherwise}.
 \end{cases}
\]
Its operations are restrictions of those of $T(\A)$. The map
\[
 \nu:\widehat\A\longrightarrow\widehat\A,
 \qquad (X,i)\longmapsto(X,i+1),
\]
acting as the identity on each copy of $\Hom_\A(X,Y)$ and $D\Hom_\A(Y,X)$, is a strict automorphism. For every $r\geq1$, there is a commutative diagram of strict $A_\infty$-functors
\[
\begin{tikzcd}[column sep=large,row sep=large]
 \widehat\A\arrow[r,"q_r"]\arrow[dr,"\pi_{\ZZ}"']
 &T_r(\A)\arrow[d,"\pi_r"]\\
 &T(\A),
\end{tikzcd}
\]
where $q_r(X,i)=(X,\bar i)$ and $\pi_{\ZZ}(X,i)=X$; both functors act by inclusions on morphism spaces. Moreover, translation $\bar t(X,\bar i)=(X,\bar i+\bar t)$, acting identically on morphism spaces, defines a strict action of $\ZZ/r\ZZ$ on $T_r(\A)$, and there are strict isomorphisms
\[
 (\widehat\A/\langle\nu^r\rangle)_{\mathrm{rep}}
 \cong T_r(\A),
 \qquad
 (T_r(\A)/(\ZZ/r\ZZ))_{\mathrm{rep}}
 \cong T(\A).
\]
\end{cor}

\begin{proof}
Apply \cref{prop:lift,prop:lift-orbit} with $G=\ZZ$ and $G=\ZZ/r\ZZ$, using \cref{lem:dual-weight}. The strictness of $q_r$, the commutativity of the diagram, and the two orbit isomorphisms follow directly from the defining formulas and \eqref{eq:orbit}.
\end{proof}

If $\mu_\A^n=0$ for all $n\ne2$, the defining formulas give the same vanishing for $T(\A)$, $T_r(\A)$, and $\widehat\A$. In the finite-dimensional degree-zero case, we recover the classical constructions; see \cite{HW83,Asa97}.

\begin{cor}
\label{cor:rfold-classical}
Suppose that $\A$ has a finite nonempty object set, finite-dimensional morphism spaces concentrated in degree zero, and $\mu_\A^n=0$ for all $n\ne2$. Let $A$ be its associated algebra. Then $\widehat\A$ is the classical repetitive category of $A$. For every $r\geq1$, the algebra associated with $T_r(\A)$ is the classical $r$-fold trivial extension $T_r(A)$.
\end{cor}

The connection between gentle algebras and partially wrapped Fukaya categories of surfaces \cite{HKK17} provides a geometric approach to their $A_\infty$-structures. Building on the identification of trivial extensions of gentle algebras with Brauer graph algebras \cite{Sch15}, Opper and Zvonareva develop an $A_\infty$-analogue of this construction \cite{OZ22}. In \cite{Xin26}, we further realize $r$-fold trivial extensions of gentle algebras as admissible fractional Brauer graph algebras. These connections motivate us to construct $A_\infty$-categories associated with admissible fractional Brauer graph algebras and to study their geometric models.

%%%%%%%%%%%%%%%%%%%%%%%%%%%%%%%%%%%%%%%%%%%%%%%%%%%%%%%%%%%%%%%%%%%%%%%%%%%%%%%%%%%%%%%%%%%%%

\section{Brauer graph algebras and their \texorpdfstring{$A_\infty$}{A-infinity}-categories}
\label{sec:algebra-surface-preliminaries}
\raggedbottom

We recall Brauer graph algebras and the Brauer graph $A_\infty$-categories of \cite{OZ22}. 

\subsection{Brauer graphs and their algebras}

We use the half-edge description of ribbon graphs in \cite[Section 1.1]{OZ22}; see also \cite{Sch18,OPS18}.

\begin{dfn}\label{dfn:ribbon}
A \defn{ribbon graph} is a tuple $\Gamma=(V,H,s,\iota,\rho)$ consisting of the following data.
\begin{itemize}
\item $V$ and $H$ are finite nonempty sets, whose elements are called \defn{vertices} and \defn{half-edges}, respectively.
\item $s:H\to V$ is a surjection. We write $H_{\vv}=s^{-1}(\vv)$ and $\operatorname{val}_\Gamma(\vv)=|H_{\vv}|$.
\item $\iota:H\to H$ is an involution without fixed points. An \defn{edge} is a pair $e(h)=\{h,\iota h\}$, and $E(\Gamma)=H/\langle\iota\rangle$ is the set of edges. The endpoints of $e(h)$ are $s(h)$ and $s(\iota h)$.
\item $\rho:H\to H$ is a permutation whose cycles are exactly the sets $H_{\vv}$, $\vv\in V$. Thus $\rho h$ is the successor of $h$ in the cyclic order at $s(h)$.
\end{itemize}
\end{dfn}

Note that loops and multiple edges are allowed; a loop contributes two half-edges to the valency of its vertex. We write $V(\Gamma)=V$ and $H(\Gamma)=H$, and assume that the underlying graph is connected unless stated otherwise. A morphism of ribbon graphs $\varphi:\Gamma\to\Gamma'$ consists of maps $\varphi_V:V\to V'$ and $\varphi_H:H\to H'$ satisfying
\[
 s'\varphi_H=\varphi_Vs,\qquad
 \iota'\varphi_H=\varphi_H\iota,\qquad
 \rho'\varphi_H=\varphi_H\rho.
\]
The graph is \defn{bipartite} if there is a partition $V=V_0\sqcup V_1$ such that each edge has one endpoint in $V_0$ and one in $V_1$.

\begin{dfn}\label{dfn:brauer-graph}
A \defn{Brauer graph} is a pair $(\Gamma,\m)$, where $\Gamma$ is a ribbon graph and $\m:V(\Gamma)\to\ZZ_{>0}$ is its \defn{multiplicity function}.
\end{dfn}

Let $(\Gamma,\m)$ be a Brauer graph. Following \cite[Section 1.2]{OZ22}, form the quiver $Q_\Gamma$ with vertices $E(\Gamma)$ and one arrow
\[
 a_h:e(h)\longrightarrow e(\rho h)
 \qquad\text{for every }h\in H(\Gamma).
\]
For $h\in H(\Gamma)$ and $t\geq1$, write
\[
 p_{h,t}=a_{\rho^{t-1}h}\cdots a_{\rho h}a_h,
 \qquad p_{h,0}=e_{e(h)}.
\]
Here $e_{e(h)}$ is the trivial path at the quiver vertex $e(h)$. The \defn{special cycle} based at $h$ is $C_h=p_{h,\operatorname{val}_\Gamma(s(h))}$; it starts and ends at $e(h)$.

Let $I_{\Gamma,\m}$ be the ideal of $\Bbbk Q_\Gamma$ generated by the relations
\begin{enumerate}
\item $C_h^{\m(s(h))}=C_{\iota h}^{\m(s(\iota h))}$ for each $h\in H(\Gamma)$;
\item $a_k a_h=0$ for all $e(k)=e(\rho h),\ k\ne\rho h$.
\end{enumerate}
Thus a nonzero product of arrows follows the cyclic order at a single vertex, and the two maximal cycles based at an edge are identified. Set
\[
 B(\Gamma,\m)=\Bbbk Q_\Gamma/I_{\Gamma,\m}.
\]

\begin{dfn}\label{dfn:bga}
A $\Bbbk$-algebra is a \defn{Brauer graph algebra} (abbr. BGA) if it is isomorphic to $B(\Gamma,\m)$ for some Brauer graph $(\Gamma,\m)$.
\end{dfn}

If $\m(s(h))=\operatorname{val}_\Gamma(s(h))=1$, our presentation retains a loop $a_h$, which the cycle relation identifies with $C_{\iota h}^{\m(s(\iota h))}$. Eliminating these redundant loops gives the presentation in \cite[Sections 2.3--2.4]{Sch18}. For the graph consisting of one edge with two distinct endpoints, both of multiplicity one, one loop is retained and the algebra is $\Bbbk[x]/(x^2)$.

We recall the following properties; see \cite[Section 2]{Sch18}.
\begin{itemize}
\item $B(\Gamma,\m)$ is finite-dimensional and symmetric, that is, $B(\Gamma,\m)\cong DB(\Gamma,\m)$ as bimodules. In particular, it is self-injective.
\item $B(\Gamma,\m)$ is special biserial. Its Gabriel quiver has at most two arrows starting and at most two arrows ending at each vertex, and every arrow has at most one nonzero continuation on either side.
\item For $e=e(h)$, the indecomposable projective right module $e_eB(\Gamma,\m)$ has one-dimensional socle, generated by the common class of $C_h^{\m(s(h))}$ and $C_{\iota h}^{\m(s(\iota h))}$.
\end{itemize}

\subsection{Graded surfaces and arc systems}
\label{sec:surface-definitions}

The cyclic orders of a ribbon graph determine an oriented surface. Replacing each vertex by a small disc and each edge by a band gives its \defn{ribbon surface} $\Sigma_\Gamma$, where the bands are attached so that $\rho$ gives the clockwise order of the half-edges; see \cite[Section 2.7]{Sch18} and \cite[Lemma 1.2]{OZ22}. The graph embeds in $\operatorname{int}\Sigma_\Gamma$, and
\[
 \Sigma_\Gamma\setminus\Gamma
 \cong\partial\Sigma_\Gamma\times[0,1).
\]
An embedding with this property is called \defn{filling}.

The \defn{faces} of $\Gamma$ are the cycles of $\rho\iota$; write $F(\Gamma)$ for their set. They correspond to the boundary components of $\Sigma_\Gamma$. The \defn{perimeter} of a face is the length of its cycle, counting repeated edge occurrences separately. If $g(\Sigma_\Gamma)$ is the genus, then
\[
 |V(\Gamma)|-|E(\Gamma)|
 =2-2g(\Sigma_\Gamma)-|F(\Gamma)|.
\]

We follow the geometric constructions of \cite[Sections 3.1--3.3]{OZ22}.

\begin{dfn}\label{dfn:punctured-surface}
A \defn{punctured surface} $(\Sigma,\mathcal P)$ consists of a compact connected oriented smooth surface $\Sigma$ with nonempty boundary and a finite nonempty set $\mathcal P\subset\operatorname{int}\Sigma$ of \defn{punctures}.
\end{dfn}

\begin{dfn}
\textnormal{(cf.~\cite[Definitions 3.2 and 3.4]{OZ22})}
\label{dfn:surface-arc}
Let $(\Sigma,\mathcal P)$ be a punctured surface. An \defn{arc} in this paper is a smooth immersion $\gamma:[0,1]\to\Sigma$ such that
\begin{itemize}
\item $\gamma(0),\gamma(1)\in\mathcal P$ and $\gamma((0,1))\subset\operatorname{int}\Sigma\setminus\mathcal P$;
\item $\gamma$ is injective on $[0,1)$ and on $(0,1]$;
\item if $\gamma(0)=\gamma(1)$, its image does not bound a disc in $\Sigma$ whose interior contains no punctures.
\end{itemize}
The endpoints may coincide. The last condition excludes an empty monogon. Note that a loop surrounding a boundary component is allowed.

A \defn{regular homotopy} between arcs $\gamma_0$ and $\gamma_1$ is a smooth map $H:[0,1]\times[0,1]\to\Sigma$ such that, writing $H_s=H(s,-)$, we have $H_0=\gamma_0$ and $H_1=\gamma_1$, and, for every $s\in[0,1]$, the map $H_s$ is an immersion satisfying
\[
 H_s(0)=\gamma_0(0),\qquad
 H_s(1)=\gamma_0(1),\qquad
 H_s((0,1))\subset
 \operatorname{int}\Sigma\setminus\mathcal P.
\]
Throughout, homotopies between arcs mean regular homotopies relative to their endpoints.
\end{dfn}

\begin{dfn}\label{dfn:surface-arcs}
An \defn{arc system} $\U$ on $(\Sigma,\mathcal P)$ is a finite collection of pairwise nonhomotopic arcs, considered without orientation, whose interiors are disjoint. At each puncture, the incident branches have distinct outward tangent rays, and every puncture is incident to at least one arc. Write $|\U|$ for the union of their images.
\begin{itemize}
\item The arc system $\U$ is \defn{full} if cutting $\Sigma$ along its arcs gives a disjoint union of discs and annuli, with each annulus containing exactly one component of $\partial\Sigma$.
\item A full arc system is \defn{admissible} if, for every $\vv\in \mathcal P$, there is an embedded path
\[
 c_{\vv}:[0,1]\longrightarrow(\Sigma\setminus|\U|)\cup\{\vv\},
 \qquad c_{\vv}(0)=\vv,\quad c_{\vv}(1)\in\partial\Sigma.
\]
Such a path is called a \defn{cutting path} at $\vv$. We choose $c_{\vv}$ with its interior in $\operatorname{int}\Sigma\setminus|\U|$. A \defn{cut} is a collection $(c_{\vv})_{\vv\in \mathcal P}$ of pairwise disjoint cutting paths.
\end{itemize}
\end{dfn}

The arcs of $\U$ form a ribbon graph $\Gamma_\U$ with vertex set $\mathcal P$, edge set $\U$, and half-edges given by endpoint occurrences. Its cyclic orders are the clockwise orders at the punctures.

At a fixed point $\xx\in\Sigma\setminus\mathcal P$, the tangent plane $T_{\xx}\Sigma$ is a two-dimensional real vector space. Its projectivization identifies all nonzero vectors spanning the same line:
\[
 \mathbb P(T_{\xx}\Sigma)
 =\bigl(T_{\xx}\Sigma\setminus\{0\}\bigr)/\sim,
 \qquad
 \vec{v}\sim \vec{w}\ \Longleftrightarrow\
 \vec{w}=\lambda \vec{v}\text{ for some }\lambda\in\mathbb R\setminus\{0\}.
\]
Thus the point represented by $\vec{v}$ is the unoriented line $\mathbb R \vec{v}=\{a \vec{v}:a\in\mathbb R\}$. To visualize the quotient, choose a norm on $T_{\xx}\Sigma$ and first replace $\vec{v}$ by $\vec{v}/\lVert \vec{v}\rVert$. Each line then meets the unit circle in two opposite points. Identifying each opposite pair gives another circle, namely $\mathbb P(T_{\xx}\Sigma)$; see \cref{fig:projectivization}. Note that the norm only helps draw the quotient and is not part of its definition. For the underlying vector-bundle constructions and real projective spaces, see \cite[Sections 2--3]{MS74}.

Projectivizing the tangent bundle performs this construction separately at every point of the surface:
\[
 \mathbb P(T\Sigma)
 =\{(\xx,L):\xx\in\Sigma,\ L\subset T_{\xx}\Sigma
       \text{ is a one-dimensional real subspace}\}.
\]
The projection $(\xx,L)\mapsto \xx$ has the circle $\mathbb P(T_{\xx}\Sigma)$ as its \emph{fibre} over $\xx$. Over a coordinate neighbourhood $U\subset\Sigma$, this family is identified with $U\times S^1$, a three-dimensional space. Over a subinterval of a parametrized arc, the corresponding family of circles can be drawn as a cylinder; see \cref{fig:arc-grading}. A single tangent plane contributes one circle to this family.

\begin{figure}[htbp]
\centering
\begin{tikzpicture}[>=Stealth,scale=.96]
\begin{scope}[shift={(-4.5,0)}]
 \filldraw[fill=gray!6,draw=gray!45] (-1.55,-1.05) rectangle (1.55,1.25);
 \draw[gray!60] (-1.4,0) -- (1.4,0);
 \draw[gray!60] (0,-.9) -- (0,1.1);
 \draw[thick,darkblue] (30:-1.5) -- (30:1.5);
 \draw[->,darkblue,thick] (0,0) -- (30:.7);
 \draw[->,darkblue,thick] (0,0) -- (210:.7);
 \fill[darkblue] (30:.7) circle (1.6pt) node[above left] {$\vec{v}$};
 \fill[darkblue] (30:1.4) circle (1.6pt) node[above] {$2\vec{v}$};
 \fill[darkblue] (210:.7) circle (1.6pt) node[below] {$-\vec{v}$};
 \filldraw[fill=white,draw=black] (0,0) circle (1.6pt);
 \node[below right] at (0,0) {$0$};
 \node at (0,1.63) {$T_{\xx}\Sigma\setminus\{0\}$};
 \node at (0,-1.6) {(a) Nonzero vectors};
\end{scope}
\draw[->,thick] (-2.65,.2) -- (-1.55,.2);
\node[text width=1.6cm,align=center,font=\small] at (-2.1,.8) {normalize};
\begin{scope}
 \draw[gray!65] (0,0) circle (1);
 \draw[dashed,darkblue] (30:1) -- (210:1);
 \fill[darkblue] (30:1) circle (2pt)
   node[above right,inner sep=2pt] {$\dfrac{\vec{v}}{\lVert \vec{v}\rVert}$};
 \fill[darkblue] (210:1) circle (2pt)
   node[below left,inner sep=2pt] {$-\dfrac{\vec{v}}{\lVert \vec{v}\rVert}$};
 \node at (0,1.63) {unit circle};
 \node at (0,-1.6) {(b) Two opposite points};
\end{scope}
\draw[->,thick] (1.65,.2) -- (2.95,.2);
\node[text width=1.9cm,align=center,font=\small] at (2.3,.9) {identify opposite points};
\begin{scope}[shift={(4.5,0)}]
 \draw[gray!65] (0,0) circle (1);
 \fill[darkblue] (60:1) circle (2.5pt)
   node[right,inner sep=3pt] {$\mathbb R \vec{v}$};
 \node at (0,1.63) {$\mathbb P(T_{\xx}\Sigma)$};
 \node at (0,-1.6) {(c) One unoriented line};
\end{scope}
\end{tikzpicture}
\caption{Projectivizing a single tangent plane.}
\label{fig:projectivization}
\end{figure}
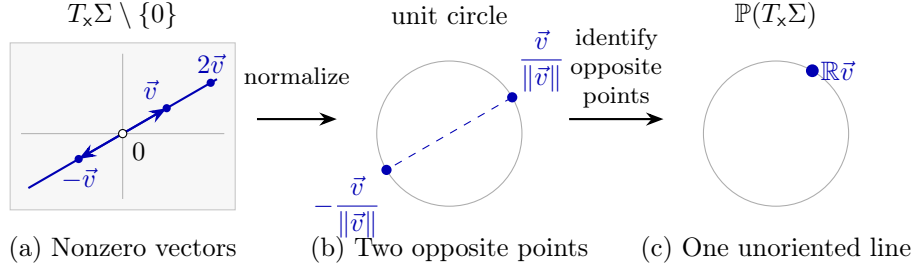

\begin{dfn}\label{dfn:line-field}
A \defn{line field} on $(\Sigma,\mathcal P)$ assigns a line $\eta(\xx)\subset T_{\xx}\Sigma$ to every $\xx\in\Sigma\setminus\mathcal P$, continuously in $\xx$. Equivalently, it is a continuous section
\[
 \eta:\Sigma\setminus \mathcal P\longrightarrow
 \mathbb P(T\Sigma)|_{\Sigma\setminus \mathcal P},
\]
meaning that $\eta$ chooses one point on each circle $\mathbb P(T_{\xx}\Sigma)$, with the chosen point varying continuously; see \cref{fig:line-field-fibres}. A homotopy of line fields is a homotopy through sections. The triple $(\Sigma,\mathcal P,\eta)$ is called a \defn{graded punctured surface}. A line field is \defn{orientable} if there is a continuous vector field $\vec{v}$ on $\Sigma\setminus\mathcal P$ such that
\[
 \vec{v}(\xx)\in T_{\xx}\Sigma\setminus\{0\},\qquad
 \eta(\xx)=\mathbb R \vec{v}(\xx)
 \quad(\xx\in\Sigma\setminus\mathcal P).
\]
Thus orientability means that one can continuously choose a direction on every line of $\eta$. Here $\mathbb R \vec{v}(\xx)=\{a \vec{v}(\xx):a\in\mathbb R\}$ denotes the real line spanned by $\vec{v}(\xx)$.
\end{dfn}

\begin{figure}[htbp]
\centering
\begin{tikzpicture}[>=Stealth,scale=.94]
% Unoriented lines in one tangent plane.
\begin{scope}[shift={(-4.8,0)}]
 \filldraw[fill=gray!7,draw=gray!55]
   (-1.45,-.9) -- (.85,-.9) -- (1.45,.9) -- (-.85,.9) -- cycle;
 \draw[gray!70] (-1.05,0) -- (1.05,0);
 \draw[gray!70] (85:-.75) -- (85:.75);
 \draw[thick,darkblue] (24.5:-1.12) -- (24.5:1.12);
 \fill (0,0) circle (1.5pt) node[below right] {$0$};
 \node[darkblue,above] at (.7,.4) {$\eta(\xx)$};
 \node at (0,1.2) {$T_{\xx}\Sigma$};
 \node at (0,-1.75) {(a) Tangent lines};
\end{scope}
% One fibre, viewed as a circle of directions.
\begin{scope}[shift={(-.7,0)}]
 \draw[gray!65] (0,0) circle (.94);
 \fill[gray!75] (0:.94) circle (1.9pt);
 \fill[gray!75] (170:.94) circle (1.9pt);
 \fill[darkblue] (49:.94) circle (2.6pt)
   node[above right,inner sep=2pt] {$\eta(\xx)$};
 \node at (0,1.5) {$\mathbb P(T_{\xx}\Sigma)$};
 \node at (0,-1.75) {(b) The direction circle};
\end{scope}
\draw[->,gray!75] (-3.03,0) -- (-2.25,0);
% A few fibres over a path in a local patch of the surface.
\begin{scope}[shift={(4.25,0)}]
 \filldraw[fill=gray!8,draw=gray!60]
   (-1.8,-1.1) -- (1.45,-1.1) -- (2,-.05) -- (-1.25,-.05) -- cycle;
 \draw[gray!65] (-1.35,-.58) .. controls (-.4,-.47) and (.4,-.7) .. (1.5,-.55);
 \foreach \figx/\angle in {-1.2/110,0/49,1.2/10} {
   \draw[densely dotted,gray!65] (\figx,-.57) -- (\figx,.9);
   \draw[gray!65] (\figx,.9) circle (.5);
   \fill (\figx,-.57) circle (1.4pt);
   \fill[darkblue] (\figx,.9) ++(\angle:.5) circle (2.4pt);
 }
 \draw[darkblue,thick] (-1.37101,1.36985)
   .. controls (-.9,1.43) and (0,1.41) .. (.32803,1.27735)
   .. controls (.7,1.13) and (1.2,.96) .. (1.6924,.98682);
 \node[below] at (0,-.57) {$\xx$};
 \node[darkblue,above] at (.328,1.48) {$\eta(\xx)$};
 \node[right] at (1.87,-.22) {$\Sigma$};
 \node at (0,-1.75) {(c) A continuous choice};
\end{scope}
\end{tikzpicture}
\caption{A line field as a choice of direction at every point. In (a), opposite vectors span the same line; each such line is one point on the circle in (b). The blue line and blue point represent the same choice $\eta(\xx)$. In (c), three fibres over a path in the surface are drawn, with blue points illustrating a continuous choice along that path.}
\label{fig:line-field-fibres}
\end{figure}
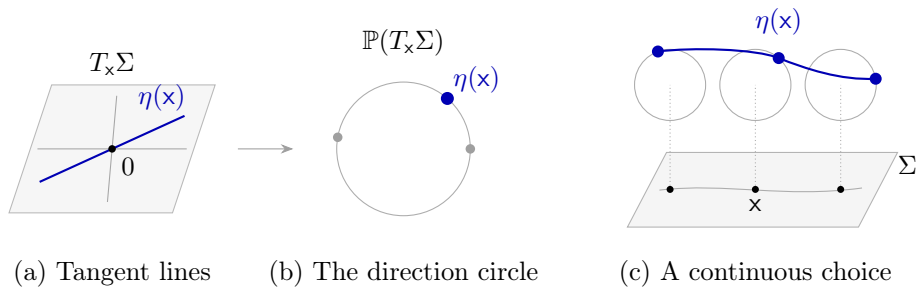

The construction of \cite[Example 3.7]{OZ22} equips $\Sigma_\Gamma\setminus V(\Gamma)$ with a line field $\eta_\Gamma$ whose homotopy class is determined by $\Gamma$; see \cref{fig:canonical-line-field}. It proceeds as follows.
\begin{enumerate}[label=\textup{(\arabic*)}]
\item Cut each edge band along an embedded path that joins its two boundary sides and crosses the edge once transversely. The paths are disjoint, and each resulting region contains exactly one vertex.
\item On each region, choose the local line field in \cite[Example 3.7]{OZ22}. Near the vertex it is tangent to circles surrounding the vertex; along the cut intervals it is tangent to those intervals; and along the graph edges it is transverse to the edges.
\item Glue the paired cut intervals. Their tangent lines are identified by the gluing, so the local fields agree there and define the line field $\eta_\Gamma$.
\end{enumerate}
The curves drawn in \cref{fig:canonical-line-field} indicate the line field by their tangent lines. They carry no orientation.

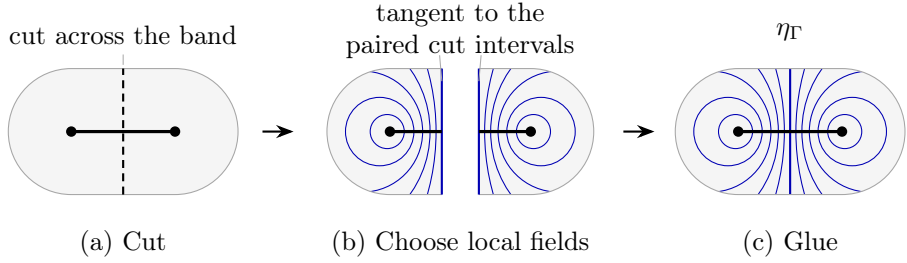
\begin{figure}[htbp]
\centering
\begin{tikzpicture}[>=Stealth,scale=.98]
% A single edge and its ribbon neighbourhood.
\begin{scope}[shift={(-4.6,0)}]
 \fill[gray!8] (-.7,.85) -- (.7,.85)
   arc[start angle=90,end angle=-90,radius=.85]
   -- (-.7,-.85) arc[start angle=270,end angle=90,radius=.85];
 \draw[gray!70] (-.7,.85) -- (.7,.85)
   arc[start angle=90,end angle=-90,radius=.85]
   -- (-.7,-.85) arc[start angle=270,end angle=90,radius=.85];
 \draw[very thick] (-.7,0) -- (.7,0);
 \fill (-.7,0) circle (2pt);
 \fill (.7,0) circle (2pt);
 \draw[densely dashed,thick] (0,-.85) -- (0,.85);
 \node[above] at (0,1.05) {cut across the band};
 \draw[gray!65] (0,1.04) -- (0,.9);
 \node[below] at (0,-1.2) {(a) Cut};
\end{scope}
\draw[->,thick] (-2.72,0) -- (-2.3,0);
% The two cut regions, with matching tangent lines on the cut.
\begin{scope}[shift={(-.05,0)}]
 \foreach \side in {-1,1} {
  \begin{scope}[xscale=\side,xshift=.25cm]
   \fill[gray!8] (0,.85) -- (.7,.85)
     arc[start angle=90,end angle=-90,radius=.85]
     -- (0,-.85) -- cycle;
   \begin{scope}
    \clip (0,.85) -- (.7,.85)
      arc[start angle=90,end angle=-90,radius=.85]
      -- (0,-.85) -- cycle;
    \foreach \ratio in {.16,.3,.46,.63,.8} {
     \pgfmathsetmacro{\centre}{.7*(1+\ratio*\ratio)/(1-\ratio*\ratio)}
     \pgfmathsetmacro{\radius}{1.4*\ratio/(1-\ratio*\ratio)}
     \draw[darkblue,thin] (\centre,0) circle (\radius);
    }
   \end{scope}
   \draw[gray!70] (0,.85) -- (.7,.85)
     arc[start angle=90,end angle=-90,radius=.85]
     -- (0,-.85);
   \draw[darkblue,thick] (0,-.85) -- (0,.85);
   \draw[very thick] (0,0) -- (.7,0);
   \fill (.7,0) circle (2pt);
  \end{scope}
 }
 \node[text width=3.4cm,align=center] at (0,1.36)
   {tangent to the paired cut intervals};
 \draw[gray!65] (-.28,1.03) -- (-.25,.68);
 \draw[gray!65] (.28,1.03) -- (.25,.68);
 \node[below] at (0,-1.2) {(b) Choose local fields};
\end{scope}
\draw[->,thick] (2.15,0) -- (2.57,0);
% Gluing restores the ribbon neighbourhood and the global line field.
\begin{scope}[shift={(4.4,0)}]
 \fill[gray!8] (-.7,.85) -- (.7,.85)
   arc[start angle=90,end angle=-90,radius=.85]
   -- (-.7,-.85) arc[start angle=270,end angle=90,radius=.85];
 \begin{scope}
  \clip (-.7,.85) -- (.7,.85)
    arc[start angle=90,end angle=-90,radius=.85]
    -- (-.7,-.85) arc[start angle=270,end angle=90,radius=.85];
  \foreach \side in {-1,1} {
   \foreach \ratio in {.16,.3,.46,.63,.8} {
    \pgfmathsetmacro{\centre}{\side*.7*(1+\ratio*\ratio)/(1-\ratio*\ratio)}
    \pgfmathsetmacro{\radius}{1.4*\ratio/(1-\ratio*\ratio)}
    \draw[darkblue,thin] (\centre,0) circle (\radius);
   }
  }
  \draw[darkblue,thick] (0,-.85) -- (0,.85);
 \end{scope}
 \draw[gray!70] (-.7,.85) -- (.7,.85)
   arc[start angle=90,end angle=-90,radius=.85]
   -- (-.7,-.85) arc[start angle=270,end angle=90,radius=.85];
 \draw[very thick] (-.7,0) -- (.7,0);
 \fill (-.7,0) circle (2pt);
 \fill (.7,0) circle (2pt);
 \node at (0,1.35) {$\eta_\Gamma$};
 \node[below] at (0,-1.2) {(c) Glue};
\end{scope}
\end{tikzpicture}
\caption{Construction of $\eta_\Gamma$ for a ribbon graph with one edge. The black dots and segment are the vertices and edge of $\Gamma$; the grey boundary is $\partial\Sigma_\Gamma$. Cutting along the dashed path gives two regions. In each region, the blue curves specify the local field by their tangent lines. The fields are tangent to both copies of the cut and agree when those copies are glued.}
\label{fig:canonical-line-field}
\end{figure}

\begin{dfn}\label{dfn:ribbon-type}
A line field $\eta$ on $(\Sigma,\mathcal P)$ is of \defn{ribbon type} if
\begin{itemize}
\item there is a filling embedding $\Gamma\hookrightarrow\Sigma$ of a ribbon graph with vertex set $\mathcal P$ such that $\eta$ is homotopic to $\eta_\Gamma$;
\item near each puncture, there are local coordinates centred at the puncture in which
\[
 \eta(x,y)=\mathbb R(-y,x)\qquad\text{for }(x,y)\ne(0,0).
\]
Equivalently, $\eta$ is tangent to the concentric circles $x^2+y^2=c$, where $c>0$ is sufficiently small.
\end{itemize}
\end{dfn}

For a graded arc, one also records how to rotate the chosen line $\eta(\gamma(t))$ to the tangent line of the arc. Since different paths around the direction circle may have the same endpoints but make different numbers of turns, we use the grading to retain this information.

\begin{dfn}\label{dfn:graded-arc}
For a smooth immersed arc $\gamma:[0,1]\to\Sigma$, let $\dot{\gamma}(t)\in T_{\gamma(t)}\Sigma$ denote its velocity vector at parameter $t$, for $0<t<1$. Since $\gamma$ is an immersion, $\dot{\gamma}(t)\ne0$. A \defn{grading} of $\gamma$ is a homotopy class, relative to its endpoints in the homotopy parameter, of maps
\[
 g_\gamma:[0,1]\times(0,1)\longrightarrow\mathbb P(T\Sigma)
\]
such that $g_\gamma(u,t)\in\mathbb P(T_{\gamma(t)}\Sigma)$ and
\[
 g_\gamma(0,t)=\eta(\gamma(t)),\qquad
 g_\gamma(1,t)=\mathbb R\dot{\gamma}(t).
\]
A \defn{graded arc} is an arc with a grading. For $s\in\ZZ$, the shift $\gamma[s]$ adds $s$ clockwise turns in the projective tangent fibre to its grading; a negative value of $s$ means turns in the opposite direction. An arc system is \defn{graded} if a grading has been chosen for each arc.
\end{dfn}

For each fixed $t$, the map $u\mapsto g_\gamma(u,t)$ is a path in the single circle $\mathbb P(T_{\gamma(t)}\Sigma)$. These paths must vary continuously with $t$. 
The choice of a grading records the homotopy class of this family of paths, with their endpoints fixed. The shift $\gamma[1]$ adds one clockwise turn to every path in this family, as illustrated in \cref{fig:arc-grading}.
For the construction of gradings of arcs on surfaces, see \cite[Sections 2a--2b]{Sei00} and \cite[Definition 3.10]{OZ22}.

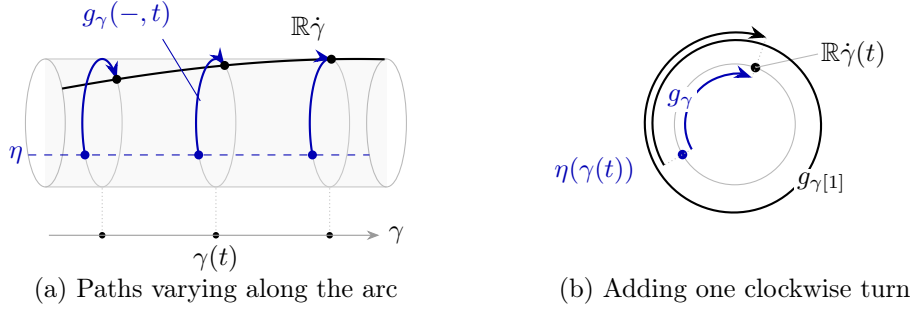
\begin{figure}[htbp]
\centering
\begin{tikzpicture}[>=Stealth,scale=.94]
% A local portion of the circle bundle pulled back to the arc.
\begin{scope}[shift={(-3.2,0)}]
 \fill[gray!5] (-2.4,-.9) rectangle (2.4,.9);
 \draw[gray!45] (-2.4,-.9) -- (2.4,-.9);
 \draw[gray!45] (-2.4,.9) -- (2.4,.9);
 \draw[gray!55] (-2.4,0) ellipse[x radius=.28,y radius=.9];
 \draw[gray!55] (2.4,0) ellipse[x radius=.28,y radius=.9];
 \draw[dashed,darkblue] (-2.64,-.45) -- (2.16,-.45);
 \draw[thick,domain=-2.4:2.4,samples=65,variable=\t]
   plot ({\t+.28*cos(64+13.125*\t)},{.9*sin(64+13.125*\t)});
 \foreach \figx/\ang in {-1.6/43,0/64,1.6/85} {
   \draw[gray!45] (\figx,0) ellipse[x radius=.28,y radius=.9];
   \draw[thick,darkblue,->,domain=210:\ang,samples=40,variable=\a]
     plot ({\figx+.28*cos(\a)},{.9*sin(\a)});
   \fill[darkblue] ({\figx+.28*cos(210)},{.9*sin(210)}) circle (1.8pt);
   \fill ({\figx+.28*cos(\ang)},{.9*sin(\ang)}) circle (1.8pt);
   \draw[densely dotted,gray!50] (\figx,-.92) -- (\figx,-1.55);
   \fill (\figx,-1.58) circle (1.3pt);
 }
 \draw[->,gray!80] (-2.35,-1.58) -- (2.3,-1.58);
 \node[right] at (2.3,-1.58) {$\gamma$};
 \node[below] at (0,-1.58) {$\gamma(t)$};
 \node[darkblue,left,inner sep=1pt] at (-2.67,-.45) {$\eta$};
 \node[above] at (1.3,1.03) {$\mathbb R\dot{\gamma}$};
 \node[darkblue,above] at (-1.25,1.2) {$g_\gamma(-,t)$};
 \draw[darkblue,thin] (-.75,1.18) -- (-.22,.31);
 \node at (0,-2.35) {(a) Paths varying along the arc};
\end{scope}
% The same endpoints, but a different homotopy class of path.
\begin{scope}[shift={(4.1,-.02)}]
 \draw[gray!55] (0,0) circle (.85);
 \fill[darkblue] (210:.85) circle (1.9pt);
 \node[darkblue,left,inner sep=1pt] at (-1.36,-.65) {$\eta(\gamma(t))$};
 \fill (70:.85) circle (1.9pt);
 \draw[gray!55] (70:.85) -- (1.17,.99);
 \node[right,inner sep=1pt] at (1.17,.99) {$\mathbb R\dot{\gamma}(t)$};
 \draw[densely dotted,gray!55] (210:.7) -- (210:1.15);
 \draw[densely dotted,gray!55] (70:.7) -- (70:1.30);
 \draw[thick,darkblue,->] (210:.7)
   arc[start angle=210,end angle=70,radius=.7];
 \node[darkblue,fill=white,inner sep=1pt] at (-.76,.35) {$g_\gamma$};
 \draw[thick,->,domain=0:1,samples=130,variable=\u]
   plot ({(1.15+.15*\u)*cos(210-500*\u)},{(1.15+.15*\u)*sin(210-500*\u)});
 \node[fill=white,inner sep=1pt] at (1.22,-.82) {$g_{\gamma[1]}$};
 \node at (0,-2.33) {(b) Adding one clockwise turn};
\end{scope}
\end{tikzpicture}
\caption{A grading specifies a path from the line field to the tangent line in every fibre over the arc. In (a), the cylinder represents a portion of the family of direction circles over $\gamma$; the blue paths vary continuously with $t$. For a fixed $t$, panel (b) compares a grading with its shift by $1$. Both paths have the same endpoints, but the shifted path makes one additional clockwise turn. Radial offsets in (b) separate the paths for visibility.}
\label{fig:arc-grading}
\end{figure}

The same definition applies to an immersed path away from the punctures. We use such paths to define degrees at punctures. For an edge of $\Gamma\subset\Sigma_\Gamma$, transversality to $\eta_\Gamma$ gives a \defn{canonical grading}, that is, in each projective tangent fibre, take the embedded clockwise path from $\eta_\Gamma$ to the tangent line of the edge.

\begin{dfn}\label{dfn:intersection-degree}
Let $\gamma,\delta$ be graded arcs or graded paths. An \defn{oriented intersection} from $\gamma$ to $\delta$ specifies an occurrence on each path of a common intersection. Away from $\mathcal P$, the intersection is required to be transverse. At a puncture, the two outward branches must have distinct tangent rays and determine the sector traversed clockwise from $\gamma$ to $\delta$. We write $\vv\in\gamma\overrightarrow{\cap}\delta$. The \defn{degree} $\deg(\vv)$ is defined as follows; see \cref{fig:intersection-degree,fig:puncture-degree}.
\begin{itemize}
\item If $\vv\notin \mathcal P$, follow $g_\gamma$ at $\vv$, then the embedded clockwise path $\varepsilon$ from the tangent line of $\gamma$ to that of $\delta$, and finally the reverse of $g_\delta$ at $\vv$. This is a loop in $\mathbb P(T_{\vv}\Sigma)$; its class is $\deg(\vv)\in\ZZ$, the number of clockwise turns minus the number of counterclockwise turns of the closed path.
\item If $\vv\in \mathcal P$, choose a short embedded graded path $u$ across the specified sector, meeting $\gamma$ and $\delta$ transversely at its endpoints away from $\mathcal P$. For the resulting oriented intersections $\ww_\gamma\in\gamma\overrightarrow{\cap}u$ and $\ww_\delta\in\delta\overrightarrow{\cap}u$, set
\[
 \deg(\vv)=\deg(\ww_\gamma)-\deg(\ww_\delta).
\]
\end{itemize}
\end{dfn}

The degree at a puncture is independent of the choice of the path $u$ and its grading. In particular, shifting the gradings gives
\[
 \deg_{\gamma[s],\delta[t]}(\vv)
 =\deg_{\gamma,\delta}(\vv)+s-t,
\]
in agreement with the shifts in $\Add(\A)$.

For the gradings drawn in \cref{fig:intersection-degree}\textup{(b)}, the return path reverses the rotation accumulated along $g_\gamma$ and $\varepsilon$, so $\deg(\vv)=0$. In \cref{fig:intersection-degree}\textup{(c)}, the underlying arcs and line field are unchanged, but the grading of $\delta$ is replaced by that of $\delta[-1]$. The resulting closed path makes one clockwise turn, so its degree is $1$. Thus the degree depends on the gradings as well as on the geometric intersection.

\begin{figure}[htbp]
\centering
\begin{tikzpicture}[>=Stealth,scale=.94]
\begin{scope}[shift={(-4.65,0)}]
 \draw[rounded corners,gray!45] (-1.35,-1.05) rectangle (1.35,1.05);
 \draw[thick,darkblue] (20:-1.15) -- (20:1.15)
   node[above left] {$\gamma$};
 \draw[thick] (-25:-1.15) -- (-25:1.15)
   node[right] {$\delta$};
 \fill (0,0) circle (1.7pt) node[below left] {$\vv$};
 \node at (0,-2.1) {(a) Crossing arcs};
\end{scope}
\foreach \figx/\case in {-.15/0,4.6/1} {
\begin{scope}[shift={(\figx,0)}]
 \draw[gray!55] (0,0) circle (1);
 \draw[densely dotted,gray!55] (160:1) -- (160:1.43);
 \draw[densely dotted,gray!55] (40:.79) -- (40:1.21);
 \draw[densely dotted,gray!55] (-50:.79) -- (-50:1.43);
 \draw[thick,darkblue,->] (160:1.21)
   arc[start angle=160,end angle=40,radius=1.21];
 \node[darkblue,fill=white,inner sep=1pt] at (111:1.21) {$g_\gamma$};
 \draw[thick,->] (40:.79)
   arc[start angle=40,end angle=-50,radius=.79];
 \node at (-4:.53) {$\varepsilon$};
 \ifnum\case=0
   \draw[thick,gray!80!black,->] (-50:1.43)
     arc[start angle=-50,end angle=160,radius=1.43];
   \node[fill=white,inner sep=1pt] at (76:1.67) {$\overline{g_\delta}$};
   \node at (0,-2.1) {(b) $\deg_{\gamma,\delta}(\vv)=0$};
 \else
   \draw[thick,gray!80!black,->] (-50:1.43)
     arc[start angle=-50,end angle=-200,radius=1.43];
   \node[fill=white,inner sep=1pt] at (0,-1.63) {$\overline{g_{\delta[-1]}}$};
   \node at (0,-2.1) {(c) $\deg_{\gamma,\delta[-1]}(\vv)=1$};
 \fi
 \fill (160:1) circle (1.6pt);
 \node[left,inner sep=1pt] at (-1.56,.7) {$\eta(\vv)$};
 \fill (40:1) circle (1.6pt);
 \fill (-50:1) circle (1.6pt);
 \draw[gray!55] (40:1) -- (1.47,.71);
 \draw[gray!55] (-50:1) -- (1.34,-.91);
 \node[right,inner sep=1pt] at (1.47,.71) {$\mathbb R\dot{\gamma}$};
 \node[right,inner sep=1pt] at (1.34,-.91) {$\mathbb R\dot{\delta}$};
\end{scope}
}
\end{tikzpicture}
\caption{Degrees $0$ and $1$ at the same geometric intersection. }
\label{fig:intersection-degree}
\end{figure}
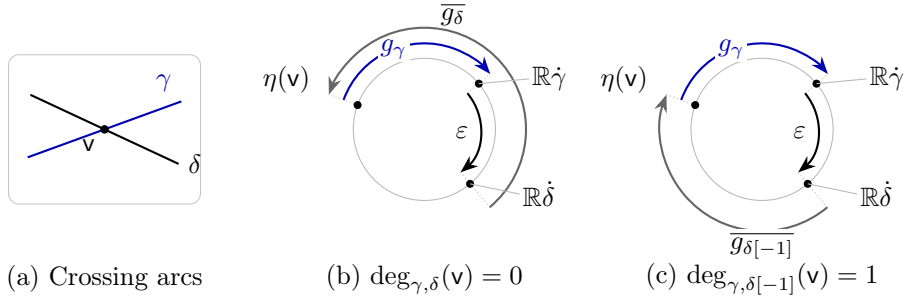

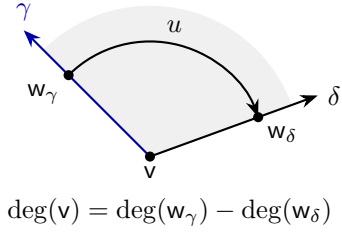
\begin{figure}[htbp]
\centering
\begin{tikzpicture}[>=Stealth,scale=1.05]
 \fill[gray!12] (0,0) -- (135:1.9)
   arc[start angle=135,end angle=20,radius=1.9] -- cycle;
 \draw[thick,darkblue,->] (0,0) -- (135:2.25)
   node[above] {$\gamma$};
 \draw[thick,->] (0,0) -- (20:2.25)
   node[right] {$\delta$};
 \draw[thick,->] (135:1.45)
   arc[start angle=135,end angle=20,radius=1.45];
 \node[above] at (78:1.48) {$u$};
 \fill (135:1.45) circle (1.8pt)
   node[below left] {$\ww_\gamma$};
 \fill (20:1.45) circle (1.8pt)
   node[below right] {$\ww_\delta$};
 \fill (0,0) circle (1.8pt) node[below] {$\vv$};
 \node at (.25,-.7) {$\deg(\vv)=\deg(\ww_\gamma)-\deg(\ww_\delta)$};
\end{tikzpicture}
\caption{Computing the degree at a puncture. }
\label{fig:puncture-degree}
\end{figure}

\begin{dfn}\label{dfn:winding-number}
For an oriented immersed closed curve $\gamma:S^1\to\Sigma\setminus \mathcal P$, its \defn{winding number} $\omega_\eta(\gamma)$ is the signed intersection number of the tangent lift $t\mapsto\mathbb R\dot{\gamma}(t)$ with $\eta(\Sigma\setminus \mathcal P)$ in $\mathbb P(T\Sigma)$. Equivalently, after a perturbation making the tangencies transverse, it counts them positively when the tangent line crosses $\eta$ clockwise, and negatively for the reverse crossing. For an arc transverse to $\eta$ near its endpoints, use the same signed count along its interior.
\end{dfn}

The signed intersections occur at parameters where $\mathbb R\dot{\gamma}(t)=\eta(\gamma(t))$. Thus $\omega_\eta(\gamma)$ measures the rotation of the tangent line relative to the line field along the curve. For a clockwise circle in a disc with a horizontal line field, the tangent vector makes one clockwise turn, so its unoriented line makes two turns in the direction circle and $\omega_\eta(\gamma)=2$. For a circle around a puncture with the line field tangent to concentric circles, the tangent line and the line field rotate together. A small constant rotation of the line field gives a homotopy of line fields, so it does not change the winding number. The rotated field has no tangencies with the circle, giving $\omega_\eta(\gamma)=0$. These examples are shown in \cref{fig:winding-number}, where the blue segments in \textup{(b)} depict the rotated line field.

\begin{figure}[htbp]
\centering
\begin{tikzpicture}[>=Stealth,scale=.96]
\begin{scope}[shift={(-3.5,0)}]
 \filldraw[rounded corners,fill=gray!4,draw=gray!40]
   (-2,-1.55) rectangle (2,1.6);
 \foreach \x in {-1.6,-.8,0,.8,1.6}
   \foreach \y in {-1.2,-.6,0,.6,1.2}
     \draw[darkblue!60] (\x-.2,\y) -- (\x+.2,\y);
 \draw[thick] (0,0) circle (1.05);
 \draw[thick,->] (20:1.05) arc[start angle=20,end angle=-35,radius=1.05];
 \fill (0,1.05) circle (1.7pt);
 \fill (0,-1.05) circle (1.7pt);
 \draw[thick] (-.36,1.05) -- (.36,1.05);
 \draw[thick] (-.36,-1.05) -- (.36,-1.05);
 \node[fill=white,inner sep=1pt] at (0,1.38) {$+1$};
 \node[fill=white,inner sep=1pt] at (0,-1.38) {$+1$};
 \node[right,fill=white,inner sep=1pt] at (1.08,.25) {$\gamma$};
 \node at (0,2.02) {horizontal line field};
 \node at (0,-2.05) {(a) $\omega_\eta(\gamma)=2$};
\end{scope}
\begin{scope}[shift={(3.5,0)}]
 \filldraw[rounded corners,fill=gray!4,draw=gray!40]
   (-2,-1.55) rectangle (2,1.6);
 \foreach \r in {.5,1.05,1.43} {
  \foreach \a in {0,45,90,135,180,225,270,315} {
   \draw[darkblue!65] (\a:\r) ++({\a+110}:-.18) -- ++({\a+110}:.36);
  }
 }
 \draw[thick] (0,0) circle (1.05);
 \draw[thick,->] (20:1.05) arc[start angle=20,end angle=-35,radius=1.05];
 \filldraw[fill=white,draw=black] (0,0) circle (2pt);
 \node[below right] at (0,0) {$\vv$};
 \node[right,fill=white,inner sep=1pt] at (1.08,.25) {$\gamma$};
 \node at (0,2.02) {line field around a puncture};
 \node at (0,-2.05) {(b) $\omega_\eta(\gamma)=0$};
\end{scope}
\end{tikzpicture}
\caption{Winding counts rotation of the tangent line relative to the line field. Both curves are traversed clockwise.}
\label{fig:winding-number}
\end{figure}
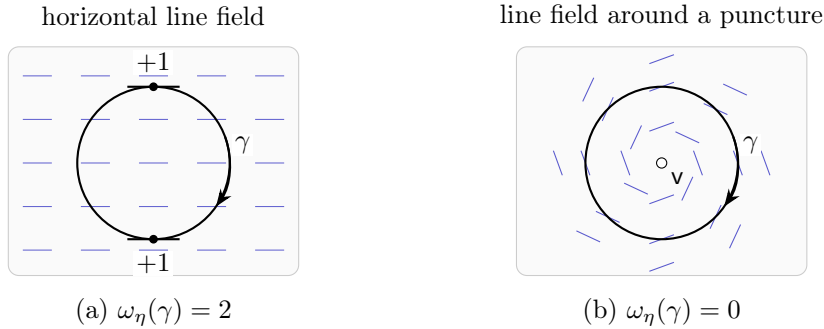

For a closed curve, the winding number can also be interpreted as the degree of a map between circles; see \cite[Section 2.2]{Hat02} for the notion of degree.

 Note that winding numbers are invariant under regular homotopies.
For the canonical ribbon field,
\[
 \omega_{\eta_\Gamma}(\gamma)=0\quad(\gamma\in E(\Gamma)),
 \qquad \omega_{\eta_\Gamma}(\varepsilon_{\vv})=0,
 \qquad \omega_{\eta_\Gamma}(B)=-\ell(B),
\]
where $\varepsilon_{\vv}$ is a small counterclockwise simple closed curve around $\vv\in V(\Gamma)$, and $B$ is a boundary component with its orientation chosen so that $\Sigma$ lies on its right and corresponding face perimeter $\ell(B)$; see \cite[Example 3.16]{OZ22}. Moreover, $\eta_\Gamma$ is orientable if and only if $\Gamma$ is bipartite; see \cite[Lemma 7.9]{OZ22}.

\begin{dfn}\label{dfn:disc-sequence}
Let $\U$ be a full graded arc system. A \defn{marked disc} is a continuous map from a disc with $n\geq3$ distinguished boundary points to $\Sigma$, which is an orientation-preserving smooth immersion away from those points, maps its interior into $\Sigma\setminus \mathcal P$, and maps its boundary sides to arcs of $\U$ and its distinguished points to punctures. List its sides in counterclockwise order as $\delta_1,\ldots,\delta_n$, and set $\delta_{n+1}=\delta_1$. The corners give oriented intersections $\vv_i\in\delta_i\overrightarrow{\cap}\delta_{i+1}$. The corresponding sequence of corner paths $(a_n,\ldots,a_1)$ is called a \defn{disc sequence}. Here $a_i$ follows the successors at $\vv_i$ from the endpoint branch of $\delta_i$ to that of $\delta_{i+1}$ through the sector of the disc.
\end{dfn}

Thus $a_i$ is a path in $Q_{\Gamma_\U}$ from $\delta_i$ to $\delta_{i+1}$. Degrees are additive along such paths, and the corner degrees satisfy
\[
 |a_i|=\deg(\vv_i),\qquad \sum_{i=1}^n|a_i|=n-2;
\]
see \cite[Lemma 3.19 and Remark 3.23]{OZ22}. An admissible arc system admitting no marked discs is called \defn{formal}.

\subsection{Brauer graph \texorpdfstring{$A_\infty$}{A-infinity}-categories}
\label{sec:brauer-category-definition}

Winding numbers determine signs in the cycle relations of the algebra associated with a graded arc system. We first specify these relations and the complementary paths appearing in the higher operations.

\begin{dfn}\textnormal{(cf. \cite[Definition 1.9]{OZ22})}\label{dfn:modified-bga}
Let $(\Gamma,\m)$ be a Brauer graph and let $\omega:E(\Gamma)\to\ZZ$. The \defn{modified Brauer graph algebra} $B(\Gamma,\m,\omega)$ is the quotient of $\Bbbk Q_\Gamma$ by the zero relations in $I_{\Gamma,\m}$ and
\[
 C_h^{\m(s(h))}
 =(-1)^{\omega(e(h))}C_{\iota h}^{\m(s(\iota h))}
 \qquad(h\in H(\Gamma)).
\]
\end{dfn}

Note that for $\omega=0$, we recover $B(\Gamma,\m)$.

For a nonzero nontrivial path represented by $u=p_{h,t}$, where $1\leq t\leq\m(s(h))\operatorname{val}_\Gamma(s(h))$, define its \defn{complementary path} by
\[
 u^*=p_{\rho^t h,\,\m(s(h))\operatorname{val}_\Gamma(s(h))-t}.
\]
It has the reverse source and target, and
\[
 u^*u=C_h^{\m(s(h))},\qquad
 uu^*=C_{\rho^t h}^{\m(s(h))}.
\]
In particular, the complement of a maximal path is a trivial path.

Now fix a graded punctured surface $(\Sigma,\mathcal P,\eta)$ of ribbon type, a graded admissible arc system $\U$, and a multiplicity function $\m:\mathcal P\to\ZZ_{>0}$. Use the winding numbers of the arcs in the modified algebra $B(\Gamma_\U,\m,\omega_\eta)$, choosing an orientation of each arc to evaluate $\omega_\eta$; the cycle relations are independent of this choice. Assign to a successor arrow the degree of its oriented intersection, and give the loop at a valency-one puncture degree zero. Path degrees are additive, and trivial paths have degree zero. Since $\eta$ is of ribbon type, the special cycles have degree zero. The relations of $B(\Gamma_\U,\m,\omega_\eta)$ are therefore homogeneous and $|u^*|=-|u|$.

\begin{dfn}
\textnormal{(cf.~\cite[Definition 4.4]{OZ22})}
\label{dfn:brauer-category}
For the graded admissible arc system $\U$ fixed above, the \defn{Brauer graph $A_\infty$-category} $\B(\U,\m)$ has objects $X_\gamma$, $\gamma\in\U$, and graded morphism spaces
\[
 \Hom_{\B(\U,\m)}(X_\gamma,X_\delta)
 =e_\delta B(\Gamma_\U,\m,\omega_\eta)e_\gamma.
\]
Its operations are specified on path representatives as follows.
\begin{enumerate}[label=\textup{(\arabic*)}]
\item Set $\mu^1=0$ and
\[
 \mu^2(b,a)=(-1)^{|a|}ba
\]
for homogeneous composable paths $a,b$, using the product of $B(\Gamma_\U,\m,\omega_\eta)$. The unit of $X_\gamma$ is $e_\gamma$.

\item For each disc sequence $(a_n,\ldots,a_1)$, where $a_i:X_{\delta_i}\to X_{\delta_{i+1}}$, set
\begin{equation}\label{eq:brauer-polygon-operations}
\begin{aligned}
 \mu^n(ba_n,a_{n-1},\ldots,a_1)&=b
 &&\text{if }ba_n\ne0,\\
 \mu^n(a_n,\ldots,a_2,a_1b)&=(-1)^{|b|}b
 &&\text{if }a_1b\ne0.
\end{aligned}
\end{equation}
Here $b$ is any composable path, including a trivial path.

\item For the same disc sequence, $2\leq j\leq n$, and a path $b$ with $ba_j\ne0$, set
\begin{equation}\label{eq:brauer-complement-operation}
\begin{aligned}
 &\mu^n\bigl(a_n,\ldots,a_{j+1},
       a_j(ba_j)^*,ba_j,a_{j-1},\ldots,a_2\bigr)
       =(-1)^\varepsilon a_1^*,\\
 &\varepsilon=\sum_{i=1}^{j-1}|a_i|+|ba_j|
                 +\sum_{i=2}^{j}\omega_\eta(\delta_i).
\end{aligned}
\end{equation}
When $j=n$ or $j=2$, the corresponding empty sequence of arguments is omitted.
\end{enumerate}
The formulas apply to every cyclic ordering of a disc sequence. Extend by multilinearity and require all operations of arity at least three containing a unit to vanish.
\end{dfn}

The operations are well-defined; see \cite[Lemma 4.5]{OZ22}. In fact, these higher-operation formulas will all be needed when checking the group weights in \cref{sec:geometric-operations}.

\begin{thm}
\textnormal{(cf.~\cite[Proposition 4.6, Corollary 5.15]{OZ22})}
\label{thm:oz-constructions}
\begin{enumerate}[label=\textup{(\arabic*)}]
\item The operations in \cref{dfn:brauer-category} satisfy \eqref{eq:relations} and define an $A_\infty$-category.
\item Let $(\Gamma,\m)$ be a Brauer graph. On $(\Sigma_\Gamma,V(\Gamma),\eta_\Gamma)$, regard the edges of $\Gamma$ as arcs with their canonical gradings. Then $\B(\Gamma,\m)$ is concentrated in degree zero, has $\mu^n=0$ for $n\ne2$, and its associated algebra is $B(\Gamma,\m)$. Consequently,
\[
 H^0\bigl(\Tw\B(\Gamma,\m)\bigr)
 \simeq\K^b\bigl(\proj\text{-}B(\Gamma,\m)\bigr).
\]
\end{enumerate}
\end{thm}

We next recall how a marked disc expresses any one of its sides as a twisted complex of the other sides. This extends the polygon example in \cite[Section 3.3]{HKK17} to Brauer graph $A_\infty$-categories; see the proof of \cite[Proposition 6.1]{OZ22}.

\begin{prop}\label{prop:brauer-polygon}
Let $(a_n,\ldots,a_1)$ be a disc sequence in $\B(\U,\m)$. Write $a_i:X_{i-1}\to X_i$, with $X_n=X_0$, and set $t_i=\sum_{j=1}^i|a_j|-(i-1)$ with $1\leq i\leq n-1$. There is an isomorphism in $H^0(\Tw\B(\U,\m))$
\[
 X_0\cong
 \left(\bigoplus_{i=1}^{n-1}X_i[t_i],\delta\right),
 \qquad
 \delta_{i,i+1}=a_{i+1},
\]
where all other components of $\delta$ are zero. If a side occurs more than once, its occurrences give separate summands in the displayed direct sum. In particular,
\[
 y_{X_0}\in
 \thick_{\D(\B(\U,\m))}\{y_{X_1},\ldots,y_{X_{n-1}}\}.
\]
\end{prop}

\begin{proof}
Write $T$ for the displayed twisted complex. Since $t_{i+1}-t_i=|a_{i+1}|-1$, every twisting component $a_{i+1}:X_i[t_i]\to X_{i+1}[t_{i+1}]$ has degree one. The polygon calculation gives zero for every operation on a proper consecutive sequence of corners of the marked disc. Consequently, $\delta$ satisfies \eqref{eq:mc-prelim}.

Define $f:X_0\to T$ and $g:T\to X_0$ by the following diagrams. The horizontal arrows in the lower rows are the twisting components, and all components of $f$ and $g$ not displayed are zero.
\[
f:\qquad
\begin{tikzcd}[column sep=large,row sep=large]
 X_0\arrow[r]\arrow[d,"a_1"']
 &0\arrow[r]&0\\
 X_1[t_1]\arrow[r,"a_2"]
 &\cdots\arrow[r,"a_{n-1}"]
 &X_{n-1}[t_{n-1}]
\end{tikzcd}
\]
\[
g:\qquad
\begin{tikzcd}[column sep=large,row sep=huge]
 X_0\arrow[r]
 &0\arrow[r]
 &0\\
 X_1[t_1]\arrow[r,"a_2"]
 &\cdots\arrow[r,"a_{n-1}"]
 &X_{n-1}[t_{n-1}]
   \arrow[ull,"{(-1)^{t_1+\cdots+t_{n-1}}a_n}"']
\end{tikzcd}
\]
Since $t_1=|a_1|$ and $t_{n-1}=-|a_n|$, both $f$ and $g$ have degree zero. By \eqref{eq:twisted-operations} and \eqref{eq:brauer-polygon-operations}, $f$ and $g$ satisfy
\[
\begin{aligned}
\mu^1_{\Tw\B(\U,\m)}(f)&=0,
&\mu^1_{\Tw\B(\U,\m)}(g)&=0,\\
\mu^2_{\Tw\B(\U,\m)}(g,f)&=1_{X_0},
&\mu^2_{\Tw\B(\U,\m)}(f,g)&=1_T.
\end{aligned}
\]
Hence $[f]$ and $[g]$ are inverse isomorphisms. Applying the Yoneda functor to $T$ gives the assertion about representable modules.
\end{proof}

\begin{prop}
\textnormal{(cf.~\cite[Proposition 6.1 and Corollary 6.5]{OZ22})}
\label{prop:brauer-arc-independence}
Fix $(\Sigma,\mathcal P,\eta,\m)$ with $\eta$ of ribbon type.
\begin{enumerate}[label=\textup{(\arabic*)}]
\item Let $\U$ be a graded admissible arc system and let $\gamma\in\U$ be such that $\U\setminus\{\gamma\}$ is full. Then the natural full inclusion
\[
 \B(\U\setminus\{\gamma\},\m)
 \hookrightarrow\B(\U,\m)
\]
is a Morita equivalence. In fact, it induces an equivalence on $H^0(\Tw(-))$.
\item For any two graded admissible arc systems $\U$ and $\U'$ on $(\Sigma,\mathcal P,\eta)$, the categories $\B(\U,\m)$ and $\B(\U',\m)$ are Morita equivalent.
\end{enumerate}
\end{prop}

\begin{proof}
\noindent
(1) Removing $\gamma$ preserves the cutting paths, so $\U\setminus\{\gamma\}$ is admissible. The full inclusion sends a successor arrow to the path between the same endpoint branches in $Q_{\Gamma_\U}$. A marked disc containing $\gamma$ provides the twisted complex in \cref{prop:brauer-polygon}, using only arcs of $\U\setminus\{\gamma\}$. Hence \cref{prop:morita-criterion} gives the Morita equivalence. Moreover, the image of the induced functor on $H^0(\Tw(-))$ contains every arc object up to isomorphism. This functor is fully faithful and exact by \cref{prop:twisted-standard}, so its essential image is closed under shifts and cones and contains all twisted complexes. It is therefore an equivalence.

\smallskip\noindent
(2) Apply (1) along the sequence of additions and deletions in \cite[Proposition 6.4]{OZ22}. Changes of grading replace the corresponding objects by shifts and hence also preserve the Morita equivalence class.
\end{proof}
In the next section, we lift the operations of $\B(\U,\m)$ through cyclic coverings of punctured surfaces.

%%%%%%%%%%%%%%%%%%%%%%%%%%%%%%%%%%%%%%%%%%%%%%%%%%%%%%%%%%%%%%%%%%%%%%%%%%%%%%%%%%%%%%%%%%%%%
\section{Geometric models and \texorpdfstring{$A_\infty$}{A-infinity}-structures for admissible fractional Brauer graph algebras}
\label{sec:nakayama}

Note that any admissible fractional Brauer graph algebra has a Brauer graph algebra as its reduced form \cite{LL26}. We realize this reduction by a cyclic covering of punctured surfaces, specified by a Nakayama character introduced below, and lift the operations of the Brauer graph $A_\infty$-category through the covering. The resulting category has a (modified) AFBGA as its associated graded algebra.

\subsection{Admissible fractional Brauer graph algebras and reduced forms}

We retain the ribbon graph notation $\Gamma=(V,H,s,\iota,\rho)$, the quiver $Q_\Gamma$, and the successor paths $p_{h,t}$ from \cref{dfn:ribbon,dfn:bga}. In particular, successor paths are read from right to left.

\begin{dfn}
\label{dfn:afbg}
An \defn{admissible fractional Brauer graph} (abbr. AFBG) is a pair $(\Gamma,\dG)$, where $\Gamma$ is a connected ribbon graph and $\dG:V(\Gamma)\to\ZZ_{>0}$ is a \defn{degree function} such that the permutation
\[
 \nu:H(\Gamma)\longrightarrow H(\Gamma),
 \qquad \nu(h)=\rho^{\dG(s(h))}h
\]
satisfies
\[
 \nu\iota=\iota\nu,
 \qquad
 \nu^j h\ne\iota h
 \quad(h\in H(\Gamma),\ j\in\ZZ).
\]
The \defn{fractional multiplicity} at $\vv$ is $\dG(\vv)/\operatorname{val}_\Gamma(\vv)$. We call $\nu$ the \defn{Nakayama automorphism}; it fixes the vertices of $\Gamma$ and acts on edges by $\nu(e(h))=e(\nu h)$.
\end{dfn}

Since $\nu$ commutes with $\iota$, the paths $p_{h,\dG(s(h))}$ and $p_{\iota h,\dG(s(\iota h))}$ both start at $e(h)$ and end at $\nu(e(h))$.
Let $I_{\Gamma,\dG}$ be the ideal of $\Bbbk Q_\Gamma$ generated by
\begin{enumerate}[label=\textup{(\arabic*)}]
\item $p_{h,\dG(s(h))}=p_{\iota h,\dG(s(\iota h))}$ for $h\in H(\Gamma)$;
\item $a_k a_h=0$ for $e(k)=e(\rho h)$ and $k\ne\rho h$.
\end{enumerate}
Set
\[
 \Lambda(\Gamma,\dG)=\Bbbk Q_\Gamma/I_{\Gamma,\dG}.
\]

\begin{dfn}\label{dfn:afbga}
A $\Bbbk$-algebra is an \defn{admissible fractional Brauer graph algebra} (abbr. AFBGA) if it is isomorphic to $\Lambda(\Gamma,\dG)$ for some AFBG $(\Gamma,\dG)$.
\end{dfn}

An AFBGA is finite-dimensional, self-injective, and special biserial; see \cite{LL26}. Each $p_{h,\dG(s(h))}$ is a maximal nonzero successor path, and redundant arrows may occur when $\dG(\vv)=1$; see \cite[Remark 2.4]{Xin26tt}. If every fractional multiplicity is integral, write $\dG(\vv)=\m(\vv)\operatorname{val}_\Gamma(\vv)$; then $\Lambda(\Gamma,\dG)=B(\Gamma,\m)$.

\begin{prop}
\label{prop:reduction}
Let $(\Gamma,\dG)$ be an AFBG and let $r$ be the order of $\nu$. Write $o(\vv)=|H_{\vv}/\langle\nu\rangle|$ for the number of half-edge orbits at $\vv$.
\begin{enumerate}[label=\textup{(\arabic*)}]
\item The action of $\langle\nu\rangle$ is free on half-edges and on edges.
\item The quotient $\Gamma_{\red}=\Gamma/\langle\nu\rangle$ is a ribbon graph with vertex set $V(\Gamma)$ and half-edge set $H(\Gamma)/\langle\nu\rangle$. For every $\vv\in V(\Gamma)$,
\[
 \operatorname{val}_{\Gamma_{\red}}(\vv)=o(\vv),
 \qquad
 \operatorname{val}_\Gamma(\vv)=r\,o(\vv).
\]
\item The numbers $ \m(\vv)=\frac{\dG(\vv)}{o(\vv)}$
are positive integers coprime to $r$, and
$ \frac{\dG(\vv)}{\operatorname{val}_\Gamma(\vv)}
 =\frac{\m(\vv)}r.$
\end{enumerate}
\end{prop}

\begin{proof}
(1) By \cite[Lemma 2.7]{Xin26tt}, every half-edge orbit has $r$ elements. If $\nu^j$ fixes $e(h)$, then $\nu^j h=h$ or $\nu^j h=\iota h$. Admissibility excludes the second possibility, so $r$ divides $j$.

\smallskip\noindent
(2) Apply the quotient construction in \cite[Section 2.4]{Xin26tt}. Admissibility ensures that the induced involution has no fixed points. The valency formulas follow by counting the orbits in $H_{\vv}$.

\smallskip\noindent
(3) Integrality and the fractional multiplicity formula are \cite[Lemma 2.8 and Proposition 2.9]{Xin26tt}. 

Fix $\vv\in V(\Gamma)$ and $h\in H_{\vv}$. Since $\rho$ acts on $H_{\vv}$ as a cycle of length $\operatorname{val}_\Gamma(\vv)$ and $\nu=\rho^{\dG(\vv)}$ on $H_{\vv}$, we have that for any $k\geq1$,
$\nu^k h=h$  if and only if $\operatorname{val}_\Gamma(\vv)\mid k\dG(\vv)$. Hence the smallest positive integer $k$ satisfying this condition is
$\frac{\operatorname{val}_\Gamma(\vv)}
{\gcd\bigl(\operatorname{val}_\Gamma(\vv),\dG(\vv)\bigr)}$.
Thus every $\nu$-orbit in $H_{\vv}$ has this cardinality. Since $H_{\vv}$ contains $\operatorname{val}_\Gamma(\vv)$ half-edges, the number of orbits is
$o(\vv)=\gcd\bigl(\operatorname{val}_\Gamma(\vv),\dG(\vv)\bigr)$.
Together with $\operatorname{val}_\Gamma(\vv)=r\,o(\vv)$ and $\dG(\vv)=\m(\vv)o(\vv)$, this gives $\gcd(r,\m(\vv))=1$.
\end{proof}

\begin{dfn}\label{dfn:reduced-form}
The Brauer graph $(\Gamma_{\red},\m)$ in \cref{prop:reduction} is the \defn{reduced form} of $(\Gamma,\dG)$. The algebra
\[
 \Lambda_{\mathrm{red}}=B(\Gamma_{\red},\m)
\]
is the \defn{reduced form} of $\Lambda(\Gamma,\dG)$.
\end{dfn}

The denominator $r$ is therefore common to all fractional multiplicities when they are written in lowest terms. The covering constructed below realizes $\Gamma\to\Gamma_{\red}$ on surfaces and realizes $\Lambda_{\mathrm{red}}$ as an orbit category of $\Lambda(\Gamma,\dG)$, with the specified idempotents retained.

\subsection{Modified algebras and lifted \texorpdfstring{$A_\infty$}{A-infinity}-structures}
\label{sec:geometric-operations}

As in \cref{dfn:modified-bga}, winding numbers enter through signs in the relations. We first allow these signs in an AFBGA, then obtain the higher operations by applying \cref{prop:lift} to $\B(\U,\m)$.

\begin{dfn}\label{dfn:modified-afbga}
Let $(\Gamma,\dG)$ be an AFBG and let $\omega:E(\Gamma)\to\ZZ$ satisfy $\omega(\nu e)=\omega(e)$ for every edge $e$. The \defn{modified AFBGA} $\Lambda(\Gamma,\dG,\omega)$ is the quotient of $\Bbbk Q_\Gamma$ by the zero relations in $I_{\Gamma,\dG}$ and
\[
 p_{h,\dG(s(h))}
 =(-1)^{\omega(e(h))}p_{\iota h,\dG(s(\iota h))}
 \qquad(h\in H(\Gamma)).
\]
\end{dfn}

Only the parity of $\omega$ affects the algebra. The invariance of $\omega$ ensures that $\nu$ preserves the relations. For $\omega=0$ we recover $\Lambda(\Gamma,\dG)$, and for $r=1$ we recover \cref{dfn:modified-bga}.

We describe cyclic coverings and the homomorphism recording how the lift of a closed curve moves between the points of a fibre. For background on these notions, see \cite[Section 4]{For81}, \cite[Section 1.3]{Hat02}, and \cite[Section 5.1]{OZ22}.

\begin{dfn}
\label{dfn:cyclic-surface-cover}
Let $r\geq1$ and let $(\Sigma,\mathcal P)$ and $(\widetilde\Sigma,\widetilde{\mathcal P})$ be punctured surfaces. A continuous surjection
\[
\pi_\Sigma:(\widetilde\Sigma,\widetilde{\mathcal P})
\longrightarrow(\Sigma,\mathcal P),
\qquad \pi_\Sigma^{-1}(\mathcal P)=\widetilde{\mathcal P},
\]
is an \defn{$r$-sheeted cyclic covering, fully ramified over $\mathcal P$}, if the following conditions hold.
\begin{itemize}
\item Every $\xx\in\Sigma\setminus\mathcal P$ has a neighbourhood $U\subseteq\Sigma\setminus\mathcal P$ whose inverse image is a disjoint union of $r$ open sets, each mapped onto $U$ by an orientation-preserving diffeomorphism. These are the $r$ local sheets; in particular, the fibre $\pi_\Sigma^{-1}(\xx)$ consists of $r$ points.
\item There is an orientation-preserving diffeomorphism $\nu:\widetilde\Sigma\to\widetilde\Sigma$ such that $\pi_\Sigma\nu=\pi_\Sigma$, $\nu^r=\mathrm{id}$, and, for every $\xx\in\Sigma\setminus\mathcal P$ and $\widetilde\xx\in\pi_\Sigma^{-1}(\xx)$,
\[
\pi_\Sigma^{-1}(\xx)
=\{\widetilde\xx,\nu\widetilde\xx,\ldots,\nu^{r-1}\widetilde\xx\},
\]
with these $r$ points distinct. Its powers form the \defn{deck group}, identified with $\ZZ/r\ZZ$ by $\nu^s\leftrightarrow\bar s$. We fix $\nu$ as the \defn{deck generator}.
\item Each $\vv\in\mathcal P$ has a unique preimage $\widetilde\vv$. In orientation-preserving local coordinates centred at these points, the map $\pi_\Sigma$ is $z\mapsto z^r$.
\end{itemize}
\end{dfn}

\begin{dfn}\label{dfn:cover-monodromy}
Fix a covering $\pi_\Sigma$ and its deck generator $\nu$ as above. Choose $\xx\in\Sigma\setminus\mathcal P$ and $\widetilde\xx\in\pi_\Sigma^{-1}(\xx)$.
\begin{itemize}
\item For a loop $\gamma:[0,1]\to\Sigma\setminus\mathcal P$ with $\gamma(0)=\gamma(1)=\xx$, let $\widetilde\gamma:[0,1]\to\widetilde\Sigma\setminus\widetilde{\mathcal P}$ be its unique lift satisfying
\[
\pi_\Sigma\circ\widetilde\gamma=\gamma,
\qquad \widetilde\gamma(0)=\widetilde\xx.
\]
The endpoint has the form $\widetilde\gamma(1)=\nu^j\widetilde\xx$ for a unique $\bar j\in\ZZ/r\ZZ$. We call $\bar j$ the \defn{monodromy} of $\gamma$. Thus the lift is closed precisely when $\bar j=\bar0$.
\item Starting instead at $\nu^s\widetilde\xx$ gives the lift $\nu^s\widetilde\gamma$, which ends at $\nu^{s+j}\widetilde\xx$. Hence monodromy is independent of the chosen point of the fibre.
\item Homotopic loops relative to the basepoint have the same monodromy, and concatenating loops adds their monodromies. The resulting homomorphism from $\pi_1(\Sigma\setminus\mathcal P,\xx)$ to the abelian group $\ZZ/r\ZZ$ therefore factors through $H_1(\Sigma\setminus\mathcal P;\ZZ)$. We write
\[
\chi:H_1(\Sigma\setminus\mathcal P;\ZZ)\longrightarrow\ZZ/r\ZZ,
\qquad \chi([\gamma])=\bar j.
\]
\end{itemize}
\end{dfn}

For $\vv\in\mathcal P$, let $\gamma_{\vv}$ be a small clockwise loop around $\vv$. To regard it as a loop based at $\xx$, join it to $\xx$ by a path. Changing that path conjugates the based loop and hence does not change $\chi([\gamma_{\vv}])$. Full ramification implies that $\chi([\gamma_{\vv}])$ generates $\ZZ/r\ZZ$.
In \cref{fig:cyclic-monodromy}, the covering has three sheets. One clockwise circuit downstairs lifts from $\widetilde\xx$ to $\nu\widetilde\xx$, so $\chi([\gamma])=\bar1$. 

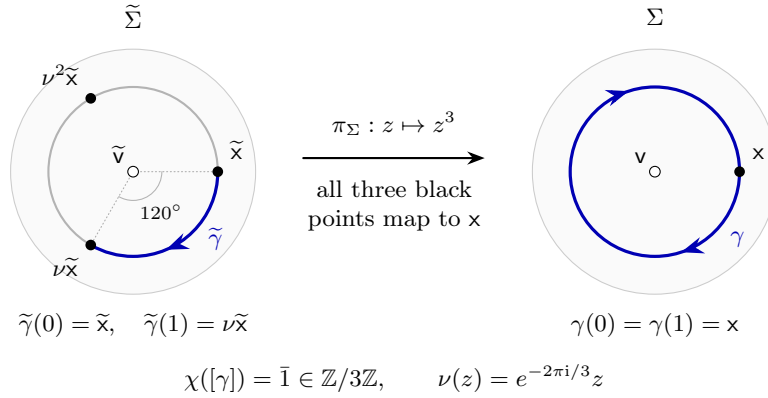
\begin{figure}[htbp]
\centering
\begin{tikzpicture}[>=Stealth,line cap=round,line join=round,font=\small]
\begin{scope}[shift={(-3.45,0)}]
 \filldraw[fill=gray!4,draw=gray!40] (0,0) circle (1.62);
 \draw[gray!60,thick] (0,0) circle (1.12);
 \draw[darkblue,very thick,postaction={decorate},
   decoration={markings,mark=at position .56 with {\arrow{Stealth}}}]
   (0:1.12) arc[start angle=0,end angle=-120,radius=1.12];
 \draw[densely dotted,gray!65] (0,0)--(0:1.12);
 \draw[densely dotted,gray!65] (0,0)--(-120:1.12);
 \draw[gray!60] (0:.38) arc[start angle=0,end angle=-120,radius=.38];
 \node[font=\scriptsize] at (-55:.65) {$120^\circ$};
 \filldraw[fill=white,draw=black] (0,0) circle (2pt);
 \node[above left,inner sep=3pt] at (0,0) {$\widetilde\vv$};
 \fill (0:1.12) circle (2pt) node[above right=2pt] {$\widetilde\xx$};
 \fill (-120:1.12) circle (2pt) node[below left=2pt] {$\nu\widetilde\xx$};
 \fill (120:1.12) circle (2pt) node[above left=2pt] {$\nu^2\widetilde\xx$};
 \node[darkblue,right,inner sep=2pt] at (-44:1.28) {$\widetilde\gamma$};
 \node at (0,2.05) {$\widetilde\Sigma$};
 \node at (0,-2.03) {$\widetilde\gamma(0)=\widetilde\xx,\quad
                        \widetilde\gamma(1)=\nu\widetilde\xx$};
\end{scope}
\draw[->,thick] (-1.2,.18)--(1.2,.18)
 node[midway,above=5pt] {$\pi_\Sigma:z\mapsto z^3$};
\node[align=center,text width=2.5cm] at (0,-.45)
 {all three black points map to $\xx$};
\begin{scope}[shift={(3.45,0)}]
 \filldraw[fill=gray!4,draw=gray!40] (0,0) circle (1.62);
 \draw[darkblue,very thick,postaction={decorate},
   decoration={markings,mark=at position .20 with {\arrow{Stealth}},
                        mark=at position .70 with {\arrow{Stealth}}}]
   (0:1.12) arc[start angle=0,end angle=-360,radius=1.12];
 \filldraw[fill=white,draw=black] (0,0) circle (2pt);
 \node[above left,inner sep=3pt] at (0,0) {$\vv$};
 \fill (0:1.12) circle (2pt) node[above right=2pt] {$\xx$};
 \node[darkblue,right,inner sep=2pt] at (-44:1.28) {$\gamma$};
 \node at (0,2.05) {$\Sigma$};
 \node at (0,-2.03) {$\gamma(0)=\gamma(1)=\xx$};
\end{scope}
\node at (0,-2.73) {$\chi([\gamma])=\bar1\in\ZZ/3\ZZ,
 \qquad \nu(z)=e^{-2\pi\mathrm i/3}z$};
\end{tikzpicture}
\caption{Monodromy in the local $3$-sheeted covering $z\mapsto z^3$.}
\label{fig:cyclic-monodromy}
\end{figure}

\begin{dfn}\label{dfn:nakayama-character}
Let $(\Sigma,\mathcal P)$ be a punctured surface, let $\m:\mathcal P\to\ZZ_{>0}$, and let $r\geq1$ be an integer. A \defn{Nakayama character of order $r$} on $(\Sigma,\mathcal P,\m)$ is a group homomorphism
\[
\chi:H_1(\Sigma\setminus\mathcal P;\ZZ)\longrightarrow\ZZ/r\ZZ
\]
such that $\m(\vv)\chi([\gamma_{\vv}])=\bar1$ for every $\vv\in\mathcal P$, where $\gamma_{\vv}$ is a small clockwise loop around $\vv$.
\end{dfn}

\Needspace{7\baselineskip}
\begin{dfn}\label{def:datum}
A \defn{geometric datum} is a tuple $\mathfrak D=(\Sigma,\mathcal P,\eta,\m,r,\chi)$ with the following properties.
\begin{itemize}
\item $(\Sigma,\mathcal P,\eta)$ is a graded punctured surface of ribbon type.
\item $\m:\mathcal P\to\ZZ_{>0}$ and $r\geq1$ is an integer.
\item $\chi$ is a Nakayama character of order $r$ on $(\Sigma,\mathcal P,\m)$.
\end{itemize}
\end{dfn}

The defining condition of a Nakayama character implies that $\m(\vv)$ is coprime to $r$ and that $\chi([\gamma_{\vv}])$ generates $\ZZ/r\ZZ$ for every $\vv\in\mathcal P$. Since $\mathcal P$ is nonempty, $\chi$ is surjective. The classification of coverings therefore gives a connected cyclic covering of $\Sigma\setminus\mathcal P$ with monodromy $\chi$; see \cite[Theorem 1.38 and Proposition 1.39]{Hat02}. This covering extends over each puncture $\vv$ by adding a single point $\widetilde\vv$ above it, with local model $z\mapsto z^r$; see \cite[Sections 4--5]{For81}. We fix such an extension and denote it by
\[
\pi_\Sigma:(\widetilde\Sigma,\widetilde{\mathcal P})
\longrightarrow(\Sigma,\mathcal P).
\]
For $r=1$, we take the identity covering.

Let $\U$ be a graded admissible arc system on $(\Sigma,\mathcal P,\eta)$. For each $\gamma\in\U$, choose a lift $\widetilde\gamma$ by lifting its interior, extending to the unique preimages of its endpoints, and reparametrizing smoothly near these endpoints. All lifts of $\gamma$ are
$\nu^s\widetilde\gamma$, for all $s\in\ZZ/r\ZZ$,
where $\nu$ is the deck generator corresponding to $\bar1\in\ZZ/r\ZZ$.

To assign a weight to a successor arrow $a:X_\gamma\to X_\delta$, choose a small circle around its puncture, meeting each incident branch once. Follow the circle clockwise from the branch of $\gamma$ specified by $a$ to the successive branch of $\delta$; at a valency-one puncture, take the full clockwise circuit. Lift this path starting on the corresponding branch of $\widetilde\gamma$. Define
\[
\weight(a)=s\in\ZZ/r\ZZ
\quad\Longleftrightarrow\quad
\text{the lifted path ends on }\nu^s\widetilde\delta.
\]
The starting and ending points lie on the inverse image of the small circle.
For example, in \cref{fig:successor-weight}, the chosen lifts are $\widetilde\gamma$ and $\widetilde\delta$. The clockwise path representing $a$ lifts from $\widetilde\gamma$ to $\nu\widetilde\delta$, so $\weight(a)=\bar1\in\ZZ/2\ZZ$.

Extend $\weight$ additively to paths and assign weight $\bar0$ to every trivial path. The cohomological degree and the group weight are separate gradings.

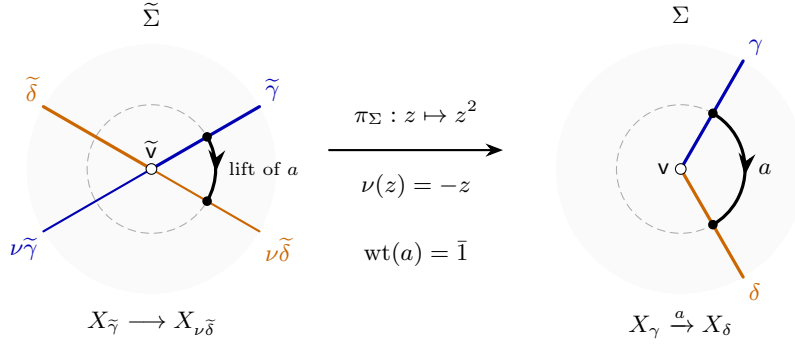
\begin{figure}[htbp]
\centering
\begin{tikzpicture}[>=Stealth,line cap=round,line join=round,font=\small]
\begin{scope}[shift={(-3.5,0)}]
 \fill[gray!4] (0,0) circle (1.65);
 \draw[gray!65,densely dashed] (0,0) circle (.85);
 % Local portions of the chosen lifts and their translates.
 \draw[darkblue,very thick] (0,0)--(30:1.65);
 \draw[orange!80!black,very thick] (0,0)--(150:1.65);
 \draw[darkblue,thick] (0,0)--(210:1.65);
 \draw[orange!80!black,thick] (0,0)--(330:1.65);
 \node[darkblue,above right,inner sep=2pt] at (30:1.65) {$\widetilde\gamma$};
 \node[orange!80!black,above left,inner sep=2pt] at (150:1.65) {$\widetilde\delta$};
 \node[darkblue,below left,inner sep=2pt] at (210:1.65) {$\nu\widetilde\gamma$};
 \node[orange!80!black,below right,inner sep=2pt] at (330:1.65) {$\nu\widetilde\delta$};
 % The lifted clockwise corner, from gamma-tilde to nu delta-tilde.
 \draw[black,very thick,postaction={decorate},
  decoration={markings,mark=at position .55 with {\arrow{Stealth}}}]
  (30:.85) arc[start angle=30,end angle=-30,radius=.85];
 \fill (30:.85) circle (1.8pt);
 \fill (-30:.85) circle (1.8pt);
 \node[right,inner sep=3pt,font=\scriptsize] at (.91,0) {lift of $a$};
 \filldraw[fill=white,draw=black] (0,0) circle (2pt);
 \node[above,inner sep=4pt] at (0,0) {$\widetilde\vv$};
 \node at (0,2.04) {$\widetilde\Sigma$};
 \node at (0,-2.04) {$X_{\widetilde\gamma}\longrightarrow X_{\nu\widetilde\delta}$};
\end{scope}

\draw[->,thick] (-1.15,.25)--(1.15,.25)
 node[midway,above=6pt] {$\pi_\Sigma:z\mapsto z^2$};
\node at (0,-.22) {$\nu(z)=-z$};
\node at (0,-1.12) {$\weight(a)=\bar1$};

\begin{scope}[shift={(3.5,0)}]
 \fill[gray!4] (0,0) circle (1.65);
 \draw[gray!65,densely dashed] (0,0) circle (.85);
 \draw[darkblue,very thick] (0,0)--(60:1.65);
 \draw[orange!80!black,very thick] (0,0)--(-60:1.65);
 \node[darkblue,above right,inner sep=2pt] at (60:1.65) {$\gamma$};
 \node[orange!80!black,below right,inner sep=2pt] at (-60:1.65) {$\delta$};
 % The clockwise successor corner downstairs.
 \draw[black,very thick,postaction={decorate},
  decoration={markings,mark=at position .55 with {\arrow{Stealth}}}]
  (60:.85) arc[start angle=60,end angle=-60,radius=.85];
 \fill (60:.85) circle (1.8pt);
 \fill (-60:.85) circle (1.8pt);
 \node[right,inner sep=4pt] at (.88,0) {$a$};
 \filldraw[fill=white,draw=black] (0,0) circle (2pt);
 \node[left,inner sep=4pt] at (0,0) {$\vv$};
 \node at (0,2.04) {$\Sigma$};
 \node at (0,-2.04) {$X_\gamma\xrightarrow{a}X_\delta$};
\end{scope}
\end{tikzpicture}
\caption{Local lifting of the successor corner $a:X_\gamma\to X_\delta$ under $\pi_\Sigma(z)=z^2$, with $\nu(z)=-z$. }
\label{fig:successor-weight}
\end{figure}

\begin{lem}\label{lem:weights}
The path weights define decompositions into graded subspaces
\[
 \Hom_{\B(\U,\m)}(X_\gamma,X_\delta)
 =\bigoplus_{s\in\ZZ/r\ZZ}
   \Hom_{\B(\U,\m)}^s(X_\gamma,X_\delta)
\]
such that the units have weight $\bar0$ and every operation satisfies \eqref{eq:homogeneous}.
\end{lem}

\begin{proof}
For a half-edge $h$ at $\vv$, one traversal of $C_h$ follows a clockwise meridian, so $\weight(C_h)=\chi([\gamma_{\vv}])$. Since $\chi$ is a Nakayama character, $\weight(C_h^{\m(\vv)})=\bar1$. Thus the two paths in every signed cycle relation have the same weight. Therefore, the quotient algebra is graded by $\ZZ/r\ZZ$. 

If $(a_n,\ldots,a_1)$ is a disc sequence, the corner paths, joined along the truncated sides, bound the marked disc away from the punctures. This loop has trivial monodromy, giving
\[
 \sum_{i=1}^n\weight(a_i)=\bar0.
\]
The paths $ba_n,a_{n-1},\ldots,a_1$ therefore have total weight $\weight(b)$, as do $a_n,\ldots,a_2,a_1b$. These are the output weights of the two operations in \eqref{eq:brauer-polygon-operations}.

Note that for every specified nontrivial path $u$ and its complement, we have
$\weight(u^*)=\bar1-\weight(u)$. 
For \eqref{eq:brauer-complement-operation}, fix $2\leq j\leq n$. The replacement of $a_j$ by the two inputs $a_j(ba_j)^*,ba_j$ increases the total weight by $\bar1$, since
\[
 \weight\bigl(a_j(ba_j)^*\bigr)+\weight(ba_j)
 =\weight(a_j)+\bar1.
\]
The full input sequence in that formula omits $a_1$, so its total weight is
\[
 \sum_{i=2}^n\weight(a_i)+\bar1
 =\bar1-\weight(a_1)
 =\weight(a_1^*).
\]
Scalar signs do not affect weights. The remaining higher operations and $\mu^1$ vanish, proving \eqref{eq:homogeneous} for all arities.
\end{proof}

On the complement of the punctures, the \defn{pulled-back line field} is
\[
 (\pi_\Sigma^*\eta)(\xx)
 =(d\pi_\Sigma|_{\xx})^{-1}\bigl(\eta(\pi_\Sigma(\xx))\bigr),
 \qquad \xx\in\widetilde\Sigma\setminus\widetilde{\mathcal P}.
\]
The derivative is a linear isomorphism of tangent planes there, so it carries each line and each grading path to a unique lift. This is the pullback construction for vector bundles and their sections; see \cite[Section 3]{MS74}.

\begin{thm}\label{thm:construction}
Let $\mathfrak D$ be a geometric datum and let $\U$ be a graded admissible arc system.
\begin{enumerate}[label=\textup{(\arabic*)}]
\item There is a unique $A_\infty$-category $\F_{\mathfrak D}(\U)$ with objects $(X_\gamma,s)$, $\gamma\in\U$, $s\in\ZZ/r\ZZ$, and morphism spaces
\[
 \Hom_{\F_{\mathfrak D}(\U)}
       ((X_\gamma,s),(X_\delta,t))
 =\Hom_{\B(\U,\m)}^{t-s}(X_\gamma,X_\delta)
\]
for which the object map $(X_\gamma,s)\mapsto X_\gamma$ and the inclusions on morphism spaces define a strict functor
\[
 \pi:\F_{\mathfrak D}(\U)\longrightarrow\B(\U,\m).
\]
Its operations are the restrictions of those in \cref{dfn:brauer-category}.
\item Translation $\nu(X_\gamma,s)=(X_\gamma,s+\bar1)$, acting identically on the corresponding morphism spaces, is a strict action. For the representatives $(X_\gamma,\bar0)$,
\[
 \bigl(\F_{\mathfrak D}(\U)/\langle\nu\rangle\bigr)_{\mathrm{rep}}
 \cong\B(\U,\m)
\]
as $A_\infty$-categories.
\item Let $\widetilde\U$ consist of all lifts of the arcs of $\U$. Give it the pulled-back gradings and the line field $\widetilde\eta=\pi_\Sigma^*\eta$. At the unique puncture $\widetilde \vv$ over $\vv$, set
$\dG(\widetilde \vv)=\m(\vv)\operatorname{val}_{\Gamma_\U}(\vv)$.
Then $(\Gamma_{\widetilde\U},\dG)$ is an AFBG, with Nakayama automorphism $\nu$ and fractional multiplicity $\m(\vv)/r$ at $\widetilde \vv$. The associated graded algebra of $\F_{\mathfrak D}(\U)$ is
$\Lambda(\Gamma_{\widetilde\U},\dG,\omega_{\widetilde\eta})$.
\end{enumerate}
\end{thm}

\begin{proof}
(1)--(2) Apply \cref{prop:lift,prop:lift-orbit} to \cref{lem:weights}. The Stasheff identities are inherited from \cref{thm:oz-constructions}.

\smallskip\noindent
(3) By \cite[Section 5.1.2]{OZ22}, The lifted arcs are full, and a lift of each cutting path supplies a cutting path at the corresponding puncture. Their gradings and the line field pull back along the local diffeomorphism away from the punctures. The pullback of the local concentric-circle field has the same form near $\widetilde \vv$. A filling graph realizing the ribbon-type class of $\eta$ lifts to a filling graph realizing the class of $\widetilde\eta$.

There are $r\operatorname{val}_{\Gamma_\U}(\vv)$ half-edges at $\widetilde \vv$. Following $\operatorname{val}_{\Gamma_\U}(\vv)$ successors takes a lifted half-edge to its translate by $\nu^{\chi([\gamma_{\vv}])}$. Hence the defining condition of the Nakayama character gives $\rho^{\dG(\widetilde \vv)}h=\nu h$.
The deck transformation commutes with edge reversal. Moreover, no nontrivial deck transformation preserves a lifted arc. It follows that $\nu^j h\ne\iota h$, proving the AFBG conditions.

The objects $(X_\gamma,s)$ correspond to the arcs $\nu^s\widetilde\gamma$. A successor arrow of $Q_{\Gamma_\U}$ lifts uniquely from each such arc, and the lifted arrows are exactly those of $Q_{\Gamma_{\widetilde\U}}$. The two maximal paths starting at a lifted edge end at its translate by $\nu$; their signed relation is the lift of the corresponding cycle relation downstairs. The zero relations lift to the nonsuccessor pairs. Finally, winding numbers are preserved under lifting,
$\omega_{\widetilde\eta}(\nu^s\widetilde\gamma)
 =\omega_\eta(\gamma)$, and corner degrees are preserved by the definition using paths across sectors. Thus the lifted presentation is precisely $\Lambda(\Gamma_{\widetilde\U},\dG,\omega_{\widetilde\eta})$.
\end{proof}

The term Nakayama character reflects the identity $\weight(C_h^{\m(\vv)})=\bar1$ for a half-edge $h$ at $\vv$: each lifted maximal successor path ends at the image of its starting object under $\nu$, which is the Nakayama automorphism of the lifted AFBG. These paths need not be cycles. The higher operations on $\F_{\mathfrak D}(\U)$ are specified by lifting those of $\B(\U,\m)$.

Note that the choices of distinguished arc lifts only relabel the objects. Indeed, if $\widetilde\gamma$ is replaced by $\nu^{b_\gamma}\widetilde\gamma$, then an arrow $a:X_\gamma\to X_\delta$ has new weight
\[
 \weight'(a)=b_\gamma+\weight(a)-b_\delta.
\]
The map $(X_\gamma,s)\mapsto(X_\gamma,s+b_\gamma)$, with identity maps on the underlying paths, is a strict isomorphism from the newly labelled category to $\F_{\mathfrak D}(\U)$.

Finally, we realize the quotient $\Gamma\to\Gamma_{\red}$ as a cyclic covering of ribbon surfaces and show that the construction in \cref{thm:construction} recovers the original AFBGA.

\begin{prop}\label{prop:nakayama-surface-cover}
Let $(\Gamma,\dG)$ be an AFBG with reduced form $(\Gamma_{\red},\m)$, and let $r=\operatorname{ord}(\nu)$. The graph quotient extends to a map
\[
 \pi_\Sigma:
 (\Sigma_\Gamma,V(\Gamma))
 \longrightarrow(\Sigma_{\Gamma_{\red}},V(\Gamma_{\red}))
\]
which is an $r$-sheeted cyclic covering, fully ramified at every vertex, with deck group $\langle\nu\rangle$. Its monodromy $\chi$ is a Nakayama character of order $r$ on $(\Sigma_{\Gamma_{\red}},V(\Gamma_{\red}),\m)$. The canonical line fields can be chosen so that
\[
 \eta_\Gamma=\pi_\Sigma^*\eta_{\Gamma_{\red}}
 \quad\text{on }\Sigma_\Gamma\setminus V(\Gamma).
\]
Take $\mathfrak D=(\Sigma_{\Gamma_{\red}},V(\Gamma_{\red}),\eta_{\Gamma_{\red}},\m,r,\chi)$ and $\U=E(\Gamma_{\red})$ with the canonical gradings. Then \cref{thm:construction} recovers $\Lambda(\Gamma,\dG)$, that is, the category $\F_{\mathfrak D}(\U)$ is concentrated in degree zero, has $\mu^n=0$ for $n\ne2$, and has associated algebra $\Lambda(\Gamma,\dG)$.
\end{prop}

\begin{proof}
Realize a vertex $\vv$ by a disc with equally spaced half-edge attachments. By \cref{prop:reduction}, the action of $\nu=\rho^{\dG(\vv)}$ is the clockwise rotation through $2\pi\m(\vv)/r$. It has order $r$ and acts freely away from the centre. The action extends over the edge bands because $\nu\iota=\iota\nu$, and no band is stabilized by a nontrivial power of $\nu$. The quotient discs and bands are exactly the ribbon surface of $\Gamma_{\red}$. Locally at a vertex the quotient has the form $z\mapsto z^r$, and elsewhere it is an unbranched covering.

A small clockwise loop around $\vv$ downstairs moves a lifted half-edge by $\rho^{o(\vv)}$. Since $\dG(\vv)=\m(\vv)o(\vv)$ and $\operatorname{val}_\Gamma(\vv)=r\,o(\vv)$, we have
\[
 \rho^{o(\vv)}|_{H_{\vv}}=\nu^a|_{H_{\vv}}
 \quad\Longleftrightarrow\quad
 a\m(\vv)\equiv1\pmod r.
\]
Thus $\chi$ is a Nakayama character. Lift the cuts used to construct $\eta_{\Gamma_{\red}}$ in \cref{sec:surface-definitions}; the local fields lift to the same construction on the vertex discs of $\Gamma$. This gives the stated representatives of the canonical line fields.
The lifts of the arcs in $E(\Gamma_{\red})$ form the arc system $E(\Gamma)$, with its canonical gradings. The final assertion follows from \cref{thm:oz-constructions,thm:construction}.
\end{proof}

In \cref{sec:derived}, we compare the lifted categories associated with different graded admissible arc systems for the same geometric datum.

\subsection{An example of a lifted admissible fractional Brauer graph \texorpdfstring{$A_\infty$}{A-infinity}-category}
\label{sec:triangle-double-cover}

We describe a double covering of a disc with three punctures and the associated $A_\infty$-categories.

\smallskip\noindent
\emph{The punctured disc and the algebra downstairs.}
Let $\Sigma$ be a disc with punctures $\mathcal P=\{\vv_0,\vv_1,\vv_2\}$, and let $\U=\{\delta_0,\delta_1,\delta_2\}$ be the sides of an embedded triangle, listed counterclockwise, with $\vv_i$ the common endpoint of $\delta_i$ and $\delta_{i+1}$. Throughout this subsection, $i$ is read modulo $3$. Choose cutting paths from the punctures to $\partial\Sigma$ outside the triangle, as in \cref{fig:triangle-double-cover}. Cutting along $\U$ gives a triangular disc and an annulus containing $\partial\Sigma$, so $\U$ is admissible.

The ribbon graph $\Gamma_\U$ has vertices $\vv_0,\vv_1,\vv_2$, edges $\delta_0,\delta_1,\delta_2$, and clockwise order $(\delta_i,\delta_{i+1})$ at $\vv_i$; see \cref{fig:triangle-ribbon-downstairs}. Its ribbon surface is an annulus. Filling its inner triangular boundary gives the disc $\Sigma$, with the graph vertices regarded as punctures.

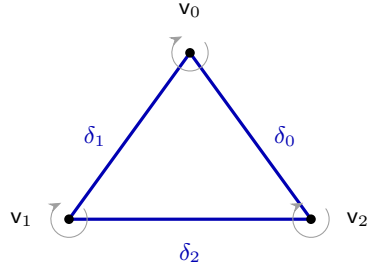
\begin{figure}[htbp]
\centering
\begin{tikzpicture}[>=Stealth,line cap=round,line join=round,font=\small]
 \coordinate (v0) at (0,1.15);
 \coordinate (v1) at (-1.6,-1.05);
 \coordinate (v2) at (1.6,-1.05);
 \draw[very thick,darkblue] (v2)--(v0) node[midway,right=5pt] {$\delta_0$};
 \draw[very thick,darkblue] (v0)--(v1) node[midway,left=5pt] {$\delta_1$};
 \draw[very thick,darkblue] (v1)--(v2) node[midway,below=5pt] {$\delta_2$};
 \foreach \i in {0,1,2}{
  \begin{scope}[shift={(v\i)}]
   \draw[gray!75,->] (35:.24) arc[start angle=35,end angle=-245,radius=.24];
   \fill (0,0) circle (1.8pt);
  \end{scope}
 }
 \node[above=10pt] at (v0) {$\vv_0$};
 \node[left=10pt] at (v1) {$\vv_1$};
 \node[right=10pt] at (v2) {$\vv_2$};
% \node at (0,-1.98) {$\vv_i:\quad(\delta_i,\delta_{i+1})\quad\text{clockwise}$};
\end{tikzpicture}
\caption{The ribbon graph $\Gamma_\U$.}
\label{fig:triangle-ribbon-downstairs}
\end{figure}

The tree $\delta_1\cup\delta_2$ fills $\Sigma$. Let $\eta$ be its canonical line field and set $\m(\vv_i)=1$. The tree edges have winding number zero, while the parity relation for the triangular disc in \cite[Corollary 3.21]{OZ22} gives
\[
\omega_\eta(\delta_0)\equiv1\pmod2,\qquad
\omega_\eta(\delta_1)=\omega_\eta(\delta_2)=0.
\]
Write $X_i=X_{\delta_i}$. In \cref{fig:triangle-quiver-downstairs}, the arrows $a_i$ are the inner corners and $b_i=a_i^*$ are their complementary outer corners. Since the three inner corner degrees sum to $1$, choose the arc gradings so that
\[
|a_0|=1,\qquad |a_1|=|a_2|=0,\qquad |b_i|=-|a_i|.
\]
\begin{figure}[htbp]
\centering
\begin{tikzpicture}[>=Stealth,line cap=round,line join=round,font=\small,
 every node/.style={inner sep=3pt},every path/.style={thick}]
 \node (X0) at (0,1.55) {$X_0$};
 \node (X1) at (-1.8,-1.05) {$X_1$};
 \node (X2) at (1.8,-1.05) {$X_2$};
 \draw[->,darkblue] (X0) to[bend left=12] node[right] {$a_0$} (X1);
 \draw[->,darkblue] (X1) to[bend left=12] node[above] {$a_1$} (X2);
 \draw[->,darkblue] (X2) to[bend left=12] node[left] {$a_2$} (X0);
 \draw[->,red!65!black] (X1) to[bend left=12] node[left] {$b_0$} (X0);
 \draw[->,red!65!black] (X2) to[bend left=12] node[below] {$b_1$} (X1);
 \draw[->,red!65!black] (X0) to[bend left=12] node[right] {$b_2$} (X2);
\end{tikzpicture}
\caption{The quiver of $B(\Gamma_\U,1,\omega_\eta)$.}
\label{fig:triangle-quiver-downstairs}
\end{figure}
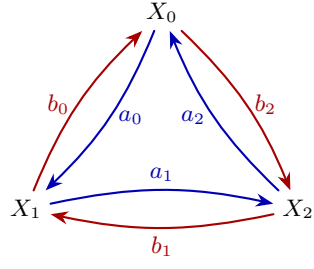

The associated algebra of $\B(\U,1)$ is $B(\Gamma_\U,1,\omega_\eta)$, with quiver shown in \cref{fig:triangle-quiver-downstairs} and relations
\[
\begin{gathered}
a_{i+1}a_i=0,\qquad b_i b_{i+1}=0
\qquad(i\in\ZZ/3\ZZ),\\
b_0a_0=-a_2b_2,\qquad
b_1a_1=a_0b_0,\qquad
b_2a_2=a_1b_1.
\end{gathered}
\]
The inner triangle gives the disc sequence $(a_2,a_1,a_0)$. In particular, \eqref{eq:brauer-polygon-operations} gives
\[
\mu^3(a_2,a_1,a_0)=1_{X_0},\qquad
\mu^3(a_0,a_2,a_1)=1_{X_1},\qquad
\mu^3(a_1,a_0,a_2)=1_{X_2}.
\]

\smallskip\noindent
\emph{The double covering.}
Take $r=2$ and let $\chi:H_1(\Sigma\setminus\mathcal P;\ZZ)\to\ZZ/2\ZZ$ be the Nakayama character sending each small clockwise loop $\gamma_{\vv_i}$ to $\bar1$. The covering
\[
\pi_\Sigma:(\widetilde\Sigma,\widetilde{\mathcal P})
\longrightarrow(\Sigma,\mathcal P)
\]
is obtained by taking two copies of the disc cut along the three cutting paths and gluing the left side of each cut in one copy to the right side in the other copy. The deck generator $\nu$ exchanges the copies. The two copies of each puncture become a single puncture $\widetilde\vv_i$, and the triangular disc has two disjoint lifts.

Choose $\widetilde\delta_0,\widetilde\delta_1,\widetilde\delta_2$ to bound one lifted triangle, and write $\delta_i^s=\nu^s\widetilde\delta_i$ for each $s\in\ZZ/2\ZZ$. All sheet labels are read modulo $2$; see \cref{fig:triangle-double-cover}.

\begin{figure}[htbp]
\centering
\begin{tikzpicture}[>=Stealth,font=\small]
% The upper discs are the cut sheets of the connected covering.
\foreach \s/\figx in {0/-3.25,1/3.25}{
 \begin{scope}[shift={(\figx,2.35)}]
  \ifnum\s=0
   \colorlet{trianglecolor}{darkblue}
  \else
   \colorlet{trianglecolor}{orange!85!black}
  \fi
  \filldraw[fill=gray!5,thick] (0,0) circle (1.82);
  \foreach \j/\ang in {0/90,1/210,2/330}{
   \coordinate (p\j) at (\ang:.92);
   \draw[densely dashed,thick,red!65!black]
     (\ang:.94) -- (\ang:1.82);
   \node[font=\scriptsize,red!65!black]
     at ({\ang+15}:1.43) {$L_{\j}^{\s}$};
   \node[font=\scriptsize,red!65!black]
     at ({\ang-15}:1.43) {$R_{\j}^{\s}$};
  }
  \fill[trianglecolor!8] (p0) -- (p1) -- (p2) -- cycle;
  \draw[very thick,trianglecolor] (p2) -- (p0)
     node[pos=.52,right=2pt] {$\delta_0^{\s}$};
  \draw[very thick,trianglecolor] (p0) -- (p1)
     node[pos=.52,left=2pt] {$\delta_1^{\s}$};
  \draw[very thick,trianglecolor] (p1) -- (p2)
     node[midway,below=3pt] {$\delta_2^{\s}$};
  \foreach \j in {0,1,2}
    \fill (p\j) circle (2pt);
  \draw[->,trianglecolor] (.28,-.12)
    arc[start angle=0,end angle=285,radius=.28];
  \node at (0,-2.16)
    {$\mu^3(a_2^{\s},a_1^{\s},a_0^{\s})=1_{X_0^{\s}}$};
  \node[above] at (0,1.92) {cut copy $\s$};
 \end{scope}
}
\node[align=center,text width=3.8cm] at (0,2.38)
  {$L_j^0\sim R_j^1$\\$R_j^0\sim L_j^1$\\[3pt]$j=0,1,2$};
\draw[->,gray!75,thick] (-2.65,-.17) -- (-1.36,-.81);
\draw[->,gray!75,thick] (2.65,-.17) -- (1.36,-.81)
  node[midway,right=3pt] {$\pi_\Sigma$};
\begin{scope}[shift={(0,-2.1)}]
 \filldraw[fill=gray!5,thick] (0,0) circle (1.82);
 \foreach \j/\ang in {0/90,1/210,2/330}{
  \coordinate (q\j) at (\ang:.92);
  \draw[densely dashed,thick,red!65!black]
    (\ang:.94) -- (\ang:1.82);
 }
 \fill[darkblue!8] (q0) -- (q1) -- (q2) -- cycle;
 \draw[very thick,darkblue] (q2) -- (q0)
   node[pos=.5,right=2pt] {$\delta_0$};
 \draw[very thick,darkblue] (q0) -- (q1)
   node[pos=.5,left=2pt] {$\delta_1$};
 \draw[very thick,darkblue] (q1) -- (q2)
   node[midway,below=3pt] {$\delta_2$};
 \fill (q0) circle (2pt) node[above right=1pt] {$\vv_0$};
 \fill (q1) circle (2pt) node[left=4pt] {$\vv_1$};
 \fill (q2) circle (2pt) node[right=4pt] {$\vv_2$};
 \draw[->,darkblue] (.28,-.12)
   arc[start angle=0,end angle=285,radius=.28];
 \node at (0,-2.16)
   {$\mu^3(a_2,a_1,a_0)=1_{X_0}$};
\end{scope}
\end{tikzpicture}
\caption{The disc $\Sigma$ and the construction of its double covering $\widetilde\Sigma$.}
\label{fig:triangle-double-cover}
\end{figure}
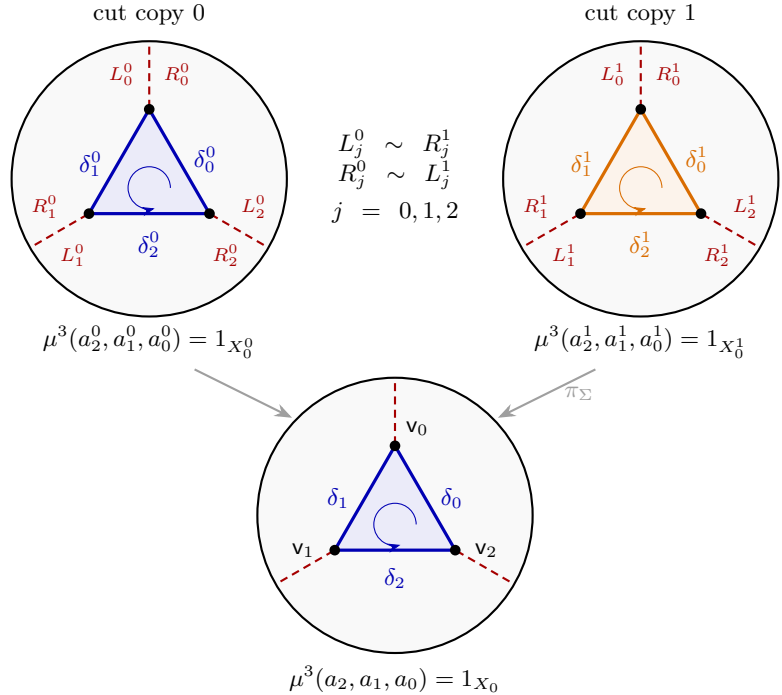

The inner corners stay on the same sheet, whereas each outer corner crosses one cut. Thus
\[
\weight(a_i)=\bar0,\qquad \weight(b_i)=\bar1.
\]
At $\widetilde\vv_i$, the clockwise order of the four incident branches is
$\delta_i^0$, $\delta_{i+1}^0$, $\delta_i^1$ and $\delta_{i+1}^1$.
Consequently, setting $\dG(\widetilde\vv_i)=2$ gives an AFBG $(\Gamma_{\widetilde\U},\dG)$ whose Nakayama automorphism exchanges the sheets. Its fractional multiplicities are $1/2$, and its reduced form is the triangle with multiplicities one. Put $\widetilde\eta=\pi_\Sigma^*\eta$ and $\mathfrak D=(\Sigma,\mathcal P,\eta,\m,2,\chi)$, and give the lifted arcs their pulled-back gradings.

\Cref{fig:triangle-ribbon-upstairs} shows the lifted ribbon graph. The clockwise orders give two boundary walks $(\delta_0^s,\delta_1^s,\delta_2^s)$, one for each $s=0,1$, and a third boundary walk
$(\delta_0^0,\delta_2^1,\delta_1^0, \delta_0^1,\delta_2^0,\delta_1^1)$.
Thus its ribbon surface has three boundary components. Since the graph has three vertices and six edges, the genus $g$ of this ribbon surface satisfies
$g=\frac{2-3-(3-6)}{2}=1$.

Filling the two triangular boundaries gives $\widetilde\Sigma$, of genus one with one boundary component and three marked punctures. The original surface $\Sigma$ has genus zero. In \cref{fig:triangle-double-cover}, the two upper discs are the cut copies used to construct $\widetilde\Sigma$; their crosswise identifications produce the connected surface just described.

\begin{figure}[htbp]
\centering
\begin{tikzpicture}[>=Stealth,line cap=round,line join=round,font=\small]
 \foreach \i/\x in {0/-3,1/0,2/3}{
  \coordinate (v\i) at (\x,0);
  \coordinate (u\i) at ({\x-.4243},.4243);
  \coordinate (w\i) at ({\x+.4243},.4243);
  \coordinate (p\i) at ({\x+.4243},-.4243);
  \coordinate (q\i) at ({\x-.4243},-.4243);
 }
 % Sheet-zero edges: the port order at every vertex is NW, NE, SE, SW.
 \draw[very thick,darkblue] (u0) .. controls (-5,2.2) and (5,2.2) .. (w2)
  node[pos=.5,above=3pt] {$\delta_0^0$};
 \draw[very thick,darkblue] (w0) .. controls (-1.9,1.1) and (-1.1,1.1) .. (u1)
  node[pos=.5,above=3pt] {$\delta_1^0$};
 \draw[very thick,darkblue] (w1) .. controls (1.1,1.1) and (1.9,1.1) .. (u2)
  node[pos=.5,above=3pt] {$\delta_2^0$};
 % White underlays distinguish edge crossings from vertices.
 \draw[very thick,orange!85!black] (q0) .. controls (-5,-2.1) and (1.8,-1.7) .. (p1)
  node[pos=.32,below=4pt] {$\delta_1^1$};
 \draw[very thick,orange!85!black,preaction={draw=white,line width=4.8pt}]
  (q1) .. controls (-1.8,-1.7) and (5,-2.1) .. (p2)
  node[pos=.68,below=4pt] {$\delta_2^1$};
 \draw[very thick,orange!85!black,preaction={draw=white,line width=4.8pt}]
  (p0) .. controls (0,-3) and (0,-3) .. (q2)
  node[pos=.5,below=4pt] {$\delta_0^1$};
 \foreach \i in {0,1,2}{
  \draw[very thick,darkblue] (u\i)--(v\i)--(w\i);
  \draw[very thick,orange!85!black] (p\i)--(v\i)--(q\i);
  \begin{scope}[shift={(v\i)}]
   \draw[gray!80,->] (25:.25) arc[start angle=25,end angle=-255,radius=.25];
   \fill (0,0) circle (1.8pt);
   \node[above=15pt] {$\widetilde\vv_{\i}$};
  \end{scope}
 }
 %\node at (0,-3.65) {$\widetilde\vv_i:\quad
   %(\delta_i^0,\delta_{i+1}^0,\delta_i^1,\delta_{i+1}^1)
   %\quad\text{clockwise}$};
\end{tikzpicture}
\caption{The ribbon graph $\Gamma_{\widetilde\U}$, with $\dG(\widetilde\vv_i)=2$.}
\label{fig:triangle-ribbon-upstairs}
\end{figure}

\smallskip\noindent
\emph{The algebra upstairs.}
By \cref{thm:construction}, the associated algebra of $\F_{\mathfrak D}(\U)$ is $\Lambda(\Gamma_{\widetilde\U},\dG,\omega_{\widetilde\eta})$, with quiver shown in \cref{fig:triangle-quiver-upstairs}. Here $X_i^s=(X_i,s)$, and the lifted arrows retain their degrees. Its relations are
\[
\begin{gathered}
a_{i+1}^s a_i^s=0,\qquad
b_i^{s+1}b_{i+1}^s=0,\\
b_i^s a_i^s
=(-1)^{\omega_\eta(\delta_i)}
 a_{i-1}^{s+1}b_{i-1}^s
\qquad(i\in\ZZ/3\ZZ,\ s\in\ZZ/2\ZZ).
\end{gathered}
\]

\begin{figure}
\centering
\begin{tikzpicture}[>=Stealth,line cap=round,line join=round,font=\small,
 vertex/.style={inner sep=3pt,fill=white},
 arrowlabel/.style={fill=white,inner sep=1.5pt},
 crossing/.style={preaction={draw=white,line width=4pt}},
 every path/.style={thick}]
 \node[vertex] (X00) at (-3,1.4) {$X_0^0$};
 \node[vertex] (X10) at (-4.3,-1) {$X_1^0$};
 \node[vertex] (X20) at (-1.7,-1) {$X_2^0$};
 \node[vertex] (X01) at (3,1.4) {$X_0^1$};
 \node[vertex] (X11) at (1.7,-1) {$X_1^1$};
 \node[vertex] (X21) at (4.3,-1) {$X_2^1$};
 % Inner corners: one directed triangle on each sheet.
 \draw[->,darkblue] (X00)--node[left=3pt] {$a_0^0$}(X10);
 \draw[->,darkblue] (X10)--node[below=3pt] {$a_1^0$}(X20);
 \draw[->,darkblue] (X20)--node[left=3pt] {$a_2^0$}(X00);
 \draw[->,orange!85!black] (X01)--node[right=3pt] {$a_0^1$}(X11);
 \draw[->,orange!85!black] (X11)--node[below=3pt] {$a_1^1$}(X21);
 \draw[->,orange!85!black] (X21)--node[right=3pt] {$a_2^1$}(X01);
 % Outer corners: all six arrows change sheets.
 \draw[->,red!65!black] (X10) .. controls (-5,6) and (0,3.5) .. (X01)
   node[pos=.35,above=3pt,arrowlabel] {$b_0^0$};
 \draw[->,red!65!black,crossing] (X00) .. controls (0,3.5) and (5,6) .. (X21)
   node[pos=.65,above=3pt,arrowlabel] {$b_2^0$};
 \draw[->,red!65!black] (X21) .. controls (4.3,-3) and (-4.3,-3) .. (X10)
   node[pos=.5,below=3pt,arrowlabel] {$b_1^1$};
 \draw[->,red!65!black] (X20)--node[below=3pt,arrowlabel] {$b_1^0$}(X11);
 \draw[->,red!65!black,crossing] (X11) .. controls (1,1.2) and (-1,1.8) .. (X00)
   node[pos=.72,above=3pt,arrowlabel] {$b_0^1$};
 \draw[->,red!65!black,crossing] (X01) .. controls (1,1.8) and (-1,1.2) .. (X20)
   node[pos=.28,above=3pt,arrowlabel] {$b_2^1$};
\end{tikzpicture}
\caption{The lifted quiver. The blue and orange triangles correspond to the two lifted triangles in \cref{fig:triangle-double-cover}; the dark red arrows change sheets.}
\label{fig:triangle-quiver-upstairs}
\end{figure}
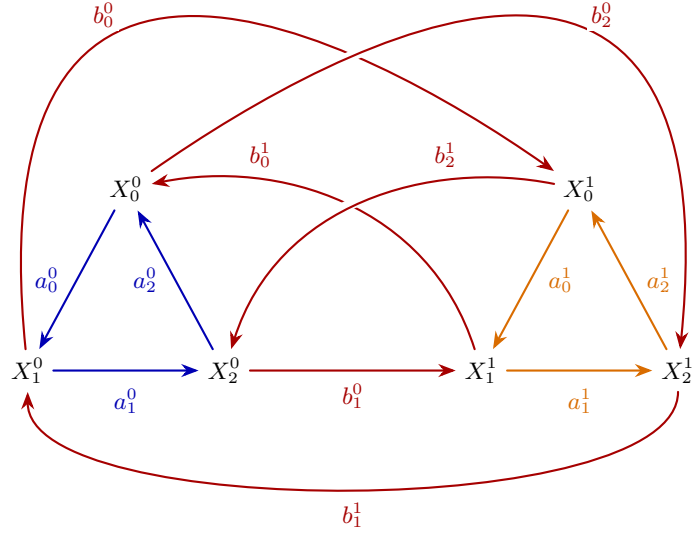

The signs in these relations can be removed by rescaling the arrows. For $i=0,1,2$ and $s=0,1$, put $\widehat a_i^s=(-1)^s a_i^s$ and $\widehat b_i^s=(-1)^i b_i^s$. The relations between maximal paths become
$$\widehat b_i^s\widehat a_i^s=\widehat a_{i-1}^{s+1}\widehat b_{i-1}^s.$$
Thus the associated algebra is isomorphic to the ordinary AFBGA $\Lambda(\Gamma_{\widetilde\U},\dG)$. The top and socle of the indecomposable right projective at $X_i^s$ are supported at $X_i^s$ and $X_i^{s+1}$, respectively. Hence the algebra is not symmetric and is not a Brauer graph algebra.

\smallskip\noindent
\emph{The lifted higher operations.}
We use the original arrows $a_i^s,b_i^s$ for the covering formulas. The strict functor of \cref{thm:construction} is
\[
\pi:\F_{\mathfrak D}(\U)\longrightarrow\B(\U,1),
\qquad X_i^s\longmapsto X_i,\quad
a_i^s\longmapsto a_i,\quad b_i^s\longmapsto b_i.
\]
The two lifted triangles give
\[
\begin{aligned}
\mu^3(a_2^s,a_1^s,a_0^s)&=1_{X_0^s},\\
\mu^3(a_0^s,a_2^s,a_1^s)&=1_{X_1^s},\\
\mu^3(a_1^s,a_0^s,a_2^s)&=1_{X_2^s}
\qquad(s\in\ZZ/2\ZZ).
\end{aligned}
\]
Operations involving outer corners can change sheets. The following formulas illustrate the three rules in \cref{eq:brauer-polygon-operations,eq:brauer-complement-operation}:
\[
\begin{array}{c|c}
\B(\U,1)&\F_{\mathfrak D}(\U)\\[3pt]\hline
\mu^3(b_2a_2,a_1,a_0)=b_2
 &\mu^3(b_2^s a_2^s,a_1^s,a_0^s)=b_2^s\\[3pt]
\mu^3(a_2,a_1,a_0b_0)=-b_0
 &\mu^3(a_2^{s+1},a_1^{s+1},a_0^{s+1}b_0^s)=-b_0^s\\[3pt]
\mu^3(a_2,a_1b_1,a_1)=-b_0
 &\mu^3(a_2^{s+1},a_1^{s+1}b_1^s,a_1^s)=-b_0^s.
\end{array}
\]
The first row starts at $X_0^s$ and ends at $X_2^{s+1}$; the other two start at $X_1^s$ and end at $X_0^{s+1}$. Applying $\pi$ to each identity in the right column gives the corresponding identity in the left column. All higher operations are determined by lifting the formulas of \cref{dfn:brauer-category} in this way.
Finally, $\nu$ sends $X_i^s,a_i^s,b_i^s$ to $X_i^{s+1},a_i^{s+1},b_i^{s+1}$, respectively, and
\[
\bigl(\F_{\mathfrak D}(\U)/\langle\nu\rangle\bigr)_{\mathrm{rep}}
\cong\B(\U,1).
\]

%\raggedbottom
\section{Elementary moves and derived equivalences}
\label{sec:derived}

We compare the categories of \cref{thm:construction} by adding and deleting arcs on the same surface. The relation associated with a marked disc in \cref{prop:brauer-polygon} lifts to every sheet of the covering and supplies the required twisted complexes. We then allow a change of surface that preserves the line field, the multiplicities, and the Nakayama character.

\subsection{Lifting elementary moves}

Fix a geometric datum $\mathfrak D=(\Sigma,\mathcal P,\eta,\m,r,\chi)$. Following \cite[Definition 6.2]{OZ22}, an \defn{elementary move} adds or deletes one arc, with both arc systems admissible. The gradings on the common arcs are retained; any new arc is given a new grading. For example, a flip replaces the diagonal of a quadrilateral by deleting the old diagonal and then adding the new one, as in \cref{fig:elementary-move}.

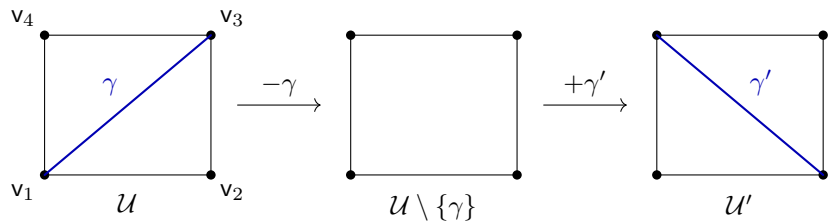
\begin{figure}[htbp]
\centering
\begin{tikzpicture}[scale=.88,line cap=round,line join=round]
\foreach \shift in {0,4.6,9.2}{
 \begin{scope}[xshift=\shift cm]
  \draw (0,0)--(2.5,0)--(2.5,2.1)--(0,2.1)--cycle;
  \foreach \corner in {(0,0),(2.5,0),(2.5,2.1),(0,2.1)}
   \fill \corner circle (2pt);
 \end{scope}
}
\draw[blue!70!black,thick] (0,0)--node[above left]{$\gamma$}(2.5,2.1);
\draw[blue!70!black,thick] (9.2,2.1)--node[above right]{$\gamma'$}(11.7,0);
\draw[->] (2.9,1.05)--node[above]{$-\gamma$}(4.15,1.05);
\draw[->] (7.5,1.05)--node[above]{$+\gamma'$}(8.75,1.05);
\node[below] at (1.25,-.15) {$\U$};
\node[below] at (5.85,-.15) {$\U\setminus\{\gamma\}$};
\node[below] at (10.45,-.15) {$\U'$};
\node[below left] at (0,0) {$\vv_1$};
\node[below right] at (2.5,0) {$\vv_2$};
\node[above right] at (2.5,2.1) {$\vv_3$};
\node[above left] at (0,2.1) {$\vv_4$};
\end{tikzpicture}
\caption{Two elementary moves inside a quadrilateral whose interior contains no punctures. }
\label{fig:elementary-move}
\end{figure}

Note that any two admissible arc systems $\U$ and $\U'$ on $(\Sigma,\mathcal P)$ are connected by a sequence of elementary moves. This follows by adapting the argument of \cite[Proposition 6.4]{OZ22} to the arc convention in \cref{dfn:surface-arc}. We next lift the polygon relation used to compare the categories associated with each elementary move.

\begin{lem}\label{lem:polygon}
Let $(a_n,\ldots,a_1)$ be a disc sequence in $\B(\U,\m)$, with $a_i:X_{i-1}\to X_i$ and $X_n=X_0$. For $s\in\ZZ/r\ZZ$, put
\[
\begin{aligned}
 s_0&=s,\qquad s_i=s+\sum_{j=1}^i\weight(a_j)
 &&(1\leq i\leq n),\\
 t_i&=\sum_{j=1}^i|a_j|-(i-1)
 &&(1\leq i<n).
\end{aligned}
\]
Then $s_n=s$, and there is an isomorphism
\[
 (X_0,s)\cong
 \left(\bigoplus_{i=1}^{n-1}(X_i,s_i)[t_i],\delta\right)
 \quad\text{in }H^0\bigl(\Tw\F_{\mathfrak D}(\U)\bigr),
 \qquad \delta_{i,i+1}=a_{i+1},
\]
with all other twisting components zero. 
\end{lem}

\begin{proof}
By \cref{lem:weights}, the sum of the corner weights is zero, so $s_n=s$. Each corner $a_i$ therefore lifts from $(X_{i-1},s_{i-1})$ to $(X_i,s_i)$. The twisting arrows and the two comparison maps in \cref{prop:brauer-polygon} have these source and target objects; in particular, the last corner returns to $(X_0,s)$. Keep their shifts and scalar coefficients unchanged.

The functor $\Tw\pi$ applies an injective inclusion to every morphism component and commutes with the operations. The Maurer--Cartan equation, the equations saying that the comparison maps are closed, and their two inverse-composition equations map to the corresponding equations in \cref{prop:brauer-polygon}. Injectivity on each component proves the lifted equations. The comparison maps thus give the stated isomorphism.
\end{proof}

\begin{prop}
\label{prop:moves}
Let $\U$ be a graded admissible arc system and let $\gamma\in\U$ be such that $\U\setminus\{\gamma\}$ is full. Choose one lift of each arc in $\U$, and use the same chosen lifts for the arcs in $\U\setminus\{\gamma\}$.
\begin{enumerate}[label=\textup{(\arabic*)}]
\item The natural strict full inclusion
\[
 F:\F_{\mathfrak D}(\U\setminus\{\gamma\})\hookrightarrow\F_{\mathfrak D}(\U)
\]
induces an equivalence $H^0(\Tw F)$. In particular, $F$ is a Morita equivalence.
\item The Morita equivalence class of $\F_{\mathfrak D}(\U)$ is independent of the admissible arc system and of the individual arc gradings.
\end{enumerate}
\end{prop}

\begin{proof}
\noindent
(1) The full inclusion downstairs is given by \cref{prop:brauer-arc-independence}. A corner after deleting $\gamma$ is sent to the path through the same sector, now written as successive corners of $\U$. Lifting that path gives the same endpoint as lifting the original corner. The inclusion therefore preserves weights and restricts to the displayed strict full inclusion.

The polygon used in \cref{prop:brauer-arc-independence} contains $\gamma$ and otherwise only arcs of $\U\setminus\{\gamma\}$. For each $s\in\ZZ/r\ZZ$, \cref{lem:polygon} expresses $(X_\gamma,s)$ as a twisted complex over $\F_{\mathfrak D}(\U\setminus\{\gamma\})$. The fully faithful exact functor $H^0(\Tw F)$ consequently contains every arc object in its essential image. That image is closed under shifts and cones, so it contains every twisted complex. Thus $H^0(\Tw F)$ is an equivalence.

\smallskip\noindent
(2) Apply (1) along the sequence of elementary moves in \cite[Proposition 6.4]{OZ22}. Changing the grading of an arc shifts each of its lifted objects by the same integer, which preserves the category of twisted complexes up to equivalence.
\end{proof}

\subsection{Geometric models and derived equivalences}

Homotopies of line fields preserve the transported intersection degrees, but can change the signs in the modified cycle relations. The following lemma shows that homotopic line fields give Morita equivalent lifted categories.

\begin{lem}
\label{lem:line}
Fix $(\Sigma,\mathcal P,\m,r,\chi)$, and let $\eta_0$ and $\eta_1$ be homotopic line fields of ribbon type. The lifted categories associated with these two line fields are Morita equivalent.
\end{lem}

\begin{proof}
Choose a filling ribbon graph $\Gamma$ as in \cref{dfn:ribbon-type}. By \cref{prop:moves}, it suffices to compare the lifted categories associated with $E(\Gamma)$, with gradings transported along a homotopy from $\eta_0$ to $\eta_1$. Since $\Gamma$ fills $\Sigma$, there are no disc sequences, so both Brauer graph $A_\infty$-categories have $\mu^n=0$ for $n\ne2$.
The proof of \cite[Lemma 7.5]{OZ22} gives a graded isomorphism between the two associated modified Brauer graph algebras by fixing the idempotents and rescaling selected arrows at each puncture $\vv$. The required scalars satisfy $\lambda_{\vv}^{\m(\vv)}=\pm1$ and exist because $\Bbbk$ is algebraically closed.

Choose the same lifts of the arcs in $E(\Gamma)$ for both constructions. Multiplying arrows by nonzero scalars preserves the group weight of every path, so the graded isomorphism restricts to the lifted morphism spaces and preserves $\mu^2$ and the units. Since all other operations vanish, it defines a strict isomorphism of the lifted $A_\infty$-categories. The assertion now follows from \cref{prop:moves}.
\end{proof}

We call two geometric data \defn{equivalent} if they have the same $r$ and admit a diffeomorphism satisfying the three conditions below. In particular, the chosen generator $\nu$ of the cyclic covering is retained.

\begin{thm}\label{thm:geometric}
For $i=1,2$, let $\mathfrak D_i=(\Sigma_i,\mathcal P_i,\eta_i,\m_i,r,\chi_i)$ be geometric data, and let $\U_i$ be graded admissible arc systems. Suppose there is an orientation-preserving diffeomorphism
\[
 f:(\Sigma_1,\mathcal P_1)\longrightarrow(\Sigma_2,\mathcal P_2)
\]
such that
\[
\begin{aligned}
 f^*\eta_2&\simeq\eta_1,\\
 \m_2(f(\vv))&=\m_1(\vv)\qquad(\vv\in\mathcal P_1),\\
 \chi_2\circ f_*&=\chi_1.
\end{aligned}
\]
Then $\F_{\mathfrak D_1}(\U_1)$ and $\F_{\mathfrak D_2}(\U_2)$ are Morita equivalent. In particular, there is an exact equivalence
\[
 \D\bigl(\F_{\mathfrak D_1}(\U_1)\bigr)
 \simeq
 \D\bigl(\F_{\mathfrak D_2}(\U_2)\bigr)
\]
restricting to an equivalence of their perfect categories.
\end{thm}

\begin{proof}
The identity $\chi_2\circ f_*=\chi_1$ lifts $f$ to an isomorphism of the cyclic coverings that commutes with $\nu$; see \cite[Theorem 1.38 and Proposition 1.39]{Hat02}. Choose distinguished arc lifts using that isomorphism. Transporting the arcs and their gradings along $f$ preserves sectors, marked discs, complementary paths, degrees, weights, and winding signs. Hence it gives a strict isomorphism between the first category and the category for the transported data on $\Sigma_2$.

Apply \cref{lem:line} to identify the transported line field with $\eta_2$, retaining $\m_2$ and $\chi_2$, and then apply \cref{prop:moves} to change the arc system to $\U_2$. The resulting zigzag consists of Morita equivalences. The derived and perfect equivalences follow from \cref{prop:morita-criterion}.
\end{proof}

\Needspace{12\baselineskip}
\begin{cor}\label{cor:ordinary}
Under the hypotheses of \cref{thm:geometric}, suppose that both $\F_{\mathfrak D_i}(\U_i)$ are concentrated in degree zero, and let $A_i$ be their associated modified AFBGAs. There is a tilting complex $T\in\K^b(\proj\text{-}A_1)$ such that
\[
 \End_{\K^b(\proj\text{-}A_1)}(T)\cong A_2.
\]
Consequently,
\[
 \D^b(\mod\text{-}A_1)\simeq\D^b(\mod\text{-}A_2).
\]
In particular, degree-zero AFBGAs with equivalent canonical geometric data are derived equivalent.
\end{cor}

\begin{proof}
The existence of the tilting complex and the bounded derived equivalence follow from \cref{thm:geometric} and Rickard's theorem \cite{Ric89}, using $\per(A_i)\simeq\K^b(\proj\text{-}A_i)$. The last assertion follows from \cref{prop:nakayama-surface-cover}.
\end{proof}

In fact, the tilting complex $T$ can be constructed explicitly by following the elementary moves in \cref{thm:geometric} in reverse order. Use the shifted polygon complexes of \cref{lem:polygon} on each sheet to express the target arc objects in terms of the source arc objects. Their direct sum over all target arcs and sheets gives $T\in\per(A_1)$, represented by a bounded complex of finitely generated projective right $A_1$-modules. Changes of arc gradings contribute shifts, while the line-field comparison contributes the scalar isomorphisms of \cref{lem:line}. The following example illustrates this construction.

The following example illustrates this construction.

\begin{eg}
\label{eg:triangle-leaf-tilting}
Consider the arc systems in \cref{fig:triangle-leaf-move}. Let $A_1$ and $A_2$ be the degree-zero AFBGAs associated with the double lifts of $\U_1$ and $\U_2$, respectively, with multiplicity one downstairs, as in \cref{sec:triangle-double-cover}. Write $X_i^s$ for the object corresponding to $\delta_i^s$, where $s\in\ZZ/2\ZZ$.

\begin{figure}[htbp]
\centering
\begin{tikzpicture}[>=Stealth,scale=.78,line cap=round,line join=round,font=\small]
 \foreach \panel/\figx in {0/0,1/6,2/12}{
  \begin{scope}[shift={(\figx,0)}]
   \filldraw[fill=gray!3,draw=gray!55] (0,0) circle (1.85);
   \coordinate (v0) at (0,.3);
   \coordinate (v1) at (-1,-.9);
   \coordinate (v2) at (1,-.9);
   \coordinate (v3) at (0,1.35);
   \ifnum\panel=1
    \fill[orange!9] (v0)--(v1)--(v2)--cycle;
   \fi
   \draw[very thick,darkblue] (v0)--node[left=4pt] {$\delta_1$}(v1);
   \draw[very thick,darkblue] (v0)--node[left=4pt] {$\delta_3$}(v3);
   \ifnum\panel<2
    \draw[very thick,darkblue] (v1)--node[below=4pt] {$\delta_2$}(v2);
   \fi
   \ifnum\panel>0
    \draw[very thick,orange!85!black] (v2)--node[right=4pt] {$\delta_0$}(v0);
   \fi
   \foreach \i in {0,1,2,3}{\fill (v\i) circle (2pt);}
   \node[right=4pt] at (v0) {$\vv_0$};
   \node[left=3pt] at (v1) {$\vv_1$};
   \node[right=3pt] at (v2) {$\vv_2$};
   \node[above=3pt] at (v3) {$\vv_3$};
  \end{scope}
 }
 \draw[->,thick] (2.15,0)--node[above=5pt] {$+\delta_0$}(3.85,0);
 \draw[->,thick] (8.15,0)--node[above=5pt] {$-\delta_2$}(9.85,0);
 \node at (0,-2.2) {$\U_1=\{\delta_1,\delta_2,\delta_3\}$};
 \node at (6,-2.2) {$\U=\{\delta_0,\delta_1,\delta_2,\delta_3\}$};
 \node at (12,-2.2) {$\U_2=\{\delta_0,\delta_1,\delta_3\}$};
\end{tikzpicture}
\caption{Adding $\delta_0$ and then deleting $\delta_2$ changes a path into a star. In the double covering, the corresponding moves add $\delta_0^0,\delta_0^1$ and delete $\delta_2^0,\delta_2^1$, respectively.}
\label{fig:triangle-leaf-move}
\end{figure}
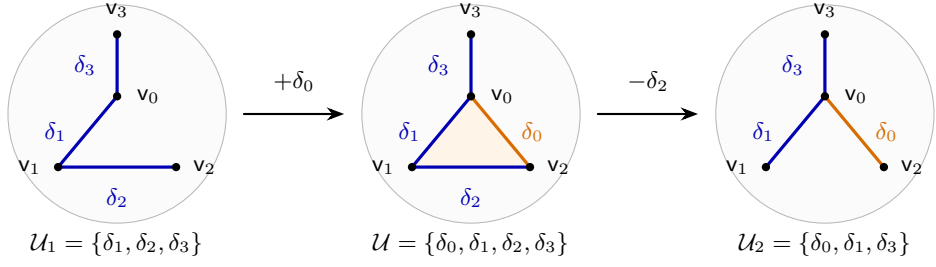

The algebras are $A_1=\Bbbk Q_1/I_1$ and $A_2=\Bbbk Q_2/J^4$, where $J$ is the arrow ideal of $\Bbbk Q_2$. Their quivers are
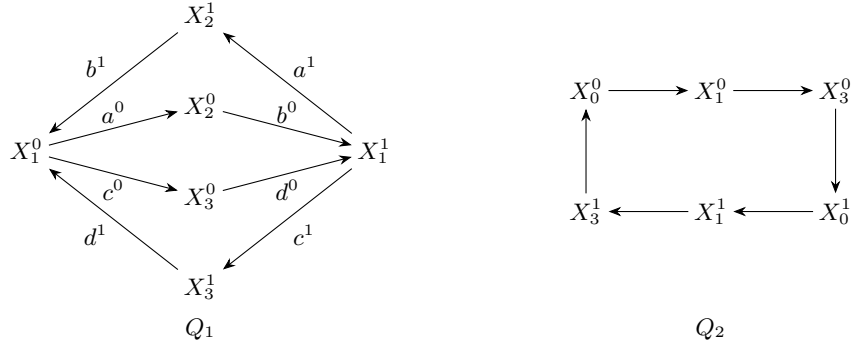
\[
\begin{tikzpicture}[>=Stealth,every node/.style={inner sep=2pt},font=\small]
\begin{scope}[xshift=-3.5cm]
 \node (L) at (-2.3,0) {$X_1^0$};
 \node (R) at (2.3,0) {$X_1^1$};
 \node (A) at (0,.6) {$X_2^0$};
 \node (B) at (0,-.6) {$X_3^0$};
 \node (C) at (0,1.8) {$X_2^1$};
 \node (D) at (0,-1.8) {$X_3^1$};
 \draw[->] (L)--node[above] {$a^0$}(A);
 \draw[->] (A)--node[above] {$b^0$}(R);
 \draw[->] (L)--node[below] {$c^0$}(B);
 \draw[->] (B)--node[below] {$d^0$}(R);
 \draw[->] (R)--node[above right] {$a^1$}(C);
 \draw[->] (C)--node[above left] {$b^1$}(L);
 \draw[->] (R)--node[below right] {$c^1$}(D);
 \draw[->] (D)--node[below left] {$d^1$}(L);
 \node at (0,-2.35) {$Q_1$};
\end{scope}
\begin{scope}[xshift=3.25cm]
 \node (A0) at (-1.65,.8) {$X_0^0$};
 \node (A1) at (0,.8) {$X_1^0$};
 \node (A3) at (1.65,.8) {$X_3^0$};
 \node (B0) at (1.65,-.8) {$X_0^1$};
 \node (B1) at (0,-.8) {$X_1^1$};
 \node (B3) at (-1.65,-.8) {$X_3^1$};
 \draw[->] (A0)--(A1);
 \draw[->] (A1)--(A3);
 \draw[->] (A3)--(B0);
 \draw[->] (B0)--(B1);
 \draw[->] (B1)--(B3);
 \draw[->] (B3)--(A0);
 \node at (0,-2.35) {$Q_2$};
\end{scope}
\end{tikzpicture}
\]
The ideal $I_1$ is generated by $b^s a^s-d^s c^s$, $c^{s+1}b^s$, and $a^{s+1}d^s$ for $s\in\ZZ/2\ZZ$.

By \cref{prop:moves}, the inclusions $\U_1\subset\U\supset\U_2$ induce equivalences
\[
H^0\bigl(\Tw\F_{\mathfrak D}(\U_1)\bigr)
\xrightarrow{\sim}
H^0\bigl(\Tw\F_{\mathfrak D}(\U)\bigr)
\xleftarrow{\sim}
H^0\bigl(\Tw\F_{\mathfrak D}(\U_2)\bigr),
\]
where $\mathfrak D$ is the common geometric datum. Choose the gradings so that the three inner corners along $X_0^s\to X_1^s\to X_2^s\to X_0^s$ have degrees $0,0,1$, respectively. The shifts in \cref{lem:polygon} are then $0$ and $-1$, giving
\[
X_0^s\cong
\bigl(X_1^s\oplus X_2^s[-1],\delta\bigr),
\qquad \delta_{12}=a^s,
\]
in $H^0(\Tw\F_{\mathfrak D}(\U))$, with all other twisting components zero. The right-hand twisted complex belongs to $\Tw\F_{\mathfrak D}(\U_1)$, while the common objects $X_1^s$ and $X_3^s$ are unchanged. Thus the object $\bigoplus_s(X_0^s\oplus X_1^s\oplus X_3^s)$, representing the regular right $A_2$-module, corresponds to
\[
\bigoplus_{s\in\ZZ/2\ZZ}
\left(
\bigl(X_1^s\oplus X_2^s[-1],\delta\bigr)
\oplus X_1^s\oplus X_3^s
\right)
\]
in $H^0(\Tw\F_{\mathfrak D}(\U_1))$.

Under the identification $H^0(\Tw\F_{\mathfrak D}(\U_1))\simeq\K^b(\proj\text{-}A_1)$, the object $X_i^s$ corresponds to the stalk complex $P_i^s=e_i^sA_1$ in degree zero, where $e_i^s$ is the idempotent at $X_i^s$. The twisting component $a^s:X_1^s\to X_2^s[-1]$ therefore gives the complex
$C_s=\bigl(P_1^s\xrightarrow{a^s}P_2^s\bigr)$. Consequently, the inverse image of the regular module $A_2$ is the tilting complex
\[
T=\bigoplus_{s\in\ZZ/2\ZZ}
\bigl(C_s\oplus P_1^s\oplus P_3^s\bigr),
\qquad
\End_{\K^b(\proj\text{-}A_1)}(T)\cong A_2.
\]
This is the right tilting mutation of $A_1$ at $P_2^0\oplus P_2^1$, corresponding to the replacement of the Nakayama orbit $\{\delta_2^0,\delta_2^1\}$ by $\{\delta_0^0,\delta_0^1\}$; see \cite[Section 4.3]{Xin26tt}.
\end{eg}

\section{Combinatorial invariants and derived equivalence}
\label{sec:numerical}

An AFBGA concentrated in degree zero is self-injective, and every basic tilting complex over it is stable under the Nakayama functor; see \cite[Theorem A.4]{Aih13}. This compatibility motivates the comparison of AFBGAs through their reduced ribbon graphs and Nakayama coverings. Following \cite{OZ22}, we seek a classification of derived equivalence by data computed directly from $(\Gamma,\dG)$. We first establish four classes of derived invariants and formulate their completeness as a conjecture. We then prove completeness in reduced genus zero and in reduced genus at least two for non-bipartite graphs.

Throughout this section, all AFBGAs are finite-dimensional and concentrated in degree zero, and their defining ribbon graphs are connected. {\it We exclude the two reduced forms consisting of a single loop of multiplicity one and a single edge joining two vertices both of multiplicity two;} cf.~\cite[Lemma 3.1]{AZ22}.

\subsection{Derived invariants and the completeness conjecture}
\label{sec:derived-invariants}

For $\Lambda=\Lambda(\Gamma,\dG)$, retain the notation
\[
 r=\operatorname{ord}(\nu),\qquad
 \Gamma_{\red}=\Gamma/\langle\nu\rangle,\qquad
 \m(\vv)=\frac{r\dG(\vv)}{\operatorname{val}_\Gamma(\vv)}.
\]
The action of $\nu$ fixes vertices and is free on edges. Thus $V(\Gamma_{\red})=V(\Gamma)$ and $|E(\Gamma_{\red})|=|E(\Gamma)|/r$. Moreover, $\Gamma$ and $\Gamma_{\red}$ are simultaneously bipartite: each edge of the quotient has the same endpoint vertices as any of its lifts.

For a face $F$ of $\Gamma_{\red}$, let $\ell(F)$ be its perimeter. Choose a half-edge $h\in H(\Gamma)$ above its boundary and define $\weight(F)\in\ZZ/r\ZZ$ by
\begin{equation}\label{eq:face-rotation}
 (\rho\iota)^{\ell(F)}h=\nu^{\weight(F)}h.
\end{equation}
Since $\rho\iota$ commutes with $\nu$ and $\nu$ acts freely on half-edges, the residue is independent of the lift and the starting half-edge. We compare faces by the pairs $(\ell(F),\weight(F))$, requiring both entries to agree for corresponding faces.

\begin{lem}\label{lem:rotation-label}
In the canonical covering of \cref{prop:nakayama-surface-cover}, the boundary circle corresponding to $F$, traversed with the surface on its right, has monodromy $\weight(F)$.
\end{lem}
\begin{proof}
Each application of $\rho\iota$ lifts one step of the boundary walk. Equation \eqref{eq:face-rotation} therefore gives the endpoint of the lifted circuit, which defines its monodromy.
\end{proof}

\begin{cond}\label[cond]{cond:four-comparisons}
For $\Lambda_i=\Lambda(\Gamma_i,\dG_i)$, $i=1,2$, consider the following conditions.
\begin{enumerate}[label=\textup{(\arabic*)}]
\item\label{cond:counts}
The Nakayama automorphisms have the same order $r$, and the reduced graphs have the same numbers of vertices, edges and faces.
\item\label{cond:multiplicities}
The multisets of fractional multiplicities $\dG_i(\vv)/\operatorname{val}_{\Gamma_i}(\vv)$, with $\vv\in V(\Gamma_i)$, agree.
\item\label{cond:faces}
The multisets
\[
 \left\{\!\left\{(\ell(F),\weight(F)):F\in F(\Gamma_{i,\red})\right\}\!\right\},
 \qquad i=1,2,
\]
agree, with the cyclic groups identified by $\nu_1\mapsto\nu_2$.
\item\label{cond:bipartite}
Either both graphs are bipartite or neither is.
\end{enumerate}
\end{cond}

With $r$ fixed, condition \textup{(2)} is equivalent to equality of the multisets of reduced multiplicities $\m_i$. We call the genus of $\Sigma_{\Gamma_{\red}}$ the \defn{reduced genus} of $(\Gamma,\dG)$ and denote it by $g_{\red}$. Thus
\[
g_{\red}
=1-\frac12\bigl(
|V(\Gamma_{\red})|-|E(\Gamma_{\red})|+|F(\Gamma_{\red})|
\bigr).
\]
In particular, condition \textup{(1)} makes the reduced genera equal. All genus assumptions below refer to the reduced genus.

\begin{prop}\label{prop:four-derived-invariants}
If $\Lambda_1$ and $\Lambda_2$ are derived equivalent, then they satisfy \cref{cond:four-comparisons}.
\end{prop}

\begin{proof}
If the algebras are local, they are isomorphic, since derived equivalence of local algebras is Morita equivalence; see \cite[Corollary 2.13]{RZ03}. Admissibility gives $r=1$, and the uniqueness of the defining Brauer graph under the stated exclusions proves the assertion; see \cite[Lemma 3.1]{AZ22}. We may therefore assume that the algebras are non-local.

The reduced Brauer graph algebras are derived equivalent by \cite[Theorem 4.1]{Xin26}. Conditions \textup{(1)}, \textup{(2)}, and \textup{(4)} follow from that result and the Brauer graph invariants in \cite[Propositions 4.3 and 4.5]{AZ22}; if the reduced algebras are local, use their isomorphism and \cite[Lemma 3.1]{AZ22}. The equality of Nakayama orders also follows from \cite[Theorem 4.1]{XZ25}, because $\nu$ acts freely on edges. It remains to prove \textup{(3)}. In the representation-infinite case, we recover the perimeter--weight pairs from the shift and Nakayama functors on the stable category. Faces not detected by tubes of rank greater than one are then recovered by counting.

For the remainder of this proof, we work with finite-dimensional left modules. The derived equivalence induces a triangle equivalence between their stable module categories; see \cite{Ric89st}. Write $\mathcal N_\Lambda=D\Lambda\otimes_\Lambda-$ for the Nakayama functor and $\tau_\Lambda$ for the Auslander--Reiten translation. Since $\Lambda$ is self-injective, $[1]\cong\Omega_\Lambda^{-1}$ and $\tau_\Lambda\cong\mathcal N_\Lambda\Omega_\Lambda^2$, so $\mathcal N_\Lambda\cong\tau_\Lambda[2]$. Hence the stable equivalence commutes with both $[1]$ and $\mathcal N_\Lambda$, up to natural isomorphism. It also preserves representation-finiteness.

\smallskip\noindent
\emph{(a) Representation-finite algebras.}
By \cite[Theorem 2.35]{LL26}, each reduced form is a Brauer tree. Its unique face $F$ satisfies
\[
\ell(F)=2|E(\Gamma_{\red})|,
\qquad
\weight(F)=\sum_{\vv\in V(\Gamma_{\red})}\m(\vv)^{-1}
\quad\text{in }\ZZ/r\ZZ.
\]
Indeed, the reduced ribbon surface is a disc, and its clockwise boundary class is the sum of the clockwise puncture meridians. Apply \cref{lem:rotation-label} and the puncture monodromy formula in \cref{prop:nakayama-surface-cover}. These quantities agree by \textup{(1)} and \textup{(2)}.

\smallskip\noindent
\emph{(b) Faces detected by tubes of rank greater than one.}
Suppose that the algebras are representation-infinite. For $h\in H(\Gamma)$, let $M_h$ be the uniserial left module associated with $p_{h,\dG(s(h))-1}$, taking the simple module at $e(h)$ when $\dG(s(h))=1$. These modules are nonprojective and pairwise non-isomorphic, and are precisely the string modules at the mouths of exceptional tubes in the stable Auslander--Reiten quiver; see \cite[Lemma 5.1]{LL26}.

The successor-path relations give a projective-cover sequence
\[
0\longrightarrow M_{\rho\iota h}
\longrightarrow \Lambda e(h)
\longrightarrow M_h
\longrightarrow0.
\]
Together with \cite[Lemma 5.2]{LL26} and $\mathcal N_\Lambda\cong\tau_\Lambda[2]$, this gives
\[
\begin{aligned}
M_h[1]&\cong M_{(\rho\iota)^{-1}h},\qquad
\tau_\Lambda M_h&\cong M_{\nu^{-1}(\rho\iota)^2h},\qquad
\mathcal N_\Lambda M_h&\cong M_{\nu^{-1}h}.
\end{aligned}
\]
Note that every nonprojective indecomposable module over $\Lambda$ is a string or band module, and band modules lie in stable tubes of rank one. Thus every mouth object of a stable tube of rank greater than one is among the $M_h$. The isomorphism classes of these mouth objects are preserved by triangle equivalences.

Let $\mathcal O_\Lambda$ be the set of their joint orbits under $[1]$ and $\mathcal N_\Lambda$. The displayed formulas identify these with the corresponding $\langle\rho\iota,\nu\rangle$-orbits of half-edges, each consisting of all half-edges above a reduced face $F$. Thus $\ell(F)$ is the least positive shift taking $M_h$ into its Nakayama orbit, and $\weight(F)$ records the corresponding power of the Nakayama functor. Explicitly,
\[
\begin{aligned}
\ell(F)
&=\min\bigl\{n>0:
M_h[n]\cong\mathcal N_\Lambda^jM_h
\text{ for some }j\in\ZZ/r\ZZ\bigr\},\\
M_h[\ell(F)]
&\cong\mathcal N_\Lambda^{\weight(F)}M_h.
\end{aligned}
\]
The second equality follows from \eqref{eq:face-rotation}, since
\[
(\rho\iota)^{-\ell(F)}h=\nu^{-\weight(F)}h.
\]
The exponent is unique modulo $r$, because $\nu$ acts freely on half-edges and the nonprojective modules $M_h$ are pairwise non-isomorphic. Consequently, the stable equivalence preserves the pairs $(\ell(F),\weight(F))$ associated with all orbits in $\mathcal O_\Lambda$.

\smallskip\noindent
\emph{(c) The remaining faces.}
It remains to recover the faces whose associated exceptional tubes have rank one. Both $[1]$ and $\mathcal N_\Lambda$ preserve tube ranks, so all mouth objects above a fixed reduced face belong to tubes of the same rank. For a face not recovered in \textup{(b)}, the formula for $\tau_\Lambda M_h$ gives
\[
\tau_\Lambda M_h\cong M_h
\quad\Longleftrightarrow\quad
(\rho\iota)^2h=\nu h.
\]
Modulo $\nu$, this implies that $\ell(F)$ divides $2$. Equation \eqref{eq:face-rotation} therefore gives exactly the possibilities
\[
\ell(F)=1,\quad 2\weight(F)=\bar1,
\qquad\text{or}\qquad
\ell(F)=2,\quad \weight(F)=\bar1.
\]
The first case occurs only when $r$ is odd, and its weight is then unique. In each case, the weight is determined by the perimeter and $r$, so only the numbers of these faces remain to be recovered.

Let $n_1,n_2$ count the unrecovered faces of perimeter one and two. For $O\in\mathcal O_\Lambda$, denote its recovered perimeter by $\ell(O)$. Since the sum of all face perimeters is twice the number of edges,
\[
\begin{aligned}
n_1+n_2
&=|F(\Gamma_{\red})|-|\mathcal O_\Lambda|,\\
n_1+2n_2
&=2|E(\Gamma_{\red})|
-\sum_{O\in\mathcal O_\Lambda}\ell(O).
\end{aligned}
\]
The right-hand sides are preserved by condition \textup{(1)} and step \textup{(b)}. These two equations determine $n_1$ and $n_2$, so all remaining perimeter--weight pairs are preserved. This proves \textup{(3)}.
\end{proof}

We conjecture that these necessary conditions are also sufficient. For $r=1$, the residues vanish and the statement reduces to the derived equivalence classification of Brauer graph algebras in \cite{OZ22}.

\begin{conj}\label[conj]{conj:four-condition-all-genus}
Under the standing assumptions of this section, $\Lambda_1$ and $\Lambda_2$ are derived equivalent if and only if they satisfy \cref{cond:four-comparisons}.
\end{conj}

The conjecture can also be stated on a fixed reduced surface. Fix a reduced Brauer graph $(\Gamma_{\red},\m)$ satisfying our exclusions and a positive integer $r$ coprime to every $\m(\vv)$. For a Nakayama character $\chi$ of order $r$ on $(\Sigma_{\Gamma_{\red}},V(\Gamma_{\red}),\m)$, let $\Lambda_\chi$ be the ordinary algebra obtained by lifting its canonical arcs as in \cref{thm:construction,prop:nakayama-surface-cover}.

\begin{prop}\label{prop:four-condition-reduction}
\Cref{conj:four-condition-all-genus} is equivalent to the following assertion: for every fixed $(\Gamma_{\red},\m,r)$ as above, two Nakayama characters with equal values on every original boundary circle give derived-equivalent algebras $\Lambda_{\chi_0}$ and $\Lambda_{\chi_1}$.
\end{prop}
\begin{proof}
The conjecture implies the assertion because the four conditions agree for these algebras. Conversely, suppose the assertion holds and $\Lambda_1,\Lambda_2$ satisfy the four conditions. The canonical line fields can be identified by an orientation-preserving diffeomorphism matching punctures by multiplicity and boundary circles by $(\ell,\weight)$; this follows from \cite[Lemmas 7.9--7.10]{OZ22} and \cite[Theorem 1.8]{LP20}. Pull both Nakayama characters back to the first reduced surface. They agree on the puncture meridians and original boundary circles. Apply the assertion on $\Gamma_{1,\red}$, followed by \cref{cor:ordinary}, to obtain the required derived equivalence.
\end{proof}

Thus it suffices to fix the reduced Brauer graph and its canonical line field, and prove that changing the Nakayama character while retaining its peripheral values does not change the derived equivalence class.

The following example illustrates this phenomenon.

\begin{eg}\label{eg:character-rose}
Let $\Gamma_{\red}$ have one vertex $\vv$ and two loop edges $\delta_0,\delta_1$. Its half-edges have clockwise order $(h_0,h_1,h_2,h_3)$, with $\iota=(h_0\ h_2)(h_1\ h_3)$. The ribbon surface has genus one and one boundary component $B$. Fix $\m(\vv)=1$, $r=2$, and the canonical line field.

Write $\alpha_j$ for the successor arrow from the edge containing $h_j$ to the edge containing $h_{j+1}$, with subscripts modulo $4$. Choose the following two systems of weights $\weight_0,\weight_1$ in $\ZZ/2\ZZ$:
\[
\begin{array}{c|cccc}
 &\alpha_0&\alpha_1&\alpha_2&\alpha_3\\\hline
\weight_0&0&0&0&1\\
\weight_1&0&1&1&1
\end{array}
\]
Both systems have total weight $1$ around the puncture and induce Nakayama characters $\chi_0,\chi_1$ of order $2$. They also have the same boundary value $\chi_0([B])=\chi_1([B])=1$. Nevertheless, the characters are different, since the closed walk that follows the sector $\alpha_0$ and returns along $\alpha_2$ in the reverse direction has weight $0$ for $\chi_0$ and $1$ for $\chi_1$. Here the two sectors are joined along the corresponding arcs away from the puncture.

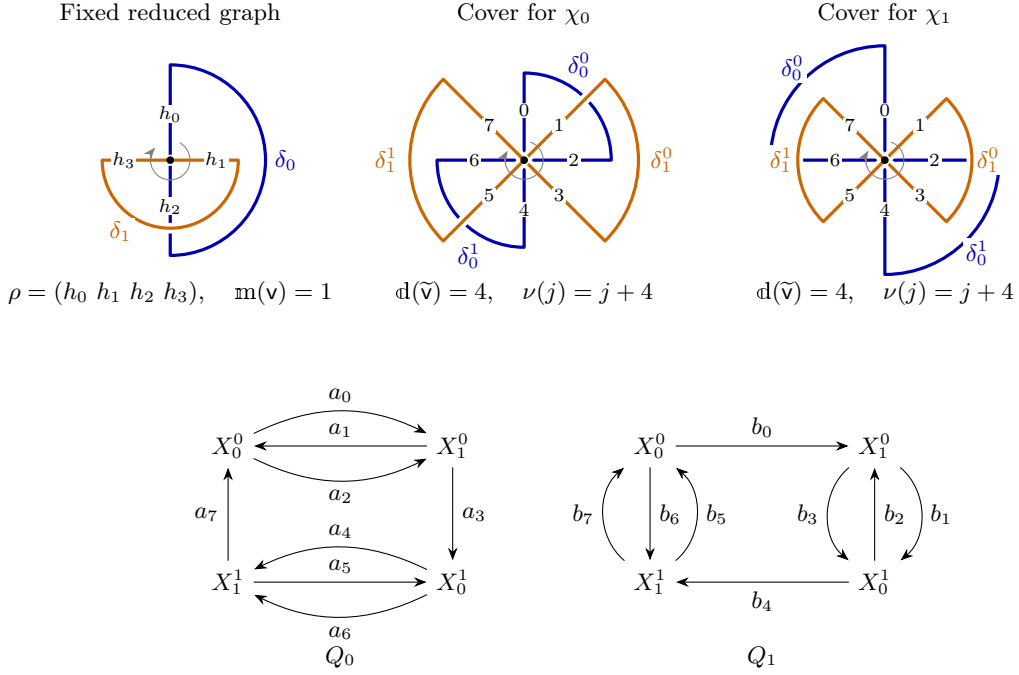
\begin{figure}[htbp]
\centering
\begin{tikzpicture}[>=Stealth,line cap=round,line join=round,font=\small,scale=.9]
% Half-edges run clockwise from the upward ray. Crossings are not vertices.
\begin{scope}[shift={(0,3.2)}]
\draw[very thick,darkblue,preaction={draw=white,line width=4pt}] (0,0)--(90:1.4) arc[start angle=90,end angle=-90,radius=1.4]--(0,0);
\node[darkblue,right] at (0:1.4) {$\delta_0$};
\draw[very thick,orange!80!black,preaction={draw=white,line width=4pt}] (0,0)--(0:1) arc[start angle=0,end angle=-180,radius=1]--(0,0);
\node[orange!80!black,fill=white,inner sep=1pt] at (-.73,-1.04) {$\delta_1$};
\foreach \j in {0,1,2,3}{\node[fill=white,inner sep=.4pt,font=\scriptsize] at ({90-90*\j}:.68) {$h_{\j}$};}
\draw[gray,->] (60:.29) arc[start angle=60,end angle=-220,radius=.29];
\fill (0,0) circle (1.6pt);
\node[above] at (0,1.85) {Fixed reduced graph};
\node at (0,-1.95) {$\rho=(h_0\ h_1\ h_2\ h_3),\quad\m(\vv)=1$};
\end{scope}
\begin{scope}[shift={(5.2,3.2)}]
\foreach \from/\to/\rr/\lab/\col in {0/2/1.28/{\delta_0^0}/darkblue,4/6/1.28/{\delta_0^1}/darkblue,1/3/1.68/{\delta_1^0}/{orange!80!black},5/7/1.68/{\delta_1^1}/{orange!80!black}}{
\draw[very thick,\col,preaction={draw=white,line width=4pt}] (0,0)--({90-45*\from}:\rr) arc[start angle={90-45*\from},end angle={90-45*\to},radius=\rr]--(0,0);
}
\node[darkblue,fill=white,inner sep=1pt] at (60:1.62) {$\delta_0^0$};
\node[darkblue,fill=white,inner sep=1pt] at (240:1.62) {$\delta_0^1$};
\node[orange!80!black,right] at (0:1.72) {$\delta_1^0$};
\node[orange!80!black,left] at (180:1.72) {$\delta_1^1$};
\foreach \j in {0,...,7}{\node[fill=white,inner sep=.5pt,font=\scriptsize] at ({90-45*\j}:.73) {$\j$};}
\draw[gray,->] (70:.28) arc[start angle=70,end angle=-210,radius=.28];
\fill (0,0) circle (1.6pt);
\node[above] at (0,1.85) {Cover for $\chi_0$};
\node at (0,-1.95) {$\dG(\widetilde\vv)=4,\quad\nu(j)=j+4$};
\end{scope}
\begin{scope}[shift={(10.5,3.2)}]
\foreach \from/\to/\rr/\lab/\col in {0/-2/1.68/{\delta_0^0}/darkblue,2/4/1.68/{\delta_0^1}/darkblue,1/3/1.28/{\delta_1^0}/{orange!80!black},5/7/1.28/{\delta_1^1}/{orange!80!black}}{
\draw[very thick,\col,preaction={draw=white,line width=4pt}] (0,0)--({90-45*\from}:\rr) arc[start angle={90-45*\from},end angle={90-45*\to},radius=\rr]--(0,0);
\node[\col,fill=white,inner sep=1pt] at ({90-22.5*(\from+\to)}:{\rr+.24}) {$\lab$};
}
\foreach \j in {0,...,7}{\node[fill=white,inner sep=.5pt,font=\scriptsize] at ({90-45*\j}:.73) {$\j$};}
\draw[gray,->] (70:.28) arc[start angle=70,end angle=-210,radius=.28];
\fill (0,0) circle (1.6pt);
\node[above] at (0,1.85) {Cover for $\chi_1$};
\node at (0,-1.95) {$\dG(\widetilde\vv)=4,\quad\nu(j)=j+4$};
\end{scope}
% Quivers: a_j and b_j go from the edge at j to the edge at j+1.
\begin{scope}[shift={(2.5,-2)}]
\node (A) at (-1.65,1) {$X_0^0$}; \node (B) at (1.65,1) {$X_1^0$};
\node (C) at (1.65,-1) {$X_0^1$}; \node (D) at (-1.65,-1) {$X_1^1$};
\draw[->] (A) to[bend left=26] node[above] {$a_0$} (B);
\draw[->] (B)--node[above] {$a_1$}(A);
\draw[->] (A) to[bend right=26] node[below] {$a_2$} (B);
\draw[->] (B)--node[right] {$a_3$}(C);
\draw[->] (C) to[bend right=26] node[above] {$a_4$} (D);
\draw[->] (D)--node[above] {$a_5$}(C);
\draw[->] (C) to[bend left=26] node[below] {$a_6$} (D);
\draw[->] (D)--node[left] {$a_7$}(A);
\node at (0,-2.1) {$Q_0$};
\end{scope}
\begin{scope}[shift={(8.7,-2)}]
\node (A) at (-1.65,1) {$X_0^0$}; \node (B) at (1.65,1) {$X_1^0$};
\node (C) at (1.65,-1) {$X_0^1$}; \node (D) at (-1.65,-1) {$X_1^1$};
\draw[->] (A)--node[above] {$b_0$}(B);
\draw[->] (B) to[bend left=50] node[right] {$b_1$}(C);
\draw[->] (C)--node[right] {$b_2$}(B);
\draw[->] (B) to[bend right=50] node[left] {$b_3$}(C);
\draw[->] (C)--node[below] {$b_4$}(D);
\draw[->] (D) to[bend right=50] node[right] {$b_5$}(A);
\draw[->] (A)--node[right] {$b_6$}(D);
\draw[->] (D) to[bend left=50] node[left] {$b_7$}(A);
\node at (0,-2.1) {$Q_1$};
\end{scope}
\end{tikzpicture}

\caption{The fixed reduced ribbon graph, the two double covers, and their quivers.}
\label{fig:character-rose}
\end{figure}

The covering graphs have degree function $\dG(\widetilde\vv)=4$. Their edge pairings, in the clockwise numbering of \cref{fig:character-rose}, are
\[
\iota_0=(0\ 2)(1\ 3)(4\ 6)(5\ 7),\qquad
\iota_1=(0\ 6)(1\ 3)(2\ 4)(5\ 7).
\]
Let $A_0=\Bbbk Q_0/I_0$ and $A_1=\Bbbk Q_1/I_1$ be the associated degree-zero AFBGAs. All arrow subscripts in the following relations lie in $\ZZ/8\ZZ$. The ideal $I_0$ is generated by
\[
\begin{gathered}
a_{\iota_0(j+1)}a_j,\\
a_{j+3}a_{j+2}a_{j+1}a_j-
a_{\iota_0(j)+3}a_{\iota_0(j)+2}a_{\iota_0(j)+1}a_{\iota_0(j)},\\
a_{j+4}a_{j+3}a_{j+2}a_{j+1}a_j
\qquad(j\in\ZZ/8\ZZ).
\end{gathered}
\]
The ideal $I_1$ is given by the same formulas with $a$ and $\iota_0$ replaced by $b$ and $\iota_1$. 

Clockwise rotation by one position intertwines the two edge pairings:
\[
\iota_1(j+1)=\iota_0(j)+1\pmod8.
\]
Consequently, the assignments
\[
X_0^s\longmapsto X_1^s,\qquad
X_1^s\longmapsto X_0^{s+1},\qquad
a_j\longmapsto b_{j+1}
\]
preserve all relations and define an algebra isomorphism $A_0\cong A_1$. In particular, the algebras are derived equivalent. This gives a simple instance in which changing the Nakayama character while retaining its peripheral values leaves the derived equivalence class unchanged.
\end{eg}

\subsection{Completeness in genus zero and the non-bipartite case}
\label{sec:completeness}

For the canonical datum $(\Sigma,\mathcal P,\eta,\m,r,\chi)$, the punctures correspond to vertices of $\Gamma_{\red}$. Their clockwise meridians satisfy
\[
 \omega_\eta(\gamma_{\vv})=0,
 \qquad \chi([\gamma_{\vv}])=\m(\vv)^{-1}.
\]
For the original boundary circle $B_F$ corresponding to a face $F$, traversed with the surface on its right, we have
\[
 \omega_\eta(B_F)=-\ell(F),
 \qquad \chi([B_F])=\weight(F).
\]
The first identity is the canonical line-field calculation of \cite[Example 3.16]{OZ22}; the second is \cref{lem:rotation-label}. Moreover, $\eta$ is orientable precisely when $\Gamma_{\red}$ is bipartite, by \cite[Lemma 7.9]{OZ22}. The four conditions therefore match the peripheral winding numbers, the peripheral monodromies, and orientability.

\begin{thm}\label{thm:four-condition-completeness}
Suppose that each $\Lambda_i$ has either reduced genus zero or reduced genus at least two with $\Gamma_i$ non-bipartite. Then \cref{cond:four-comparisons} holds if and only if
\[
 \D^b(\mod\text{-}\Lambda_1)\simeq\D^b(\mod\text{-}\Lambda_2).
\]
\end{thm}

\begin{proof}
Necessity is \cref{prop:four-derived-invariants}. For sufficiency, take the canonical geometric data of \cref{prop:nakayama-surface-cover}. The four conditions give an orientation-preserving diffeomorphism matching punctures by multiplicity and original boundary circles by $(\ell,\weight)$. Use it to identify the two punctured surfaces, and remove small open discs around the punctures. Let $S$ be the resulting compact surface. The line fields $\eta_1,\eta_2$ have equal winding numbers on each boundary circle of $S$, and the Nakayama characters $\chi_1,\chi_2$ agree on all boundary classes.

\smallskip\noindent
\textup{(a) Reduced genus zero.}
By \cite[Theorem 1.8(i)]{LP20}, the line fields are homotopic. Boundary classes generate $H_1(S;\ZZ)$, so $\chi_1=\chi_2$. The chosen diffeomorphism identifies the geometric data, and \cref{cor:ordinary} gives the derived equivalence.

\smallskip\noindent
\textup{(b) Reduced genus at least two and non-bipartite graphs.}
We construct a diffeomorphism identifying the line fields and the Nakayama characters simultaneously. All diffeomorphisms of $S$ used below fix a neighbourhood of its boundary, so they extend over the removed discs.

Write $H_1(S;\ZZ)=V\oplus R$, where $R$ is generated by the boundary classes and $V$ by a standard system of handle curves, so the intersection pairing on $V$ is unimodular and symplectic. By \cite[Proposition 1.6]{LP20}, a line field $\eta$ determines a function
\[
 q_\eta:H_1(S;\ZZ)\longrightarrow\ZZ/4\ZZ,
 \qquad q_\eta([C])=\omega_\eta(C)+2\pmod4
\]
for oriented simple closed curves $C$, satisfying
\[
 q_\eta(x+y)=q_\eta(x)+q_\eta(y)+2(x\mathbin{\cdot}y).
\]
The restrictions $(q_{\eta_i},\chi_i)|_R$ agree and are homomorphisms, since $R$ is the radical of the intersection pairing. Denote their common image by $J$ and set
\[
 K=(\ZZ/4\ZZ\oplus\ZZ/r\ZZ)/J.
\]

\smallskip\noindent
\emph{Match the values modulo peripheral classes.}
At a puncture, $(q_{\eta_i},\chi_i)$ takes the value $(2,a)$, where $a$ is a unit modulo $r$. Choose an odd integer representative of $a$. Then
\[
 (\ZZ/4\ZZ\oplus\ZZ/r\ZZ)/\langle(2,a)\rangle
 \cong\ZZ/\gcd(4,2r)\ZZ,
 \qquad (x,y)\longmapsto ax-2y.
\]
Consequently, $K$ is either $0$, $\ZZ/2\ZZ$, or $\ZZ/4\ZZ$.

If $K=0$, the two functions $V\to K$ already agree. If $K=\ZZ/2\ZZ$, they are the linear forms $q_{\eta_i}\bmod2$. If $K=\ZZ/4\ZZ$, they are the quadratic functions $aq_{\eta_i}-2\chi_i$, with polar form $2(x\mathbin{\cdot}y)$. In either nonzero case, every peripheral value of $q_{\eta_i}$ is even. Since the line fields are nonorientable, $q_{\eta_i}\bmod2$ is therefore nonzero on $V$. The symplectic group is transitive on nonzero linear forms modulo two and on these non-even quadratic functions modulo four; for the latter, use the algebraic calculation in the proof of \cite[Theorem 1.8(iii)]{LP20}. Hence some symplectic automorphism of $V$ identifies the two functions $V\to K$.

\smallskip\noindent
\emph{Match the full characters and the functions $q_{\eta_i}$.}
After this symplectic change, the difference of the pairs $(q_{\eta_i},\chi_i)$ on $V$ is a homomorphism with values in $J$. Since $V$ is free, choose a homomorphism $B:V\to R$ whose peripheral values cancel this difference. Every such change $x\mapsto x+B(x)$ is realized by a diffeomorphism fixing the boundary. Indeed, for a symplectic basis element $v$ of $V$ and a boundary class $c$, choose simple curves representing $v$ and $v+c$ that bound a pair of pants with the corresponding boundary circle. The product of their opposite Dehn twists acts by $x\mapsto x+(v\mathbin{\cdot}x)c$. These transformations generate all the required homomorphisms $B$. Symplectic automorphisms of $V$ are likewise realized by Dehn twists supported on the handles. We thus obtain a diffeomorphism $f$ with
\[
 q_{f^*\eta_2}=q_{\eta_1},
 \qquad \chi_2\circ f_*=\chi_1.
\]

\smallskip\noindent
\emph{Match the line fields.}
The equality $q_{f^*\eta_2}=q_{\eta_1}$ and the equality of the integral boundary winding numbers mean that $f^*\eta_2$ differs from $\eta_1$ by $4u$, where $u\in H^1(S;\ZZ)$ vanishes on $R$. The construction in the proof of \cite[Lemma 1.7]{LP20} removes this difference by products of opposite Dehn twists about homologous curves bounding a subsurface of genus one with two boundary components. Such products act trivially on $H_1(S;\ZZ)$, so they preserve $\chi_1$. After this adjustment, $f^*\eta_2\simeq\eta_1$ and $\chi_2\circ f_*=\chi_1$. Extending $f$ over the punctures and applying \cref{cor:ordinary} proves the assertion.
\end{proof}

Thus, in the cases of \cref{thm:four-condition-completeness}, derived equivalence can be decided from the reduced graph counts, fractional multiplicities, perimeter--weight pairs, and bipartiteness, without constructing a geometric equivalence explicitly. The general sufficiency assertion remains \cref{conj:four-condition-all-genus}.

We conclude this section with an example illustrating the necessity of condition \textup{(3)} in \cref{thm:four-condition-completeness}.

\begin{eg}\label{eg:necessity-face-weights}
Consider the algebras in \cite[Example 3.5 and Remark 4.2]{Xin26}. Let $A$ be the Kronecker algebra and let $A'$ be the radical-square-zero algebra of the oriented cycle with two vertices. Consider
\[
\Lambda_1=T_2(A),
\qquad
\Lambda_2=T_2(A').
\]
Their reduced Brauer graphs are both the graph with two vertices joined by two edges, with multiplicity one at both vertices. The numerical data in conditions \textup{(1)}, \textup{(2)}, and \textup{(4)} of \cref{cond:four-comparisons} are therefore
\[
\begin{array}{c|cc}
 & \Lambda_1 & \Lambda_2\\ \hline
r & 2 & 2\\
\bigl(|V(\Gamma_{\red})|,|E(\Gamma_{\red})|,|F(\Gamma_{\red})|\bigr)
 & (2,2,2) & (2,2,2)\\
g_{\red} & 0 & 0\\
\text{fractional multiplicities}
 & \{\frac12,\frac12\} & \{\frac12,\frac12\}\\
\text{bipartite} & \text{yes} & \text{yes}
\end{array}
\]

To compute the face weights, use the arrows $x_1,x_2,y_1,y_2$ of the common reduced algebra as in \cite[Example 3.5]{Xin26}. The two constructions assign weight $\bar1$ to the arrows in $\{x_2,y_2\}$ and $\{x_2,y_1\}$, respectively, and weight $\bar0$ to the remaining arrows. The two reduced faces have perimeter two, with successive corners $(x_1,y_2)$ and $(y_1,x_2)$. Adding the corner weights gives
\[
\begin{array}{c|cc}
 & \Lambda_1 & \Lambda_2\\ \hline
(\ell(F_1),\weight(F_1)) & (2,\bar1) & (2,\bar0)\\
(\ell(F_2),\weight(F_2)) & (2,\bar1) & (2,\bar0)
\end{array}
\]
where the weights lie in $\ZZ/2\ZZ$. Thus conditions \textup{(1)}, \textup{(2)}, and \textup{(4)} hold, whereas condition \textup{(3)} fails. By \cref{thm:four-condition-completeness}, $\Lambda_1$ and $\Lambda_2$ are not derived equivalent. Thus, even in reduced genus zero, the face perimeters alone do not suffice. Indeed, the values of the Nakayama characters on the corresponding boundary circles must also be compared.
\end{eg}

\bibliographystyle{alpha}
\bibliography{reference}
\Addresses
\end{document}